\documentclass[smallextended]{svjour3}       % onecolumn (second format)
\usepackage[bottom=4cm, right=4cm, left=4cm, top=4cm]{geometry}
\usepackage{latexsym}
\usepackage{amssymb}
\usepackage{amsmath}
\usepackage[mathscr]{eucal}
\usepackage{graphicx}
\usepackage{hyperref}
\usepackage{caption}
\usepackage{subcaption}
\usepackage{setspace}

\renewcommand{\qed}{\hfill{\ \ \rule{2mm}{2mm}} \vspace{0.2in}}

\newcommand{\ind}{1\hspace{-2.3mm}{1}}

\renewcommand{\thefigure}{\arabic{figure}}
\begin{document}

\title{Extremal Spanning Trees of Random Marked Graphs with Independent  Edge Weights}
\titlerunning{Spanning Trees of  Weighted  Random Graphs}

\author{ \textbf{Ghurumuruhan Ganesan}}
\authorrunning{G. Ganesan}
\institute{IISER Bhopal,\\
\email{gganesan82@gmail.com }}

\date{}
\maketitle

%\doublespacing

\begin{abstract}
In this paper, we consider a Bernoulli random graph~\(G\)  on~\(n\) vertices with non-uniform edge probabilities, where each vertex has an independent mark and each edge is equipped with an independent positive weight. The cost of an edge depends on the weight as well the marks of the endvertices and  we estimate the growth  of the maximum and minimum  cost of a spanning tree containing all the vertices.  For edge weights with heavy tails and vertex marks distributed uniformly in the unit square, we obtain a phase transition in terms of the tail decay exponent~\(s:\) If~\(s\) is large, then the maximum cost is essentially determined by the vertex locations and if~\(s\) is small, then the edge weights crucially influence the maximum cost. We derive a similar result for minimum cost spanning trees and  use martingale difference based methods to establish the~\(L^2-\)convergence of the extremal cost, appropriately scaled and centred.

\vspace{0.1in} \noindent \textbf{Key words:} Extremal cost spanning trees; Marked Bernoulli random graphs; Independent edge weights; phase transition.

\vspace{0.1in} \noindent \textbf{AMS 2000 Subject Classification:} Primary: 60D05, 60C05.
\end{abstract}

\bigskip

\renewcommand{\theequation}{\thesection.\arabic{equation}}
\setcounter{equation}{0}
\section{Introduction} \label{intro}
The study of the minimum weight spanning trees of a graph is of great practical importance
and many algorithms have been proposed over the years
for various kinds of graphs. For example, the well-known Kruskal's algorithm~\cite{cormen}
iteratively adds edges to a sequence of increasing subtree of the original graph
until a spanning tree is obtained with the constraint that no cycle is created
in any of the iterations. The spanning tree with minimum weight so obtained
is usually called the \emph{Minimum Spanning Tree} (MST).

Minimum spanning trees (MSTs) of complete graphs with \emph{random} edge weights are important from both theoretical and practical perspectives.
For independent and identically distributed (i.i.d.) edge weights with a common cumulative distribution function (cdf)~\(F(.)\) that varies linearly close to zero, Frieze~\cite{fre} studied convergence weight of the MST of the complete graph~\(K_n\) on~\(n\) vertices. Later~\cite{ald} studied convergence in the mean for the MST weight, when the edge weight distributions follow a power law distribution. Janson~\cite{janson} studied central limit theorems for a scaled and centred version of~\(MST_n\) and more recently Addario-Berry et al~\cite{add} studied bounds on the diameter of the MST. The methods involve a combination of graph evolution via Kruskal's agorithm along with a component analysis of random graphs. For MSTs with nonidentical edge weight distributions, Li and Zhang~\cite{li} use the Tutte polynomial approach~\cite{steele3} to compute expressions for the expected value of~\(MST_n.\)

MSTs of Euclidean random graphs whose nodes are randomly distributed in the unit square and whose edges are assigned weights
related to the Euclidean length, have also been extensively studied. When the weight of an edge equals its Euclidean length raised to a positive power,
we refer to the resulting MSTs as power weighted Euclidean MSTs. One of the main objects of interest in the study of power weighted Euclidean MST is its total weight: How does it scale with the number of nodes and the power weight exponent and what are its convergence properties?
Analytical results for such MSTs  have been studied extensively before (see~\cite{steele}~\cite{steele2}~\cite{kest_lee}~\cite{pen_yuk} and references therein). For example,~\cite{steele}  uses edge counting techniques to obtain variance estimates for the MST weight and~\cite{kest_lee}  use martingale methods to obtain central limit theorems (CLTs) for the MST weight, appropriately scaled and centred. In~\cite{pen_yuk} coupling arguments are used to obtain weak laws for functionals of point processes thereby including the MST as a special case. Recently~\cite{chat}  used percolation theoretic arguments to study convergence rate of the CLTs for Euclidean MSTs.

In this paper, we study maximum \emph{and} minimum cost of spanning trees of graphs obtained by placing random \emph{independent} weights in each edge of~\(K_n\) and define the \emph{cost} of an  edge as a function of both its length and the weight. Such scenarios arise often in the study of wireless networks, where nodes are terminals, edges represent communication links between terminals and the edge weight could either be a  gain (like for e.g. fading) or loss (e.g., packet delay) associated with the link~\cite{goldsmith}.

For edge weights with heavy tails and vertex marks distributed uniformly in the unit square, we obtain a phase transition in terms of the tail decay exponent~\(s.\) We show that if~\(s\) is large, then the maximum cost grows with  the Euclidean distance exponent and if~\(s\) is small, then the edge weights inverse cumulative distribution function determines the maximum cost. We derive a similar result for minimum cost spanning trees and  use martingale difference based methods to establish the~\(L^2-\)convergence of the extremal cost, appropriately scaled and centred.

%One main aspect in the above described situation is the following: A crucial property that facilitates the study of the deviation properties in random Euclidean MSTs without weights, is the fact that the maximum degree of any vertex cannot be more than~\(6,\) purely by geometrical considerations. This allows for the usage of martingale difference methods to estimate variance and related properties (see for example~\cite{ald2}). This may no longer be directly applicable when the edge weights are independent and  we use stochastic domination, martingale based methods and a ``ray segmentation" technique to bound the variance of~\(MST_n.\)

The paper is organized as follows:  In the first subsection of Section~\ref{sec_main_res}, we state our main results regarding the maximum cost spanning trees (MASTs) (Theorems~\ref{thm_max_cst_low},~\ref{thm_max_cst_up} and~\ref{thm_spat_mast}) and then illustrate the bounds using examples: Corollary~\ref{cor_repeat_mast} for the terminals with repeaters problem and Corollary~\ref{cor_example_two} for the spatial MASTs. In the next subsection, we then state our main results regarding the minimum cost spanning trees (MSTs) (Theorems~\ref{thm_min_cst_weak},~\ref{thm_min_cst_strong} and~\ref{thm_spat_mst}). As before, we illustrate the bounds using the terminals with repeaters problem (Corollary~\ref{cor_mst_repeat}) and spatial MSTs (Corollary~\ref{cor_example_two_mst}).

In Sections~\ref{sec_pf_max_low},~\ref{sec_pf_max_up} and~\ref{sec_pf_mast_spat_thm}, we prove Theorems~\ref{thm_max_cst_low},~\ref{thm_max_cst_up} and~\ref{thm_spat_mast}, respectively and then establish Corollaries~\ref{cor_repeat_mast} and~\ref{cor_example_two} in Section~\ref{sec_pf_mast_cor}. Similarly,  in Sections~\ref{sec_pf_min_weak},~\ref{sec_pf_min_strong} and~\ref{sec_pf_thm_spat_mst}, we prove Theorems~\ref{thm_min_cst_weak},~\ref{thm_min_cst_strong} and~\ref{thm_spat_mst}, respectively and then finally derive Corollaries~\ref{cor_mst_repeat} and~\ref{cor_example_two_mst} in Section~\ref{sec_pf_cor_mst}.

%Section~\ref{mst_wt_euc}, we state and prove our main result regarding deviation estimates for the minimum cost of spanning trees of Euclidean random graphs equipped with independent edge weights and in Section~\ref{mst_wt_conv}, we establish~\(L^2-\)convergence of the minimum cost, appropriately scaled and centred.

\renewcommand{\theequation}{\thesection.\arabic{equation}}
\setcounter{equation}{0}
\section{Main Results}\label{sec_main_res}
In this section, we describe our main results regarding the maximum and minimum cost of spanning trees of randomly marked graphs with independent edge weights.

We begin with  problem motivation. Consider~\(n\) terminals labelled~\(a_1,a_2,\ldots,a_n,\) placed at deterministic locations in the unit square on the plane, each capable of forming communication links with other terminals. Due to external factors like shadowing~\cite{goldsmith}, interference etc., a communication link between~\(a_i\) and~\(a_j\) is subject to failure with a certain probability, independent of other links. In addition, each successful link undergoes fading~\cite{goldsmith} (independent of other links) that affects the throughput, i.e., the amount of information that can be sent through the link.

To counter this,  terminals are equipped with  repeaters in order to provide information signal boosting capabilities. Any terminal~\(a_i\) has a repeater with a small probability~\(\epsilon_0\) independent of other terminals and a communication link between~\(a_i\) and~\(a_j\) said to be of \emph{high quality} if at least one of~\(a_i\) or~\(a_j\) has a repeater. High quality links have high throughput and low quality links have throughput close to zero. It is of interest to estimate the maximum possible throughput of a fully connected network using minimum number of links and also determine conditions under which near ``optimal" throughput could be achieved. In what follows, we first generalize the above description and estimate the maximum weight of spanning trees in randomly marked graphs and then, as a direct consequence, derive the results relevant to the terminal throughput problem.

\subsection{\bf Maximum Cost Spanning Trees}
Let~\(K_n\) be the complete graph with vertex set~\(\{1,2,\ldots,n\}\) and let~\(\{Z(h)\}_{h \in K_n}\) be independent Bernoulli random variables indexed by the edge set of~\(K_n\) and having distribution
\begin{equation}\label{x_dist}
\mathbb{P}(Z(h)=1) = p(h) = 1-\mathbb{P}(Z(h)=0),
\end{equation}
where~\(0 < p(h) < 1.\)

If the edge~\(h= (u,v)\) has endvertices~\(u\) and~\(v,\) then we define~\(p(h) = p(u,v)\) to be the edge probability and denote~\(Z(h) = Z(u,v)\) to be the \emph{state} of the edge~\(h.\) Let~\(G \subset K_n\) be the random graph formed by the set of all edges~\(h\) satisfying~\(Z(h) =1.\) If~\(p(h) = p\) for all edges~\(h \in K_n,\) then we say that~\(G\) is~\(p-\)\emph{homogenous} or homogenous with edge probability~\(p.\) Else we refer to~\(G\) as an inhomogenous random graph.

We equip each edge~\(h \in K_n\) with a positive random weight~\(W(h) = W(u,v)\) that is independent of the edge states~\(\{Z(e)\}_{e \in K_n}.\) The random variables~\(\{W(h)\}_{h \in K_n}\) are independent and identically distributed (i.i.d.).  Let~\(\{X_i\}_{1 \leq i \leq n}\) be i.i.d.\ (that are also independent of the edge states and weights) elements belonging to some set~\(\Omega_{mk}.\)  We define~\(X_u\) to be the \emph{random mark} associated with the vertex~\(u.\)

Letting~\(r : S \times S \rightarrow (0,\infty)\) be a deterministic measurable function, we define the \emph{cost} of the edge~\(h = (u,v)\) with endvertices~\(u\) and~\(v\) as
\begin{equation}\label{cost_def}
c(h)  = c(u,v) := r(X_u,X_v) \cdot W(u,v).
\end{equation}
The term~\(r(X_u,X_v)\) is denoted as the \emph{cost factor} and we define the complementary cumulative distribution function (ccdf)~\(F_c^{(ct)}\) of the edge cost~\(c(h)\) as~\[F_c^{(ct)}(x) := \mathbb{P}(c(h) > x) \] for~\(x >0.\)  For~\(y > 1,\)  we also define
\begin{equation}\label{h_c_def}
J_c(y) := \max\left\{x > 0: F^{(ct)}_c(x) \geq \frac{1}{y}\right\},
\end{equation}
to be the inverse edge cost ccdf. Similarly, we let~\(F_c(x)\) and~\(H_c(y)\) denote the edge \emph{weight} ccdf and inverse ccdf, respectively.

A component of~\(G\) is a maximal connected subgraph of~\(G\) and we say that~\(G\) is connected if~\(G\) contains a single component. A connected acyclic subgraph of~\(G\) is called a tree and we say that a tree~\({\cal T}\) is a \emph{spanning} tree of a component~\({\cal C} \subset G\) if~\({\cal T}\) contains all vertices of~\({\cal C}.\) We define the cost of a tree~\({\cal T} \subset G\) to be
\begin{equation}\label{tree_cost}
c({\cal T}) := \sum_{h \in {\cal T}}c(h),
\end{equation}
the sum of the costs of edges of~\({\cal T}.\) Let~\(\chi_n\) denote the maximum cost   of a spanning tree of the largest component of~\(G.\)

%CHEEE!! we have the following result.  Throughout constants do not depend on~\(n.\)

Denoting~\(E_{con}\) to be the event that~\(G\) is connected, we have the following bounds for~\(\chi_n.\) Throughout constants do not depend on~\(n\) and for two sequences~\(\{a_n\}\) and~\(\{b_n\},\) we use the notation~\(a_n = o(b_n)\) to denote that~\(\frac{a_n}{b_n} \longrightarrow 0\) as~\(n \rightarrow \infty.\)
\begin{theorem}\label{thm_max_cst_low} Suppose the following hold:\\
\((i)\) There are constants~\(a_0,b_0> 0,0  < \gamma_0 < \frac{1}{2}\) and~\(0 < p = p(n) < 1\) such that
\begin{equation}\label{p_cond_new}
a_{0} np \leq \sum_{v \in {\cal S}} p(u,v) \leq \sum_{v \neq u} p(u,v) \leq b_{0} np
\end{equation}
for all~\(u\) and all sets~\({\cal S}\) containing at least~\(\gamma_0 n\) vertices.\\
\((ii)\) The edge weight ccdf~\(F_c(x)\) is continuous for all large~\(x\) and the edge cost factor satisfies
\begin{equation}\label{cost_fact_cond}
\delta_{low} \leq \left(\mathbb{E}r(X_1,X_2)\right)^2 \leq \mathbb{E}r^2(X_1,X_2) \leq \delta_{up}
\end{equation}
for some constants~\(\delta_{low},\delta_{up} > 0.\)\\
There is a constant~\(\lambda > 0\) such that if~\(p \geq \frac{\lambda \log{n}}{n},\) then
\begin{equation}\label{dev_bound_mast_low}
\mathbb{P}\left(E_{con} \bigcap \left\{\chi_n \geq \lambda^{-1} nH_c(np) \right\} \right) \geq 1- \frac{\lambda}{n}.
\end{equation}
\end{theorem}
The technical condition~(\ref{p_cond_new}) ensures connectivity of the random graph~\(G\) with high probability, i.e., with probability~\(1-o(1)\) and  the resultant lower bound~(\ref{dev_bound_mast_low}) for the maximum cost of a spanning tree, is obtained in terms of the edge weight inverse ccdf~\(H_c(.),\) under the condition~(\ref{cost_fact_cond}) that the expected cost factor of an edge is bounded from below.

Before describing examples, we also state our next result that complements the lower bound for~\(\chi_n\) obtained in Theorem~\ref{thm_max_cst_low}.
\begin{theorem}\label{thm_max_cst_up} Suppose in addition to conditions~\((i)-(ii)\) in Theorem~\ref{thm_max_cst_low}, the following also hold:\\
\((a)\) There are constants~\(C_0,x_0 > 0\) and~\(s  > 3\) such that the edge weight ccdf~\(F_c(.)\) satisfies
\begin{equation}\label{f_scale}
F_c(ax) \leq \frac{C_0}{a^{s}}  \cdot F_c(x)
\end{equation}
for all~\(a > 1\) and all~\(x  >x_0.\)\\
\((b)\) There are constants~\(c_1,c_2 > 0\) such that the edge cost ccdf~\(F_c^{(ct)}(.)\) satisfies
\begin{equation}\label{dom_cond}
F_c(x) \geq c_1F_c^{(ct)}(c_2x) \;\;\text{ for all } x.
\end{equation}
There is a constant~\(\theta > 0\) such that if~\(p \geq \frac{\theta \log{n}}{n},\) then
\begin{equation}\label{dev_bound_mast_up}
\mathbb{P}\left(\chi_n \leq \theta n H_c(np) \right) \geq 1-  \theta^{-1} \cdot p,
\end{equation}
\begin{equation}\label{dev_bound_exp_up}
\mathbb{E}\chi_n \leq \theta n H_c(np)\;\;\;\text{ and }\;\;\;var(\chi_n)\leq \theta  n^2p H_c^2(np).
\end{equation}
\end{theorem}
The condition~\((a)\) is a scaling condition that determines the tail behaviour of the edge weight ccdf. Below, we show through examples that common distributions like power law and exponential decay satisfy~(\ref{f_scale}).  The bounds~(\ref{dev_bound_mast_up}) and~(\ref{dev_bound_exp_up}) demonstrate that the lower bound for~\(\chi_n\) obtained in Theorem~\ref{thm_max_cst_low} is the best possible, provided the edge weight ``dominates" the cost factor in the sense of~(\ref{dom_cond}).

Theorems~\ref{thm_max_cst_low}-\ref{thm_max_cst_up} describe conditions under which the  maximum cost that is influenced by the tail of the edge weight ccdf via the inverse ccdf. We now apply these bounds to the terminals with repeaters problem discussed at the beginning of the section. \\
\emph{\underline{Example 1} (Terminals with repeaters)}: Assume that the marks~\(\{X_i\}_{1 \leq i \leq n}\) are i.i.d.\ with distribution
\begin{equation}\label{faap}
\mathbb{P}(X_i = 0) = g_n = 1-\mathbb{P}(X_i=1)
\end{equation}
for some deterministic sequence~\(0 < g_n \leq 1.\) Also assume that the cost factor~\(r(X_i,X_j)\) satisfies
\begin{equation}\label{cost_def_obi}
r(X_i,X_j) = \left\{
\begin{array}{ll}
1, & \text{ if } X_iX_j  = 0\\
&\\
h_n, & \text{ otherwise},
\end{array}
\right.
\end{equation}
where~\(h_n \) is a deterministic sequence.

We interpret~\(X_i = 0\) to denote that terminal~\(a_i\) is equipped with a repeater and~\(X_i=1\) otherwise. Similarly, the cost factor~\(r(X_i,X_j)\) models the \emph{attenuation} experienced by the information signal passing through the link between~\(a_i\) and~\(a_j.\) If at least one of the terminals~\(a_i\) or~\(a_j\) is equipped with a repeater, then there is no attenuation~(\(r(X_i,X_j) =1\)) and if neither of the terminals  have a repeater, then the attenuation is severe and essentially close to zero; hence we assume that~\(h_n = o(1).\)  The weight~\(W(h)\) of the edge~\(h = (u,v)\) models the \emph{fading  gain} experienced by the communication link between terminals~\(a_u\) and~\(a_v\) and we assume that the fading is Rayleigh~\cite{goldsmith}  so that the edge weighs are exponentially distributed. Finally,~\(p(u,v)\) is the probability of a successful link between terminals~\(u\) and~\(v,\) determined by other extraneous factors, like shadowing, scattering etc. For more details on these phenomena, we refer to Chapter~\(7\) in~\cite{goldsmith}.

Suppose for now that the above communication network is connected and we choose a \emph{deterministic} spanning tree  with~\(n-1\) links, each experiencing i.i.d.\ fading.  Even if no link undergoes attenuation (i.e., we install a repeater at each terminal),  the law of large numbers implies that the average gain per link is bounded and so, with high  probability, i.e., with probability~\(1-o(1),\) the throughput in the resulting network is~\(O(n),\) where we use the notation~\(a_n = O(b_n)\) to denote that~\(a_n \leq Cb_n\) for some constant~\(C > 0\) and all~\(n\) large.

The following result estimates the \emph{maximum} possible throughput of a general communication network undergoing  Rayleigh fading.
\begin{corollary}\label{cor_repeat_mast} Suppose~\[g_n = \epsilon_0,\;\; h_n = o(1)\;\;\text{ and }\;\;p(u,v) = p \geq \frac{1}{n^{\beta}}\] for all edges~\((u,v)\) and some constants~\(0 < \beta,\epsilon_0 < 1.\) If the edge weights~\(\{W(h)\}_{h \in K_n}\) are i.i.d.\ exponential with unit mean, then there are constants~\(\theta_1,\theta_2 > 0\) such that
\begin{equation}\label{faap_gen}
\mathbb{P}\left(E_{con} \bigcap \{\theta_1 n\log{n} \leq \chi_n \leq \theta_2 n\log{n}\}\right) = 1-o(1),
\end{equation}
and~\[\theta_1 n \log{n} \leq \mathbb{E}\chi_n \leq \theta_2 n \log{n}.\] Moreover~\(\frac{\chi_n}{\mathbb{E}\chi_n} \longrightarrow 1\) in~\(L^2\) as~\(n \rightarrow \infty.\)
\end{corollary}
In words, if a small fraction of the terminals are installed with repeaters, then the maximum throughput attainable from a minimally connected network, is  of the order of~\(n\log{n}\) with high probability. i.e., with probability~\(1-o(1).\)   Moreover, this is the best possible and sharp in the sense that the maximum throughput is also concentrated around its expected value, with high probability. In this setup, we could interpret~\(\log{n}\) as the \emph{throughput gain} obtained due to Rayleigh fading.

Theorems~\ref{thm_max_cst_low}-\ref{thm_max_cst_up} evaluated the maximum cost~\(\chi_n\) under conditions where the edge weight primarily influenced the overall cost. Our final result considers spatial spanning trees and  complements Theorems~\ref{thm_max_cst_low}-\ref{thm_max_cst_up}  by describing sufficient conditions under which~\(\chi_n\) is essentially determined by the edge cost factor, rather than the edge weight.
\subsection*{\em Spatial MASTs}
Let~\(\{X_i\}_{1 \leq i \leq n}\) be i.i.d.\  with a common density~\(f(.)\) in the unit square~\(S = [0,1]^2\) satisfying
\begin{equation}\label{f_eq}
\epsilon_1 \leq f(x) \leq \epsilon_2
\end{equation}
for all~\(x \in S\) and some finite positive constants~\(\epsilon_1,\epsilon_2.\) We define~\(X_u\) to be the \emph{random location} of the vertex~\(u\) and let the Euclidean distance~\(d(X_u,X_v)\) between~\(X_u\) and~\(X_v\) denote the length of the edge~\((u,v).\) We define the cost factor of the edge~\(h = (u,v)\) to be
\begin{equation}\label{cst_def_mast_spat}
r(X_u,X_v) = \frac{1}{d^{\alpha}(X_u,X_v)},
\end{equation}
where~\(\alpha \geq 0\) is a constant.

Continuing with the applications to communication networks,~\(X_v\)  denotes the location of the terminal~\(a_v.\) Communication from~\(a_v\) to a nearby terminal requires low transmission power and so in this case, we interpret the cost factor~\(r(X_u,X_v)\) to be the \emph{savings} in transmission power for the link between~\(a_u\) and~\(a_v.\) As before, the weight~\(W(h)\) of the edge~\(h = (u,v)\) is the fading  gain experienced by the  link between~\(a_u\) and~\(a_v\) and~\(1-p(u,v)\) is the probability of link failure due to external factors.

The maximum cost~\(\chi_n\) is a measure of the overall savings in transmission power and we have the following result regarding the growth of~\(\chi_n.\)
\begin{theorem}\label{thm_spat_mast} Suppose~\(p(u,v) = p =o(1)\) for all edges~\((u,v)\) and the edge weights satisfy
\begin{equation}\label{edge_vt_ax}
\delta_0 \leq \left(\mathbb{E}W(u,v)\right)^2 \leq \mathbb{E}W^2(u,v) \leq \delta_1
\end{equation}
for some constants~\(\delta_0,\delta_1 >0.\)\\
\((a)\) There is a constant~\(\theta > 0\) such that if~\(p \geq \frac{\theta \log{n}}{n},\) then
\begin{equation}\label{dev_bds_mast_spat}
\mathbb{P}\left(E_{con} \bigcap \left\{\chi_n \geq \theta n \cdot (np)^{\alpha/2}\right\}\right) \geq 1- e^{-\theta np} - n \cdot \exp\left(-\frac{\theta}{p}\right).
\end{equation}
\((b)\) If the edge weight has bounded~\(s^{th}\) moments; i.e.,~\(\mathbb{E}W^{s}(u,v) < \infty\) for some~\(s \geq \frac{2}{\alpha}+1,\)
then there is a constant~\(\gamma > 0\) such that if~\(p \geq \frac{\gamma \log{n}}{n},\) then
\begin{equation}\label{dev_bds_mast_spat_up}
\mathbb{E}\chi_n \leq \gamma n \cdot (np)^{\alpha/2}\;\;\text{ and }\;\;\mathbb{P}\left(\chi_n \leq \gamma n \nu_n \cdot (np)^{\alpha/2}\right) \geq 1- \frac{\gamma}{\log{n}} - \gamma p,
\end{equation}
where~\[\nu_n :=
\left\{
\begin{array}{ll}
1, & \text{ if }\;\;0 < \alpha < \frac{2}{3}\\
&\\
\log{n}, & \text{ if }\;\;\frac{2}{3} \leq \alpha < 1.
\end{array}
\right. \] Moreover if~\(\alpha < \frac{2}{3},\) then~\(var(\chi_n) \leq \gamma_0 n^2p \cdot (np)^{\alpha}\) for some constant~\(\gamma_0 > 0.\)
\end{theorem}
For~\(0 < \alpha < 1,\) the above result essentially states that~\(\chi_n\) is mainly determined by the terminal locations, if the edge weights have sufficiently large moments. In the proof of Theorem~\ref{thm_spat_mast}, we also demonstrate that if~\(\alpha \geq 1,\) then the edge cost  has unbounded second moment.

Combining Theorems~\ref{thm_max_cst_low}-\ref{thm_spat_mast}, we have the following  ``phase" transition result.
\begin{corollary}\label{cor_example_two} Suppose the edge probabilities satisfy~\(p(u,v) = p = \frac{1}{n^{\beta}}\) for some~\(0 < \beta < 1\) all edges~\((u,v)\) and the edge weight ccdf~\(F_c\) satisfies
\begin{equation}\label{heav_tail_ccdf}
\frac{A_1}{x^s} \leq F_c(x) \leq \frac{A_2}{x^s}
\end{equation} for all~\(x\) large and some constants~\(A_1,A_2 > 0\) and~\(s > 1.\)\\
\((a)\) If~\( 0 < \alpha < \frac{2}{s+1},\) then
\[\mathbb{P}\left(E_{con} \bigcap \left\{ D_1 n \cdot (np)^{1/s} \leq \chi_n \leq D_2 n \cdot (np)^{1/s}\right\}\right) = 1-o(1)\]
and \[D_1 n \cdot (np)^{1/s} \leq \mathbb{E}\chi_n \leq D_2 n \cdot (np)^{1/s}\]
for some constants~\(D_1,D_2 > 0.\) Moreover,~\(\frac{\chi_n}{\mathbb{E}\chi_n} \longrightarrow 1\) in~\(L^2\) as~\(n \rightarrow \infty.\)\\
\((b)\) If either~\(\frac{2}{s-1} < \alpha < 1\) or the edge weights are i.i.d.\ exponentially distributed with finite mean, then
\[ \mathbb{P}\left(E_{con} \bigcap \left\{ C_1 n \cdot (np)^{\alpha/2} \leq \chi_n \leq C_2 n\nu_n \cdot (np)^{\alpha/2}\right\}\right) = 1-o(1)\]
and \[C_1 n \cdot (np)^{\alpha/2} \leq \chi_n \leq C_2 n \cdot (np)^{\alpha/2}\]
for some constants~\(C_1,C_2 > 0,\) where~\(\nu_n\) is as in the statement of Theorem~\ref{thm_spat_mast}. Also, if~\(\alpha < \frac{2}{3},\) then~\(\frac{\chi_n}{\mathbb{E}\chi_n} \longrightarrow 1\) in~\(L^2\) as~\(n \rightarrow \infty.\)
\end{corollary}
The above result describes how the behaviour of the maximum cost~\(\chi_n\)  varies with~\(\alpha.\) If~\(\alpha\) is very small  then the cost factor is also small and so~\(\chi_n\) is mainly determined by the edge weights, as seen in part~\((a).\) On the other hand, for larger values of~\(\alpha\) that still ensure the cost has bounded second moment, we see that~\(\chi_n\) grows with~\(\alpha,\) as in part~\((b).\)  %Equivalently, we could also  state that if the edge weights have sufficiently large moments, then the effect of edge weights on~\(\chi_n\) is small.

Some interesting future directions:\\
\((D1)\) As mentioned in the paragraph following the statement of Theorem~\ref{thm_spat_mast}, for~\(\alpha \geq 1,\) the edge cost  has unbounded second moment and~\(\chi_n,\) intuitively, should again grow with~\(\alpha.\) How exactly is the growth?\\
\((D2)\) Similarly, the case~\(\frac{2}{s+1} \leq \alpha \leq \frac{2}{s-1}\) is also interesting since in this range of~\(\alpha,\) the maximum cost~\(\chi_n\) might depend on both the cost factor and the weights. It would be nice to estimate the rate of growth of~\(\chi_n\) as a function of~\(\alpha\) and~\(s.\)\\
\((D3)\) Also, does there exist a ``critical"~\(\alpha\) value between~\(\frac{2}{s+1}\) and~\(\frac{2}{s-1},\) below which the edge weight ``dominates" and above which the edge cost factor is the main influencing factor? Or is it a critical subinterval in~\(\left[\frac{2}{s+1}, \frac{2}{s-1}\right]\)?

\subsection{\bf Minimum Spanning Trees}
Let~\(G\) be the random subgraph of the complete graph~\(K_n\) with random vertex marks~\(\{X_u\}_{1 \leq u \leq n},\) edge states~\(\{Z(h)\}_{h \in K_n}\) and positive edge weights~\(\{W(h)\}_{h \in K_n},\) as described above. Let~\(c(h)\)  be the cost of the  edge~\(h = (u,v)\) with endvertices~\(u\) and~\(v\) as defined in~(\ref{cost_def}) and let~\(F^{(ct)}_c(x)\) be the edge cost ccdf as defined prior to~(\ref{h_c_def}). We define \[F^{(ct)}(x) := \mathbb{P}(c(h) \leq x) = 1-F_c^{(ct)}(x),\;\;x >0\] to be the cumulative distribution function (cdf) of the edge cost~\(c(h)\)  and for~\(0 \leq z \leq 1,\)  we also define
\begin{equation}\label{h_def}
J(z) := \max\left\{x > 0: F^{(ct)}(x) \leq z\right\},
\end{equation}
to be the inverse cdf. Similarly, we let~\(F(x)\) and~\(H(z)\)  denote the edge \emph{weight} cdf and inverse cdf, respectively.

Defining the cost~\(c({\cal T})\) of a tree~\({\cal T} \subset G\) as in~(\ref{tree_cost}), we let~\(\tau_n\) denote the \emph{minimum} cost   of a spanning tree of the largest component of~\(G.\) Recalling that~\(E_{con}\) denotes the event that~\(G\) is connected, we have the following  result. As before constants do not depend on~\(n.\)
\begin{theorem}\label{thm_min_cst_weak} Suppose the connectivity condition~(\ref{p_cond_new}) in the statement of Theorem~\ref{thm_max_cst_low} holds with~\(0 < p = p(n) < 1\) and in addition:\\
\((I)\) The edge cost factor satisfies~\(\mathbb{E}r^2(x,X_1) \leq B\) for all~\(x \in \Omega_{mk}\) and the edge weights have bounded~\(s^{th}\) moment for some~\(s \geq 2;\) i.e.,~\(\mathbb{E}W^s(h) \leq B,\) for some constant~\(B > 0.\) \\
\((II)\) There are constants~\(c_1,c_2 > 0\) such that
\begin{equation}\label{sandwich_cond}
F(x) \geq c_1F^{(ct)}(c_2x)  \text{ for all } x.
\end{equation}
For every~\(\gamma > 0,\) there is a constant~\(\lambda > 0\) such that if~\(p \geq \frac{\lambda \log{n}}{n}\) and~\(p=o(1),\) then
\begin{equation}\label{mn_comp_bounds}
\mathbb{P}\left(E_{con} \bigcap \left\{\lambda^{-1} n \zeta_n \leq \tau_n \leq \lambda n \varphi_n \right\}\right) \geq 1- \lambda \cdot p,
\end{equation}
%\[\lambda^{-1} n \zeta_n \leq \mathbb{E}\tau_n \leq \lambda n \varphi_n + \frac{1}{n^{1+\gamma}}\]
and
\begin{equation}\label{mst_var_bonda_ax}
\lambda^{-1} n \zeta_n \leq \mathbb{E}\tau_n \leq \lambda n \varphi_n + \frac{1}{n^{1+\gamma}}\;\;\text{ and }\;\;var(\tau_n) \leq \lambda n^2p \varphi_n^2 + \frac{1}{n^{1+\gamma}},
\end{equation}
where
\begin{equation}\label{var_phi_ax}
\zeta_n  = \zeta_n(\lambda) := H\left(\frac{1}{\lambda np}\right)\;\;\text{ and }\;\;\varphi_n  = \varphi_n(\lambda) := H\left(\frac{\lambda \log{n}}{np}\right).
\end{equation}
\end{theorem}
The minimum cost of a spanning tree depends on the behaviour of the edge weight cdf close to the origin (and hence the inverse cdf).

As before, we now illustrate the bounds in Theorem~\ref{thm_min_cst_weak} for the terminals with repeater example described prior to~(\ref{x_dist}) and the spatial spanning trees problem described in Example~\(2\) of Section~\ref{sec_main_res}.\\
\emph{\underline{Example 2} (Terminals with repeaters)}: Assume that the marks~\(\{X_i\}\) are i.i.d.\ with distribution as in~(\ref{faap}) and the cost factor~\(r(X_i,X_j)\) has distribution~(\ref{cost_def_obi}), for some positive deterministic sequences~\(g_n\) and~\(h_n.\) %= o(1)\) and~\(h_n \rightarrow \infty\).

We recall that the mark~\(X_i\) of the~\(i^{th}\) vertex representing the terminal~\(a_i\) equals~\(1\) if and only if~\(a_i\) does not have a repeater. Also, we recall that the link between the terminals~\(a_i\) and~\(a_j\) is said to be bad only if neither of the terminals~\(a_i\) or~\(a_j\) has a repeater. The edge weights and cost factors have different interpretations here: the weight~\(W(i,j)\) of the edge~\((i,j)\) represents the nominal price to be paid for transmission over the link connecting terminals~\(a_i\) and~\(a_j\) and the cost factor~\(r(X_i,X_j)\)  is the extra \emph{penalty} incurred if~\((i,j)\) is found to be bad.

%This is unlike the  example corresponding to the maximal spanning trees, where each terminal had a repeater with small but constant probability and the cost factor was the attenuation encountered by a transmission signal passing the link. The reason is that, here, we are interested in \emph{minimizing} the overall penalty (cost) as opposed to \emph{maximizing} the overall gain (cost) as described in the analysis following~(\ref{faap}).

Suppose we have to pay a high  penalty for a bad link;  we model this by setting~\(h_n \rightarrow \infty.\) To avoid paying high fines, we would like to install as many repeaters as possible and so we set~\(g_n = o(1).\) Following a similar analysis as in the discussion preceding Corollary~\ref{cor_repeat_mast}, we get that the cost of a fully connected chosen randomly (without considering the link prices) is at least of the order of~\(n,\) with high probability. The following result estimates the gain achieved by selecting a minimum cost network.
\begin{corollary}\label{cor_mst_repeat} Suppose~\(h_n \rightarrow \infty\) and~\(g_n = o(1)\) and~\[\limsup h^2_ng_n < \infty, \;\;p(u,v) = p = \frac{1}{n^{\beta}}\] for all edges~\((u,v)\) and some constant~\(0 < \beta < 1.\) If the edge weights are i.i.d.\ uniform in~\([0,1],\) then there are constants~\(\delta_1,\delta_2 > 0\) such that
\begin{equation}\label{disco}
\mathbb{P}\left(E_{con} \bigcap \left\{\frac{\delta_1}{p} \leq \tau_n \leq \frac{\delta_2 \log{n}}{p} \right\}\right)= 1-o(1)
\end{equation}
and
\begin{equation} \label{trimsa}
\frac{\delta_1}{p} \leq \mathbb{E}\tau_n \leq  \frac{\delta_2 \log{n}}{p}
\end{equation}
for all~\(n\) large. Moreover~\(\frac{\tau_n}{\mathbb{E}\tau_n} \rightarrow 1\) in~\(L^2\) as~\(n \rightarrow \infty.\)
\end{corollary}
In other words, with high probability,  the minimum cost of setting up a fully connected network is of the order of~\(\frac{1}{p},\) modulo logarithmic factors. Moreover the minimum cost is concentrated around its expected value with high probability. Since~\(\frac{1}{p} = n^{\beta}\) is much smaller than order of ~\(n,\) the minimum cost of a randomly chosen network, we could interpret the term~\(np\) as the cost savings due to MSTs.

%EXP THAT SELECTING DET TREE CST IS AT LEAST ORDER OF N WHP!!!

Our next result considers the special case where the edge cost factor is bounded and estimates the MST cost for a homogenous  random graph. For convenience, we recall that~\(E_{con}\) denotes the event that~\(G\) is connected and that~\(F(.)\) and~\(H(.),\) respectively, denote the edge weight cdf and the  inverse edge weight cdf as defined in~(\ref{h_def}). We have the following result.
\begin{theorem}\label{thm_min_cst_strong} Suppose the following hold:\\
\((A)\) There exists~\(p = p(n) \in (0,1)\) and constant~\(B > 0\) such that~\[p(u,v) = p  \text{ for all edges } (u,v) \text{ and } r(x,y) \leq B \text{ for all } x,y.\]
\((B)\) There are constants~\(D,x_0,\theta > 0\) such that~\(F\) is strictly increasing in a neighbourhood of~\(\frac{x_0}{2}\) and
\begin{equation}\label{scale_two_cond}
F(kx) \geq D (\log{k})^{1+\theta} \cdot F(x)
\end{equation}
for all integers~\(k \geq 2\) and~\(0 < x < \frac{x_0}{k}.\)\\
For every~\(\gamma > 0,\) there is a constant~\(\kappa > 0\) such that if~\(np \geq \kappa \log{n},\) then
\begin{equation}\label{mn_up_wt}
\mathbb{P}\left(E_{con} \bigcap \left\{ \tau_n \leq \kappa \nu_{wt} + \frac{\kappa \varphi_n \log{n}}{p}   \right\}\right) \geq 1- \frac{1}{n^{1+\gamma}} - \exp\left(- \frac{ \nu^2_{wt}}{\kappa \mu_{wt}}\right)
\end{equation}
and
\begin{equation}\label{mn_up_wt_exp}
\mathbb{E}\tau_n \leq \kappa \nu_{wt} + \frac{\kappa \varphi_n \log{n}}{p} + \kappa n^2\cdot \exp\left(- \frac{\nu^2_{wt}}{\kappa \mu_{wt}}\right) + \frac{1}{n^{1+\gamma}},
\end{equation}  where
\begin{equation}\label{nu_wt_def}
\nu_{wt} := \sum_{j= \kappa \log{n}/p}^{n-1}H\left(\frac{\kappa}{jp}\right),\;\;\mu_{wt} := \sum_{j=\kappa \log{n}/p}^{n-1}H^2\left(\frac{\kappa \log{n}}{jp}\right)
\end{equation} and~\(\varphi_n = \varphi_n(\kappa)\) is as defined in~(\ref{var_phi_ax}).
\end{theorem}
We highlight the advantage of~(\ref{mn_up_wt}) in the next subsection, where we consider the cost of spatial MSTs. Specifically,  in the proof of Corollary~\ref{cor_example_two_mst} below, we demonstrate that the bounds in Theorem~\ref{thm_min_cst_strong} are stronger than Theorem~\ref{thm_min_cst_weak}. This is expected since~(\ref{mn_up_wt}) and~(\ref{mn_up_wt_exp}) are obtained under the (stronger) condition that the edge cost factors are \emph{absolutely} bounded, whereas Theorem~\ref{thm_min_cst_weak} only requires that the \emph{conditional} expectation is absolutely bounded.

Following the pattern established before,  our next result is a counterpart of Theorem~\ref{thm_spat_mast} and complements Theorems~\ref{thm_min_cst_weak}-\ref{thm_min_cst_strong} by describing conditions under which the overall MST cost is influenced by the edge cost factor rather than the edge weight.

%\mathbb{P}\left(\sum_{j=1}^{n-J_0} R_j \geq 2c\nu_{wt}\right) \leq 2\exp\left(-\frac{D\nu_{wt}^2}{\mu_{wt}}\right) + \frac{1}{n^{1+\gamma}} +  n^2 \cdot \exp\left(-\frac{np}{32}\right).

\subsection*{\em Spatial MSTs}
As described prior to Theorem~\ref{thm_spat_mast}, let~\(\{X_i\}_{1 \leq i \leq n}\) be i.i.d.\  with a common density~\(f(.)\) in the unit square~\(S = [0,1]^2\) satisfying~(\ref{f_eq}) for all~\(x \in S\) and some finite positive constants~\(\epsilon_1,\epsilon_2.\) Similar to~(\ref{cst_def_mast_spat}), we define the cost factor of the edge~\(h = (u,v)\) to be
\begin{equation}\label{cst_def_mst_spat}
r(X_u,X_v) = \left(\frac{d(X_u,X_v)}{\sqrt{2}}\right)^{\alpha},
\end{equation}
where~\(\alpha \geq 0\) is a constant. Since the Euclidean distance between any two vertices is at most~\(\sqrt{2},\) we see that the edge cost factor is at most~\(1.\)

As in the case of  maximum cost spanning trees,~\(X_v\)  denotes the location of the terminal~\(a_v\)  and communication from~\(a_v\) to a nearby terminal requires low transmission power. Thus we interpret the cost factor~\(r(X_u,X_v)\) to be the nominal power budget for the link between~\(a_u\) and~\(a_v.\) The weight~\(W(h)\) of the edge~\(h = (u,v)\) is the extra cost (or penalty) involved due to fading \emph{attenuation} and as before~\(1-p(u,v)\) is the probability of link failure due to external factors, like shadowing.

Recalling that~\(E_{con}\) denotes the event that the random graph~\(G\) is connected, we have the following result regarding the minimum cost~\(\tau_n\) needed to setup a fully connected communication network.
\begin{theorem}\label{thm_spat_mst} Suppose the edge weights are~\(\leq 1\) a.s.\ and the edge weight cdf~\(F\) satisfies
\begin{equation}\label{tail_cnd_mst}
\sum_{k \geq 1} k^{2/\alpha} F\left(\frac{1}{k}\right) < \infty.
\end{equation}
There are constants~\(\lambda_1,\lambda_2  >0\) such that
\begin{equation}\label{mst_dev_bds_spat}
\mathbb{P}\left(E_{con} \bigcap \left\{\frac{\lambda_1n}{(np)^{\alpha/2}} \leq \tau_n \leq \lambda_2n \left(\frac{\log{n}}{np}\right)^{\alpha/2}\right\}\right) \geq 1-\lambda_1 \cdot p
\end{equation}
and
\begin{equation}\label{mst_exp_bds_spat}
\frac{\lambda_1 n}{(np)^{\alpha/2}} \leq \mathbb{E}\tau_n \leq \lambda_2 n \left(\frac{\log{n}}{np}\right)^{\alpha/2}.
\end{equation}
\end{theorem}
The above result states that if the edge weight (or penalty) cdf decays sufficiently fast, close to the origin, then minimum cost~\(\tau_n\) essentially depends on the terminal locations. This is a worst case scenario, since~(\ref{tail_cnd_mst}) is trivially true if we impose constant penalty on each edge. In the case  when penalties are small with high probability, we expect that the minimum cost is less as well.

Combining Theorems~\ref{thm_min_cst_weak}-\ref{thm_spat_mst} and recalling that~\(E_{con}\) denotes the event that the random graph~\(G\) is connected, we have the following result.
\begin{corollary}\label{cor_example_two_mst} Let~\(\alpha > 0\) be as in~(\ref{cst_def_mst_spat}) and suppose the edge weight cdf~\(F\) satisfies~\(F(x) = x^{1/\delta}\) for all~\(0 < x< 1\) and some~\(0 < \delta < 1.\) Also suppose that the edge probability~\(p(u,v) = p = \frac{1}{n^{\beta}}\) for all edges~\((u, v)\) and some constant~\(0 < \beta < 1.\)\\
\((a)\) If~\( \alpha < \frac{2\delta}{1+\delta},\) then
\begin{equation}\label{skilp_fa}
\mathbb{P}\left(E_{con} \bigcap \left\{\frac{D_1n}{(np)^{\delta}} \leq \tau_n \leq \frac{D_2n}{(np)^{\delta}}\right\}\right) = 1-o(1)
\end{equation} and
\begin{equation}\label{skilp_ga}
\frac{D_1n}{(np)^{\delta}} \leq \mathbb{E}\tau_n \leq \frac{D_2n}{(np)^{\delta}},
\end{equation} for some constants~\(D_1,D_2 > 0.\) Moreover,~\(\frac{\tau_n}{\mathbb{E}\tau_n} \rightarrow 1\) in~\(L^2\) as~\(n \rightarrow \infty.\)\\
\((b)\) If~\( \alpha > \frac{2\delta}{1-\delta},\) then
\[\mathbb{P}\left(E_{con} \bigcap \left\{ \frac{\theta_1n}{(np)^{\alpha/2}} \leq \tau_n \leq \theta_2 n\left(\frac{\log{n}}{np}\right)^{\alpha/2} \right\}\right) = 1-o(1)\] and \[\frac{\theta_1n}{(np)^{\alpha/2}} \leq \mathbb{E}\tau_n \leq \theta_2 n\left(\frac{\log{n}}{np}\right)^{\alpha/2},\]  for some constants~\(\theta_1,\theta_2 > 0.\) Moreover, if~\(\beta > \frac{\alpha-2\delta}{1+\alpha-2\delta},\) then~\(\frac{\tau_n}{\mathbb{E}\tau_n} \rightarrow 1\) in~\(L^2\) as~\(n \rightarrow \infty.\)
\end{corollary}
Again, there is a phase transition in the behaviour of the MST cost. For small~\(\alpha,\) the cost~\(\tau_n\) depends essentially on the edge weights and beyond a certain threshold value,~\(\tau_n\) is influenced primarily by the vertex locations.

Does there exist a  critical~\(\alpha\) value for the above phase transition? If so, what is the value and what is the behaviour of~\(\tau_n\) at the critical value? These are interesting questions to explore.

%\begin{corollary}\label{thm_example} Suppose GIVE QUASI HOM EXAMPLE HERE!!! the condition~(\ref{p_cond_max_ax}) holds for~\(\frac{(\log{n})^2}{n} \leq p = o(1)\) and some constants~\(a_0,b_0 > 0, 0 < \gamma_0 < \frac{1}{2}.\)\\
%\((a)\) Suppose the vertex marks~\(\{X_i\}_{1 \leq i \leq n}\) are i.i.d.\  uniform  in~\([0,1]\and the edge weights~\(\{W(f)\}_{f \in K_n}\) are i.i.d.\ exponential with unity mean. There are constants~\(\gamma_1,\gamma_2 > 0\) such that
%\begin{equation}\label{dev_bound_mast_tot_one}
%\mathbb{P}\left(E_{con} \bigcap \left\{ \gamma_1 n\log(np) \leq \chi_n \leq \gamma_2 n \log(np) \right\} \right) = 1- o(1).
%\end{equation}
%\((b)\) Suppose the vertex marks~\(\{X_i\}_{1 \leq i \leq n}\) are i.i.d.\ exponential with unity mean and the edge weights~\(\{W(f)\}_{f \in K_n}\) satisfy a power law: \[\mathbb{P}\left(W(f)  > x\right) = \frac{1}{(s-1)x^s} \text{ x > 1} WRT HERE!!1.\]
%There are constants~\(\lambda_1,\lambda_2 > 0\) such that
%\begin{equation}\label{dev_bound_mast_tot_two}
%\mathbb{P}\left(E_{con} \bigcap \left\{ \lambda_1 n \cdot (np)^{1/s} \leq \chi_n \leq \lambda_2 n \cdot (np)^{1/s} \right\} \right) = 1- o(1).
%\end{equation}
%\end{corollary}

%EXPL HERE!!!! THAT LC VS WT DOM!!!

\renewcommand{\theequation}{\thesection.\arabic{equation}}
\setcounter{equation}{0}
\section{Preliminaries}\label{sec_prelim}
Throughout, we use the following deviation estimates regarding sums of independent  random variables.
\begin{lemma}\label{lemmax}
\((a)\) Let~\(\{W_j\}_{1 \leq j \leq r}\) be independent Bernoulli random variables with~\[\mathbb{P}(W_j = 1) = 1-\mathbb{P}(W_j = 0) > 0.\] Setting~\(S_r := \sum_{j=1}^{r} W_j,\) we have for~\(0 < \epsilon \leq \frac{1}{2}\) that
\begin{equation}\label{conc_est_f}
\mathbb{P}\left(\left|S_r - \mathbb{E}S_r\right| \geq \epsilon \mathbb{E}S_r\right) \leq 2\exp\left(-\frac{\epsilon^2}{4}\mathbb{E}S_r\right).
\end{equation}
\((b)\) Let~\(\{U_j\}_{1 \leq j \leq r}\) be positive independent random variables satisfying~\(0 \leq U_j \leq 1\) and set~\(V_r := \sum_{j=1}^{r}\lambda_jU_j\) where~\(\lambda_j >0\) are positive numbers. For any~\(\epsilon > 0\) we have that
\begin{equation}\label{conc_mom_f}
\mathbb{P}\left(\left|V_r - \mathbb{E}V_r \right| \geq \epsilon \mathbb{E}V_r \right) \leq 2\exp\left(-\frac{\epsilon^2(\mathbb{E}V_r)^2}{\sum_{j=1}^{r}\lambda_j^2}\right).
\end{equation}
\end{lemma}
For a proof of~(\ref{conc_est_f}) and~(\ref{conc_mom_f}), we refer to Appendix~\(A\) of Alon and Spencer (2008).

The following Lemma collects relevant properties of the edge cost distribution used in our proofs of the main Theorems.
\begin{lemma}\label{lemma_hz} The following properties hold:\\
\((a)\) If there exists~\(z_0 > 0\) such that~\(J_c(z)\) is strictly increasing for all~\(z > z_0,\) then for any~\(z > z_0\) we have
\begin{equation}\label{inverse_cdf_equality}
\mathbb{P}\left(c(f) > J_c(z)\right) = \frac{1}{z}.
\end{equation}
\((b)\) If there exists~\(x_0 > 0\) such that~\(F_c^{(ct)}(x_0) > 0\) and~\(F_c^{(ct)}(x)\) is continuous for all~\(x \geq x_0,\) then~\(J_c(z)\) is strictly increasing for all~\(z > \frac{1}{F_c^{(ct)}(x_0)}.\)\\
\((c)\) If the cost~\(c(h)\) of the edge~\(h\) satisfies the scaling relation~(\ref{f_scale}) for some~\(s > 2,\) then there is a constant~\(D > 0\) not depending on~\(h\) such that~\(\mathbb{E}c^2(h) \leq D.\)
\end{lemma}
Parts~\((a)\) and~\((b)\) describe sufficient conditions under which ccdf~\(F_c^{(ct)}(.)\) and the inverse ccdf~\(J_c(.)\) satisfy the inverse property and part~\((c)\) states that if the scaling condition  holds for~\(s > 2,\) then the edge weights have bounded second moments.

\emph{Proof of Lemma~\ref{lemma_hz}}: To prove~\((a),\) we let~\(z > z_0\) be arbitrary and  use the definition of~\(J_c(z)\) in~(\ref{h_c_def}) and the right continuity of the ccdf to get that~\[F_c^{(ct)}(J_c(z)) = \mathbb{P}\left(c(f) > J_c(z)\right) \leq \frac{1}{z}.\] For the converse direction, we use the fact that~\(J_c(z)\) is strictly increasing for all~\(z > z_0,\) as mentioned in Lemma statement. This necessarily implies that  for any~\(z > z_0\) and  any~\(\varepsilon > 0,\)  we must have~\(F_c^{(ct)}(J_c(z)) \geq \frac{1-\varepsilon}{z};\)  else we arrive at the contradictory relation~\[J_c(z) = J_c\left(\frac{z}{1-\varepsilon}\right).\] Combining the above, we get~(\ref{inverse_cdf_equality}) and this completes the proof of part~\((a)\) of the Lemma.

We prove part~\((b)\) by contradiction as follows.  Suppose there exists~\(u > z > z_0 := \frac{1}{F_c^{(ct)}(x_0)}\) such that~\(J_c(z) = J_c(u).\) Both~\(J_c(z)\) and~\(J_c(u)\) are necessarily at least~\(x_0\) and the right continuity of~\(F_c^{(ct)}\) further implies that~\[ F_c^{(ct)}(J_c(z)) = F_c^{(ct)}(J_c(u)) \leq \frac{1}{u}.\] But since~\(\frac{1}{u} < \frac{1}{z}\) strictly, this implies that~\( F_c^{(ct)}(J_c(z)) < \frac{1}{z}\) strictly and so  invoking the stronger continuity condition of the ccdf, we get that~\( F_c^{(ct)}((1-\eta)J_c(z)) < \frac{1}{z}\) strictly, for all small~\(\eta> 0.\) This contradicts the definition of inverse ccdf in~(\ref{h_c_def}) and so~\(J_c(z)\) is strictly increasing for all~\(z > z_0.\) This completes the proof of part~\((b)\) of the Lemma.

For the final part, we recall the constant~\(x_0\) in the statement of~(\ref{f_scale}) and write
\begin{equation}\label{w_split}
\mathbb{E}c^2(h)  = 2\int_{0}^{\infty}y \mathbb{P}\left(c(h) > y\right) dy = I_1 + I_2
\end{equation}
where
\begin{equation}\label{i_one_est}
I_1 := 2\int_{0}^{2x_0}y \mathbb{P}\left(c(h) > y\right) dy \leq 2\int_{0}^{2x_0} y = 4x_0^2
\end{equation}
and
\begin{equation}
I_2 :=  2\int_{2x_0}^{\infty}y \mathbb{P}\left(c(h) > y\right) dy = 2x_0^2\int_{2}^{\infty} a\mathbb{P}\left(c(h) > ax_0\right) da, \label{i_two_est_ax}
\end{equation}
by a change of variable~\(y = ax_0.\)

From~(\ref{f_scale}), we get for~\(a > 1\) that
\begin{equation}\nonumber
\mathbb{P}\left(c(h) \geq ax_0\right) \leq \frac{C_0}{a^s} \mathbb{P}(c(h) \geq x_0) \leq \frac{C_0}{a^s}
\end{equation}
and substituting this into~(\ref{i_two_est_ax}) and using the fact that~\(s > 2\) strictly, we then get that
\begin{equation}\label{i_two_est}
I_2 \leq 2x_0^2 \int_{2}^{\infty} \frac{C_0}{a^{s-1}} da =  \frac{C_0x_0}{2^{s-3}} < \infty.
\end{equation}
Combining~(\ref{i_two_est}) with~(\ref{i_one_est}), we get then get from~(\ref{w_split}) that the edge weights have bounded second moments. This completes the proof of the Lemma.~\(\qed\)

Let~\({\cal B} \subset \{1,2,\ldots,n\}\) be any deterministic set containing~\(b \geq 0\) vertices, where~\(b\) is a constant and let~\(G({\cal B}) \subset G\) be the subgraph of~\(G,\) obtained after removing the vertices of~\({\cal B}.\)  If~\(d_{{\cal B}}(v)\) and~\(d(v)\) respectively denote the degree of vertex~\(v\) in~\(G({\cal B})\) and~\(G,\) then clearly~\(d_{{\cal B}}(v) \leq d(v).\)  The following Lemma collects vertex degree, connectivity and edge weight properties of~\(G({\cal B}),\) used in our proof of the main Theorems.  Define
\begin{equation}\label{p_low_up}
p_{low} := \min_{1 \leq v \leq n}\frac{1}{n-1}\sum_{u \neq v} p(u,v)\;\;\text{ and }\;\; p_{up} := \max_{1 \leq v \leq n} \frac{1}{n-1}\sum_{u \neq v} p(u,v),
\end{equation}
to be the minimum and maximum possible values of the average edge probability per vertex, that possibly depend on~\(n.\)
\begin{lemma}\label{lemma_deg}
\((a)\) Suppose there are constants~\(A_0,B_0 > 0\) and~\(0 < p = p(n)  <1\) such that
\begin{equation}\label{p_cond_max_ax}
A_0p \leq p_{low}  \leq p_{up}  \leq B_0 p.
\end{equation}  If~\(p \geq \frac{M\log{n}}{n}\) for a large enough constant~\(M,\) then
\begin{equation}\label{e_deg_est_max}
\mathbb{P}\left(E_{deg}({\cal B})\right) \geq  1-\exp\left(-Dnp\right)
\end{equation}
for all~\(n\) large and some constant~\(D > 0,\) where
\begin{equation}\label{e_deg_def_ax2}
E_{deg}({\cal B})  := \bigcap_{v=1}^{n} \left\{ \frac{3A_0np}{4} \leq d_{{\cal B}}(v) \leq d(v) \leq 2B_0np \right\}.
\end{equation}
\((b)\) If the condition~(\ref{p_cond_new}) in Theorem~\ref{thm_max_cst_low} holds, then~(\ref{p_cond_max_ax}) also holds with~\(A_0 = a_0\) and~\(B_0 = 2b_0.\) Moreover if~\(p \geq \frac{M\log{n}}{n}\) for a large enough constant~\(M,\) then
\begin{equation}\label{e_con_est_max}
\mathbb{P}\left(E_{con}({\cal B})\right) \geq  1-\exp\left(-\frac{a_0np}{4}\right)
\end{equation}
for all~\(n\) large, where~\(a_0 >0\) is the constant in~(\ref{p_cond_new}) and~\(E_{con}({\cal B})\) is be the event that~\(G({\cal B})\) is connected.
\end{lemma}

\emph{Proof of Lemma~\ref{lemma_deg}}:  From~(\ref{p_low_up}), we know that  the sum of the edge probabilities of the vertex~\(v\) satisfies~\[ A_0 (n-1)p \leq (n-1)p_{low} \leq \sum_{u \neq v} p(u,v) \leq (n-1)p_{up} \leq B_0 (n-1)p,\] by Theorem statement. This implies that
\[A_0(n-1)p - b \leq \mathbb{E}d_{{\cal B}}(v) \leq \mathbb{E}d(v) \leq B_0(n-1)p\] and since~\(p \geq \frac{M\log{n}}{n}\) by Theorem statement, we have that
\[A_0(n-1)p - b \geq \frac{7A_0np}{8}\] for all~\(n\) large. Therefore applying the deviation estimate~(\ref{conc_est_f})  we get that
\begin{equation}\nonumber
\mathbb{P}\left(\frac{3A_0np}{4} \leq d_{{\cal B}}(v) \leq d(v) \leq 2B_0np\right) \geq 1-2\exp\left(-2Dnp\right)
\end{equation}
for some constant~\(D > 0,\) not depending on the choice of~\(v\) or the constant~\(M\) in Theorem statement.  Recalling the event~\(E_{deg}(.)\) defined in~(\ref{e_deg_def_ax2}), we then get by an application of the union bound that
\begin{equation} \label{trista}
\mathbb{P}\left(E_{deg}({\cal B})\right) \geq  1-2n\exp\left(-2D np\right).
\end{equation}
Since~\(p \geq \frac{M\log{n}}{n}\) we choose the constant~\(M > 0\) large enough so that~(\ref{e_deg_est_max}) holds for all~\(n\) large.
This completes the proof  of  part~\((a)\) of the Lemma.

Clearly if~(\ref{p_cond_new}) holds, then for any vertex~\(u,\) we have that
\[a_0p \leq \frac{a_0np}{n-1} \leq p_{low} \leq p_{up} \leq \frac{b_0np}{n-1} \leq 2b_0p,\] for all~\(n\) large. Thus~(\ref{p_cond_max_ax}) holds with~\(A_0 = a_0\) and~\(B_0 = 2b_0.\)

To prove the connectivity estimate~(\ref{e_con_est_max}), we use an analogous argument as in the proof of Theorem~\(7.3,\) pp.~\(164-165,\) in~\cite{boll}. For a deterministic set~\({\cal S} \subset \{1,2,\ldots,n\} \setminus {\cal B},\) of vertices, let~\(E_{cross}({\cal S})\) be the event that no edge having one endvertex in~\({\cal S}\) and the other endvertex in~\({\cal S}^c,\) is present in~\(G.\) If~\(G({\cal B})\) is disconnected, then there is necessarily a component~\({\cal C}\) in~\(G_{rem}({\cal B})\) with vertex set~\({\cal V},\) satisfying:\\
\((a)\)~\({\cal V}\) has~\(r \leq \frac{n-b}{2}\) vertices\\
\((b)\)~\(E_{cross}({\cal V})\) occurs.\\
In other words, recalling that~\(E_{con}({\cal B})\) denotes the event that~\(G({\cal B})\) is connected, we have that
\begin{equation}\label{gb_con}
E^c_{con}({\cal B}) \subseteq \bigcup_{{\cal S}}E_{cross}({\cal S}),
\end{equation}
where the union is over all deterministic sets~\({\cal S} \subset\{1,2,\ldots,n\} \setminus {\cal B},\) containing at most~\(\frac{n-b}{2}\) vertices.

For a given~\({\cal S}\) containing~\(r\) vertices, we define~\({\cal S}^c := \{1,2,\ldots,n\} \setminus {\cal S}\) and deduce that the event~\(E_{cross}({\cal S})\) occurs with probability
\begin{align}
\mathbb{P}\left(E_{cross}({\cal S})\right) &= \prod_{u \in {\cal S}} \prod_{v \in {\cal S}^c \setminus {\cal B}} \left(1-p(u,v)\right) \nonumber\\
&\leq \exp\left(-\sum_{u \in {\cal S}} \sum_{v \in {\cal S}^c \setminus {\cal B}} p(u,v)\right). \label{e_cross_est}
\end{align}
We know~\({\cal S}^c \setminus {\cal B}\) contains at least~\(\frac{n-b}{2} \geq \gamma_0n\) vertices for all~\(n\) large, where~\(0 < \gamma_0 <\frac{1}{2}\) is as in the condition~(\ref{p_cond_new}). Therefore
\[\sum_{v \in {\cal S}^c}p(u,v) \geq a_0np\] where~\(a_0 > 0\) is the constant in~(\ref{p_cond_new}) and
\begin{align}
\mathbb{P}\left(E_{cross}({\cal S})\right) &\leq \exp\left(- \sum_{u \in {\cal S}} a_0 np\right) \nonumber\\
&= \exp\left(-a_0npr \right), \label{e_cross_est2}
\end{align}
since~\({\cal S}\) has~\(r\) vertices.

The number of choices for~\({\cal S}\) is~\[ {n-b \choose r} \leq {n \choose r} \leq n^{r}\] and so the relation~(\ref{e_cross_est2}) together with the union bound implies that
\begin{equation}\label{e_con_gen_tits}
\mathbb{P}\left(E_{con}^c({\cal B})\right) \leq \sum_{r=1}^{(n-b)/2} n^{r} \cdot \exp\left(-a_0npr \right).
\end{equation}
We now set~\(p \geq \frac{M\log{n}}{n}\) and choose the constant~\(M\) larger if necessary, so that
\[n \cdot \exp\left(-a_0np \right) \leq \exp\left(-\frac{a_0np}{2}\right).\] With this choice of~\(M,\) we get from~(\ref{e_con_gen_tits}) that
\begin{align}
\mathbb{P}\left(E_{con}^c({\cal B})\right) &\leq \sum_{r=1}^{n/2} \exp\left(-\frac{a_0npr}{2}\right) \nonumber\\
&\leq \sum_{r \geq 1} \exp\left(-\frac{a_0npr}{2}\right) \nonumber\\
&= \frac{\exp\left(-a_0np/2\right)}{1-\exp\left(-a_0np/2\right)} \nonumber.
\end{align}
for all~\(n\) large. Since~\(np \geq M\log{n},\) this obtains~(\ref{e_con_est_max}) and therefore completes the proof of the Lemma.~\(\qed\)

\renewcommand{\theequation}{\thesection.\arabic{equation}}
\setcounter{equation}{0}
\section{Proof of Theorem~\ref{thm_max_cst_low}}\label{sec_pf_max_low}
We first assume the following weaker conditions to obtain a ``quasi" lower deviation estimate for~\(\chi_n.\) Specifically, suppose:\\
\((p1)\) The condition~(\ref{p_cond_max_ax}) described in Lemma~\ref{lemma_deg}\((a)\) holds for some constants~\(A_0,B_0 > 0.\)\\
\((p2)\) The edge weight ccdf~\(F^{(wt)}_c(x)\) is continuous for all large~\(x.\)\\
We show below that  there are constants~\(M,\lambda > 0\) such that if~\(p \geq \frac{M \log{n}}{n},\) then
\begin{equation}\label{dev_bound_low_mast_weak}
\mathbb{P}\left(E_{con} \bigcap \left\{\chi_n \geq  \lambda n \cdot \mu_m H_c(np) \right\} \right) \geq 1- e^{-\lambda np} -\frac{\sigma_m^2}{\lambda n\mu_m^2} - \mathbb{P}(E^c_{con}),
\end{equation}
where~\(\mu_m := \mathbb{E}r(X_1,X_2)\) and~\(\sigma_m^2 := var(r(X_1,X_2)).\)

For a set of vertices~\({\cal B},\) we recall the event~\(E_{deg}({\cal B})\) defined in~(\ref{e_deg_def_ax2}) and assume henceforth that~\(E_{deg} := E_{deg}(\emptyset)\) occurs. The first step in our proof is to estimate  the number of edges having ``sufficiently large" weight. Formally, say that an edge~\(f\) of the complete graph~\(K_n\) is \emph{heavy} if its weight~\(W(f) > H_c(np),\) where~\(H_c(.)\) is the inverse edge weight ccdf as described in the statement following~(\ref{h_c_def}). Let~\(G_{heavy}  \subset G\) be the subgraph of~\(G\) formed by heavy edges. We estimate the number~\(N_{heavy}\) of edges in~\(G_{heavy}\) as follows. If~\(E_{deg}\) occurs, then each vertex in~\(G\) has degree at least~\(\frac{3A_0np}{4}\) and by a standard handshaking argument, we know that the sum of degrees of vertices in any graph is equal to twice the number of edges. Thus the number of edges in~\(G\) is at least~\[\frac{1}{2} \cdot n \cdot \frac{3A_0np}{4} = \frac{3A_0n^2p}{8}.\]

Since~\(F^{(wt)}_c(x)\) is continuous for all large~\(x,\) we invoke parts~\((a)\) and~\((b)\) in Lemma~\ref{lemma_hz} to get that
\begin{equation}\label{heavy_edge}
\mathbb{P}\left(f \text{ is heavy}\right) = \frac{1}{np}.
\end{equation}
Thus each edge in~\(G\) is independently heavy with probability at least~\(\frac{1}{np}\)  and so~\(N_{heavy}\)  is stochastically dominated from below by a Binomial random variable with parameters~\(\frac{3n^2A_0p}{8}\) and~\(\frac{1}{np}.\) Consequently, the deviation estimate~(\ref{conc_est_f}) implies that
\[\mathbb{P}\left(N_{heavy} \geq \frac{3A_0n}{16} \,\middle\vert\, E_{deg} \right) \geq 1- \exp\left(-C n\right)\]
for some constant~\(C > 0.\) Combining this with the probability estimate~(\ref{e_deg_est_max}) for~\(E_{deg},\) we get
\begin{align}\label{n_heavy_est}
\mathbb{P}\left(\left\{N_{heavy} \geq \frac{3A_0n}{16}\right\} \bigcap E_{deg}  \right) &\geq \left(1-e^{-Cn}\right) \mathbb{P}(E_{deg}) \nonumber\\
&\geq \left(1-e^{-Cn}\right) \left(1-e^{-Dnp}\right) \nonumber\\
&\geq 1-2e^{-Dnp}
\end{align}
for all~\(n\) large, since~\(np \geq M\log{n}\) by Theorem statement.

For future use, we also obtain an upper bound for the \emph{sum} of vertex degrees in~\(G_{heavy}.\) Indeed, let~\({\cal V} \subset \{1,2,\ldots,n\}\) be any fixed set of~\(V\) vertices and let~\(d_{heavy}(u)\) be the degree of vertex~\(u\) in~\(G_{heavy}.\) A standard hand shaking argument implies that
\begin{equation}\label{grand_prix_one}
\sum_{u \in {\cal V}} d_{heavy}(u) = 2N_{heavy}({\cal V}) + N_{heavy}({\cal V}, {\cal V}^c),
\end{equation}
where~\(N_{heavy}({\cal V})\) is the number of heavy edges containing both endvertices in~\({\cal V}\) and~\(N_{heavy}({\cal V}, {\cal V}^c)\) is the number of heavy edges containing one endvertex in~\({\cal V}\) and the other endvertex in~\({\cal V}^c.\) The random variables~\(N_{heavy}({\cal V})\) and~\(N_{heavy}({\cal V}, {\cal V}^c)\) are independent and so we get from~(\ref{grand_prix_one}) that
\begin{equation}\label{grand_prix_two}
\mathbb{E}\exp\left(\sum_{u \in {\cal V}} d_{heavy}(u)\right) = \mathbb{E}\exp\left(2N_{heavy}({\cal V})\right)\mathbb{E}\left(N_{heavy}({\cal V}, {\cal V}^c)\right).
\end{equation}

To evaluate the right hand expression of~(\ref{grand_prix_two}), we introduce a couple of notations. Recalling from~(\ref{x_dist}) that~\(Z(u,v)\) is the state of the edge~\((u,v)\) in~\(G,\) we let \[Z_{heavy}(u,v) := Z(u,v) \ind(W(u,v) > J_c(np))\] denote the state of~\((u,v)\) in~\(G_{heavy},\) where~\(\ind(.)\) refers to the indicator function. From~(\ref{x_dist}) we also see that~\(Z(u,v)=1\) with probability~\(p(u,v)\) and so we get from~(\ref{heavy_edge}) that~\((u,v)\) is present in~\(G_{heavy}\) with probability~\[q(u,v) := p(u,v) \cdot \frac{1}{np}.\] In other words,
\[\mathbb{P}\left(Z_{heavy}(u,v) = 1\right) = q(u,v) = 1-\mathbb{P}\left(Z_{heavy}(u,v) = 0\right)\]
and with the above notations, we also get that
\[N_{heavy}({\cal V}, {\cal V}^c)  = \sum_{u \in {\cal V}} \sum_{v \in {\cal V}^c} Z_{heavy}(u,v)\text{ and } N_{heavy}({\cal V}, {\cal V}) = \sum_{(u,v)} Z_{heavy}(u,v),\]
where the final summation is over all edges of the complete graph~\(K_n,\) having both endvertices in~\({\cal V}.\)

Thus
\begin{align}
\mathbb{E}\exp\left(2N_{heavy}({\cal V})\right) &= \mathbb{E}\exp\left(2\sum_{(u,v)}Z_{heavy}(u,v)\right) \nonumber\\
&= \prod_{(u,v)} \mathbb{E}\exp\left(2Z_{heavy}(u,v)\right) \nonumber\\
&= \prod_{(u,v)}  \left(1- q(u,v) + e^2 q(u,v)\right) \nonumber\\
&\leq \prod_{(u,v)} \exp\left((e^2-1)q(u,v)\right) \nonumber\\
&= \exp\left((e^2-1) \sum_{(u,v)} q(u,v)\right) \nonumber\\
&= \exp\left(\frac{e^2-1}{np} \sum_{(u,v)} p(u,v)\right) \nonumber\\
&\leq \exp\left(\frac{e^2-1}{np} \sum_{u \in {\cal V}} \sum_{v \in {\cal V} \setminus \{u\}} p(u,v)\right). \label{n_heavy_vv_est}
\end{align}
Arguing similarly for~\(N_{heavy}({\cal V}, {\cal V}^c)\) we get an analogous estimate but with~\(e^2-1\) replaced by~\(e-1;\) i.e.,
\begin{equation}\label{n_heavy_vv2_est}
\mathbb{E}\exp\left(N_{heavy}({\cal V}, {\cal V}^c)\right) \leq \exp\left(\frac{e-1}{np} \sum_{u \in {\cal V}} \sum_{v \in {\cal V}^c}q(u,v)\right).
\end{equation}

We have that~\(\frac{e-1}{np} \leq \frac{e^2-1}{np}\) and so combining the generating function estimates~(\ref{n_heavy_vv2_est}) with~(\ref{n_heavy_vv_est}) and recalling that~\({\cal V}\) contains~\(V\) vertices, we see that
\begin{align}
\mathbb{E}\exp\left(2N_{heavy}({\cal V}, {\cal V})\right)\mathbb{E}\exp\left(N_{heavy}({\cal V}, {\cal V}^c)\right) &\leq \exp\left(\frac{e^2-1}{np} \sum_{ u \in {\cal V}} \sum_{v \neq u} p(u,v)\right) \nonumber\\
&\leq \exp\left(\frac{e^2-1}{np} \sum_{u \in {\cal V}}(n-1)p_{up} \right) \nonumber\\
&= \exp\left(\frac{e^2-1}{np} (n-1)p_{up} V\right) \nonumber\\
&\leq \exp\left(\frac{e^2-1}{p}p_{up}V\right) \nonumber\\
&\leq \exp\left( (e^2-1)B_0V\right), \label{e_n_heavy_est}
\end{align}
where the second  inequality in~(\ref{e_n_heavy_est}) follows from the definition of~\(p_{up}\) in~(\ref{p_low_up}) and the final estimate in~(\ref{e_n_heavy_est}) is a consequence of the fact that~\(p_{up} \leq B_0 p\) for some constant~\(B_0 > 0,\) by Theorem statement.

Plugging~(\ref{e_n_heavy_est}) into~(\ref{grand_prix_two})  we have that
\begin{equation}
\mathbb{E}\exp\left(\sum_{u \in {\cal V}} d_{heavy}(u) \right) \leq  \exp\left((e^2-1)B_0V\right)\nonumber
\end{equation}
and using the  standard Chernoff bound we get for~\(x \geq 0\) that
\begin{equation}\nonumber
\mathbb{P}\left(\sum_{u \in {\cal V}} d_{heavy}(u) \geq x \right) \leq e^{-x} \exp\left((e^2-1)B_0V\right).
\end{equation}
Setting~\(x = 10\zeta B_0V,\) where~\(\zeta \geq 1\) is a constant to be determined later, we obtain
\begin{align}
\mathbb{P}\left(\sum_{u \in {\cal V}} d_{heavy}(u) \geq 10\zeta B_0V \right) &\leq e^{-10\zeta B_0V} \exp\left((e^2-1)B_0V\right) \nonumber\\
&\leq e^{-10\zeta B_0V} e^{8B_0V} \nonumber\\
&\leq e^{-10\zeta B_0V} e^{8\zeta B_0V} \nonumber\\
&= e^{-2\zeta B_0 V} \label{grand_prix_5}
\end{align}
where the second relation in~(\ref{grand_prix_5}) is true since~\(e^2-1 < 8\) and the third estimate in~(\ref{grand_prix_5}) follows since~\(\zeta \geq 1\) by choice.

Finally, defining
\[E_{heavy} := \bigcap_{\cal V} \left\{\sum_{u \in {\cal V}} d_{heavy}(u) \leq 2\zeta B_0V\right\}\]
where the intersection is over all sets~\({\cal V}\) containing~\(V\) vertices, we get
\begin{equation}\label{talsqax}
\mathbb{P}\left(E^c_{heavy}\right) \leq {n \choose V} \cdot e^{-2\zeta B_0V} \leq \left(\frac{ne^{1-2\zeta B_0}}{V}\right)^{V},
\end{equation}
where the first inequality in~(\ref{talsqax}) follows from the union bound and the final estimate in~(\ref{talsqax}) is true since~\({a \choose b} \leq \left(\frac{ae}{b}\right)^{b}.\)  Letting~\(0 < \psi < \frac{1}{2}\) be a small constant to be determined later, we now set
\begin{equation}\label{d_choice_max}
V = \psi n\;\;\;\;\text{ and }\;\;\;\;\zeta  =\zeta(\psi) := \frac{1}{8B_0}\log\left(\frac{e}{\psi^2}\right),
\end{equation}
so that~\[\frac{n e^{1-2\zeta B_0}}{V} = \frac{ne\psi^2}{\psi n e} = \psi .\] With these choices, we get from~(\ref{talsqax}) that
\begin{equation}\label{e_heavy_appa_est}
\mathbb{P}\left(E^c_{heavy}\right) \leq \psi^{V} = \exp\left(- D_0 n\right),
\end{equation}
where~\(D_0 = D_0(\psi) := \psi\log\left(\frac{1}{\psi}\right) >0\) is a constant.

Recalling that~\(E_{con}\) denotes the event that~\(G\) is connected and setting
\[E_{net} :=  E_{con} \bigcap E_{heavy} \bigcap \left\{N_{heavy} \geq \frac{3A_0n}{16}\right\} \bigcap E_{deg},\] we apply the union bound and get from the respective probability estimates~(\ref{e_heavy_appa_est}) and~(\ref{n_heavy_est}) that
\begin{align}\label{e_net_est_max}
\mathbb{P}(E_{net}) &\geq 1-\mathbb{P}(E^c_{con}) - 3e^{-D np} - \exp\left(-D_0 n\right) \nonumber\\
&\geq 1-4e^{-D np}-  \mathbb{P}(E^c_{con}),
\end{align}
for all~\(n\) large, where we recall that~\(D > 0\) is the constant in~(\ref{n_heavy_est}) and the final estimate in~(\ref{e_net_est_max}) is true since~\(p \geq \frac{M\log{n}}{n},\) by Theorem statement.

Assuming~\(E_{net}\) occurs, we now estimate the maximum weight~\(\chi_n\) of a spanning tree of~\(G\) as follows. Let~\({\cal S}_{heavy}\) be the set of all heavy edges in~\(G.\) Since~\(E_{tot}\) occurs, the graph~\(G\) is connected and there are at least~\(\frac{3A_0n}{16}\)  edges in~\({\cal S}_{heavy}.\) We  use the occurrence of the event~\(E_{heavy}\) to iteratively extract a matching (i.e., a vertex disjoint set of edges) of size comparable to~\(n\) from~\({\cal S}_{heavy}\) as described below.

Set~\({\cal S}_1 := {\cal S}_{heavy}\) and pick a heavy edge~\(h_1 \in {\cal S}_1\) with endvertices~\(c_1\) and~\(d_1.\) Define~\({\cal M}_1 := \{h_1\}\) be the matching obtained at the end of the first iteration. Letting~\({\cal E}_1\) be the set of all heavy edges containing either~\(c_1\) or~\(d_1\) as an endvertex, we then set~\[{\cal S}_2 := {\cal S}_1 \setminus {\cal E}_1.\]   Repeating the above procedure, we pick a heavy edge~\(h_2 \in {\cal S}_2\) and set~\[{\cal M}_2 := \{h_1,h_2\}\] be the matching obtained at the end of the second iteration. As before, throw away all edges from~\({\cal S}_2\) that share an endvertex with~\(h_2\) and call the resulting set as~\({\cal S}_3.\) Continue this process until we reach a step~\(L\) such that~\({\cal S}_{L+1} = \emptyset.\)

By construction, the set~\({\cal M}_L\) of edges forms a matching of size~\(L.\) To estimate~\(L,\) we use the fact that~\(E_{heavy}\) occurs, where~\(0 < \psi < \frac{1}{2}\) is a constant. Indeed, there are~\(L\) edges in~\({\cal M}_L\) and so the total number of endvertices of the edges in~\({\cal M}_L\) is~\(2L.\) Since~\(E_{heavy}\) occurs, the sum of degrees of any~\(\psi n\) vertices in the random graph~\(G_{heavy}\) is at most~\[10\zeta \psi n = n \cdot \frac{5\psi}{4B_0}\log\left(\frac{e}{\psi^2}\right),\] by our choice of~\(\zeta\) in~(\ref{d_choice_max}).

Since~\(\psi \log{\psi} \rightarrow 0\) as~\(\psi \rightarrow 0,\) we choose~\(\psi\) small enough so that~\(10\zeta\psi < \frac{3A_0}{17}.\) With this choice of~\(\psi,\) we see that after~\(\frac{\psi n}{2}\) steps of the iteration process above, we have removed at most~\(\frac{3A_0n}{17}\) edges from~\({\cal S}_{heavy}.\) But since~\({\cal S}_{heavy}\) contains at least~\(\frac{3A_0n}{16}\) edges to begin with, we see that~\(L \geq \frac{\psi n}{2}\) and so the matching~\({\cal M}_L\) contains at least~\(\frac{\psi n}{2}\) edges.

%By choice, each edge in~\(G_{heavy}\) has Euclidean length at least~\(\varepsilon\) and weight at least~\(J_c(np)\) and so the total \emph{cost} of heavy edges in~\({\cal M}_L\) is at least
%\begin{equation}\label{heavy_weight}
%\frac{\psi n}{2} \cdot \varepsilon^{\alpha} \cdot J_c(np) = \theta_0\cdot nJ_c(np),
%\end{equation}
%for some constant~\(\theta_0 > 0.\)

Using the fact that~\(G\) is connected, we now iteratively connect the heavy edges in~\({\cal M}_L\) together to obtain a spanning tree of~\(G,\) as follows. Indeed, let~\({\cal M}_L := \{h_1,\ldots,h_L\}\)  and set~\(\Gamma_1 := \{h_1\}.\) For~\(i \geq 2,\) let~\({\cal Q}_{i-1} \subset {\cal M}_L\) be the set of all edges of~\({\cal M}_L\) present in~\(\Gamma_{i-1}\) and assume that~\(\Gamma_{i-1}\) satisfies the following properties:\\
\((q1)\)~\(\Gamma_{i-1} \subset G\) is a tree, \\
\((q2)\) there are exactly~\(i-1\) edges in~\({\cal Q}_{i-1}\) and\\
\((q3)\) no remaining edge of~\({\cal R}_{i-1} := {\cal M}_{L} \setminus {\cal Q}_{i-1}\) shares an endvertex with~\(\Gamma_{i-1}.\)\\
The graph~\(\Gamma_1\) satisfies~\((q1)-(q3).\)

\begin{figure}[tbp]
\centering
%\fbox{
\includegraphics[width=3.5in, trim= 0 550 150 50, clip=true]{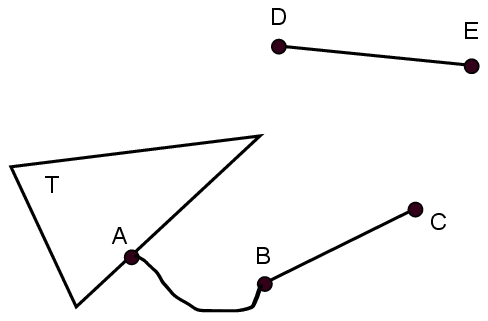}
%}
\caption{The tree~\(\Gamma_i\) is the union of the tree~\(\Gamma_{i-1} = T,\) the path~\({\cal P}_{i-1} = AB\) and the edge~\(e_i = BC.\)}
\label{fig_merge}
\end{figure}

Since~\(G\) is connected, there is a path~\({\cal P}_i\) containing at least one edge, from some vertex of~\(\Gamma_{i-1}\) to an endvertex of an edge~\(e_i \in {\cal R}_{i-1}\) that contains no other endvertex of~\({\cal R}_{i-1}.\) The graph
\[\Gamma_i := \Gamma_{i-1} \bigcup {\cal P}_{i} \bigcup \{e_i\}\] is a tree of~\(G,\)  contains~\(i\)  edges from~\({\cal M}_L\) and does not contain an endvertex of any edge from~\({\cal R}_{i-1} \setminus \{e_i\}.\) Thus~\(\Gamma_i\) satisfies properties~\((q1)-(q3)\) above and this completes the induction step. This is illustrated in Figure~\ref{fig_merge} where~\(\Gamma_{i-1}\) is represented by the triangle~\(T\) and~\({\cal R}_{i-1}\) consists of the two edges~\(BC\) and~\(DE.\)  The edge~\(e_{i} = BC\) is connected to some vertex~\(A\) in~\(\Gamma_{i-1}\) by the path~\({\cal P}_i\) represented by the wavy line, that contains no other endvertex of an edge in~\({\cal R}_{i-1}.\) Proceeding iteratively, we obtain a tree~\(\Gamma_{L} \subset G\) that contains all the~\(L\) edges of~\({\cal M}_{L}.\) Further adding more edges to~\(\Gamma_L\) if necessary, we then obtain a spanning tree~\({\cal T}_{fin}\) of~\(G\) that contains~\(\Gamma_L\) as a subgraph.

To estimate the total cost of the edges in~\({\cal T}_{fin},\) we use the fact that the marks of the endvertices of the~\(L \geq \frac{\psi n}{2}\) heavy edges~\(\{(c_i,d_i)\}_{1 \leq i \leq L}\) in the matching~\({\cal M}_L\) obtained above, are \emph{independent}. Therefore the corresponding cost factors~\(\{r(X_{c_i},X_{d_i})\}_{1 \leq i \leq L}\) defined in~(\ref{cost_def}), are i.i.d. Because~\(E_{net}\) occurs, we have that~\(L \geq \frac{\psi n}{2}\) and so defining
\[U_{fin} := \sum_{i=1}^{L} r(X_{c_i},X_{d_i}),\] we apply the Chebychev inequality to get for~\(\epsilon > 0\) that
\begin{equation}\label{azuma}
\mathbb{P}\left(U_{fin} \leq (1-\epsilon)\mathbb{E}(U_{fin}\mid E_{net}) \mid E_{net} \right) \leq \frac{D\sigma_{m}^2}{n \mu_{m}^2,}
\end{equation}
for some constant~\(D > 0,\) where~\( \mu_m = \mathbb{E}r(X_1,X_2)\) and~\(\sigma_m^2 = var(r(X_1,X_2))\) respectively denote the mean and variance of~\(r(X_1,X_2),\) as mentioned in the Theorem statement.

Again using the fact that~\(E_{net}\) occurs, we have that~\(L \geq \frac{\psi n}{2}\) and so~\[\mathbb{E}(U_{fin}\mid E_{net}) = L \cdot \mu_{m} \geq \frac{\psi n}{2} \cdot \mu_{m}.\] Therefore choosing~\(\epsilon  =\frac{1}{2}\) for example,   we get from~(\ref{azuma}) that
\[\mathbb{P}\left(U_{fin} \geq c_1 n \mu_{m} \,\middle\vert\, E_{net}\right) \geq 1-\frac{c_2\sigma_m^2}{n\mu_m^2},\] for some constants~\(c_1,c_2 > 0.\)  Combining with the estimate~(\ref{e_net_est_max}) for the event~\(E_{net},\) we then get that
\begin{align}
\mathbb{P}\left(\left\{U_{fin} \geq c_1 n \mu_{tot} \right\} \bigcap E_{net}\right) &\geq \left(1-\frac{c_2\sigma_m^2}{n\mu_m^2}\right) \left(1-4e^{-D np}-  \mathbb{P}(E^c_{con})\right) \nonumber\\
&\geq 1-4e^{-D np} - \frac{c_2\sigma_m^2}{n\mu_m^2} - \mathbb{P}(E^c_{con}),\label{e_net_est_max2}
\end{align}
for all~\(n\) large, where we recall from~(\ref{e_net_est_max}) that~\(D > 0\) is a constant.

By definition every heavy edge has weight at least~\(H_c(np)\) and if~\[\left\{U_{fin} \geq c_1 n \mu_{m}\right\} \bigcap E_{net}\] occurs, then the total cost of the edges in the matching~\({\cal M}_L\) is at least \[U_{fin} \cdot J_c(np) \geq  c_1 n\mu_m \cdot H_c(np).\] This  in turn implies that the cost of the spanning tree~\({\cal T}_{fin}\) is at least~\(c_1 n \mu_m H_c(np)\) and so the maximum cost~\(\chi_n\) of a spanning tree of~\(G\) satisfies \begin{equation}\label{chill_ammal}
\chi_n \geq c_1 n \mu_m H_c(np).
\end{equation} The probability estimate~(\ref{e_net_est_max2}) obtains the desired deviation lower bound~(\ref{dev_bound_low_mast_weak})  for~\(\chi_n.\)

We now invoke the connectivity estimate proved in Lemma~\ref{lemma_deg}\((b)\) under the stronger condition~(\ref{p_cond_new}) and the fact that the cost factors have bounded first and second moments in the sense of~(\ref{cost_fact_cond}), to obtain that if~\(p \geq \frac{M\log{n}}{n},\) then
\begin{equation}\label{dev_bound_low_mast_axxxx}
\mathbb{P}\left(E_{con} \bigcap \left\{\chi_n \geq \lambda_1 n H_c(np) \right\} \right) \geq 1- \frac{1}{\lambda_1 n},
\end{equation}
for some constant~\(\lambda_1 > 0.\) Since~\(np \rightarrow \infty\) by Theorem statement, this obtains the desired lower bound for~\(\chi_n\) in the estimate~(\ref{dev_bound_mast_low}). This completes the proof of part~\((a)\) of the Theorem.~\(\qed\)

\renewcommand{\theequation}{\thesection.\arabic{equation}}
\setcounter{equation}{0}
\section{Proof of Theorem~\ref{thm_max_cst_up}}\label{sec_pf_max_up}
We begin by obtaining  generic upper bounds for the maximum cost~\(\chi_n\) in terms of the edge cost inverse ccdf~\(J_c(.)\) defined in~(\ref{h_c_def}).
\begin{lemma}\label{lemma_max_cst_up} Suppose the condition~(\ref{p_cond_new}) in the statement of Theorem~\ref{thm_max_cst_low} holds. Also suppose there are constants~\(C_0,x_0 > 0\) and~\(s  > 3\) such that the edge \emph{cost} ccdf~\(F_c^{(ct)}\) satisfies the scaling relation~(\ref{f_scale}) for all~\(a > 1\) and all~\(x  >x_0.\) There are constants~\(M,\gamma >0\) such that if~\(p \geq \frac{M\log{n}}{n},\) then
\begin{equation}\label{exp_bound_up_mast}
\mathbb{E}\chi_n \leq \gamma n  \cdot J_c(np)\;\;\;\text{ and }\;\;\; var(\chi_n) \leq  \gamma n^2p   \cdot J_c^2(np).
\end{equation}
\end{lemma}
We remark that the above result is general and applies even for cost functions different from the structure described in~(\ref{cost_def}). %In our proof  We adapt the proof strategy for the case

%Part~\((a)\) of the above Theorem implies that  with high probability (i.e., with probability~\(1-o(1)\)), the maximum cost of a spanning tree is at least of the order of~\(nJ_c(np).\) Here and henceforth, for two sequences~\(\{a_n\}\) and~\(\{b_n\},\) we use the notation~\(a_n = o(b_n)\) to denote that~\(\frac{a_n}{b_n} \rightarrow 0\) as~\(n \rightarrow \infty.\) %In Theorem~\ref{cor_mast} below, we describe a ``quasi"homogenous random graph that satisfies~(\ref{p_cond_max_ax}) and is connected with high probability.

%The term~\(J_c(np)\) in~(\ref{dev_bound_low_mast}) could be interpreted as the gain obtained in choosing a \emph{maximum} weight matching as opposed to simply selecting a matching based on a deterministic rule. From the upper bound in~(\ref{dev_bound_up_mast}), we see that this factor is essentially the best possible since the maximum cost is at most of the order of~\(nJ_c(np)\) with high probability.

\emph{Proof of Lemma~\ref{lemma_max_cst_up}}: Let~\(G\) be the marked random graph as defined in~(\ref{x_dist}) and let~\({\cal T}_n\) be the maximum cost spanning tree of the largest component in~\(G\) with cost~\(\chi_n.\) We begin by obtaining an upper bound for the expected value of~\(\chi_n.\) We recall  the event~\(E_{deg} = E_{deg}(\emptyset)\) defined in Lemma~\ref{lemma_deg} that ensures that  each vertex in~\(G\) has degree at most~\(2B_0np,\) where~\(B_0 > 0\) is the constant in~(\ref{p_cond_max_ax}). If~\(E_{deg}\) occurs, then using the handshaking relation that the sum of vertex degrees is twice the number of edges in any graph, we see that~\(G\) has at most~\(B_0n^2p\) edges.

Assuming~\(E_{deg}\) occurs, we now use a segmentation approach to estimate the maximum cost~\(\chi_n.\) For integer~\(j \geq 0,\) say that an edge~\(h = (u,v)\) of the complete graph~\(K_n\) is~\(j-\)bad if its cost~\(c(h)\) satisfies \[c(h) \in [2jJ_c(np), 2(j+1)J_c(np)),\] where~\(J_c(.)\) is the edge cost ccdf defined in~(\ref{h_c_def}). The tree~\({\cal T}_n\) has at most~\(n-1\) edges and so the total cost of all~\(0-\)bad edges in~\({\cal T}_n\) is at most \[2(n-1)J_c(np) \leq 2nJ_c(np).\] Similarly, if~\(N_{bad}(j)\) is the total number of~\(j-\)bad edges in~\(G,\) then the total cost of all~\(j-\)bad edges in~\(G\) is at most~\(2(j+1)J_c(np)N_{bad}(j).\) Therefore,  the total cost~\(\chi_n\) of~\({\cal T}_n\) is upper bounded as
\begin{equation}\label{priscille_tits22_ax}
\chi_n \leq 2nJ_c(np) + \sum_{j \geq 1} 2(j+1)J_c(np)N_{bad}(j).
\end{equation}

To estimate~\(N_{bad}(j),\) we use the scaling relation~(\ref{f_scale}) and get for any edge~\(h\) that
\begin{align}\label{priscille_tits2}
\mathbb{P}\left(c(h) \geq 2jJ_c(np)\right) &\leq \frac{C_0}{j^{s}} \cdot \mathbb{P}\left(c(h) \geq 2J_c(np)\right) \nonumber\\
&\leq \frac{C_0}{j^{s}} \cdot \frac{1}{np},
\end{align}
where~\(C_0 > 0\) is the constant in~(\ref{f_scale}) and the final estimate in~(\ref{priscille_tits2}) follows from the definition of the inverse ccdf in~(\ref{h_c_def}).  From~(\ref{priscille_tits2}), we get that
\begin{equation}\label{njk_est}
\mathbb{E}\left(N_{bad}(j) \mid E_{deg}\right) \leq B_0 n^2p \cdot \frac{C_0}{j^s} \cdot \frac{1}{np} = \frac{\beta_0 n }{j^{s}},
\end{equation}
where~\(\beta_0 > 0\) is a constant. Plugging this into~(\ref{priscille_tits22_ax}), we get that
\begin{align}\nonumber
\mathbb{E}\left(\chi_n \mid E_{deg}\right) &\leq 2nJ_c(np) + 2nJ_c(np) \sum_{j \geq 1} \frac{(j+1)}{j^s} \nonumber\\
&= \beta_1 n J_c(np),
\end{align}
for some \emph{finite} constant~\(\beta_1 > 0,\) since both~\(\delta\) and~\(s\) are strictly larger than~\(2,\) by Theorem statement. Thus
\begin{equation}\label{e_chi_n_one}
\mathbb{E}\chi_n \ind(E_{deg}) \leq \mathbb{E}\left(\chi_n \mid E_{deg}\right) \leq \beta_1 nJ_c(np).
\end{equation}

If the complement event~\(E_{deg}^c\) occurs, then we use the direct upper bound
\[\chi_n \leq \sum_{f \in K_n} c(f),\] the total cost of all edges in the complete graph~\(K_n,\) to get that
\begin{align}\label{ewn_etot_comp}
\mathbb{E}\chi_n \ind(E^c_{deg}) &\leq \sum_{ f \in K_n} \mathbb{E}c(f) \ind(E^c_{deg}) \nonumber\\
&= \sum_{f \in K_n} \mathbb{E}c(f) \mathbb{P}(E^c_{deg}),
\end{align}
since the event~\(E_{deg}\) does not depend on the edge weights or the cost factors. Further using the fact that the edge cost~\(c(f)\) has bounded moment (see Lemma~\ref{lemma_hz}\((c)\))  and recalling the estimate~(\ref{e_deg_est_max}) for the event~\(E_{deg},\) we then get from~(\ref{ewn_etot_comp}) that
\begin{equation}\label{e_chi_n_two}
\mathbb{E}\chi_n \ind(E_{deg}^c) \leq D_1 n^2 e^{-D np},
\end{equation}
where we recall that the constant~\(D > 0\) in~(\ref{e_deg_est_max}) does not depend on the choice of the constant~\(M\) in Theorem statement.

Combining~(\ref{e_chi_n_two}) and~(\ref{e_chi_n_one}) we then get that
\begin{equation}\label{ewn_pen_fin}
\mathbb{E}\chi_n \leq \beta_1 nJ_c(np) + D_1 n^2e^{-Dnp}
\end{equation}
and we choose~\(np \geq M\log{n}\) for a large enough constant~\(M\) so that~\(n^2e^{-Dnp} \leq 1.\) Further, using the fact that the inverse ccdf~\(J_c(z)\) is increasing in~\(z,\) we get that
\[J_c(np) \geq J_c(2) > 0\] and so we get that~\(\mathbb{E}\chi_n \leq 2\beta_1 nJ_c(np)\) for all~\(n\) large. This obtains the desired expectation bound for the maximum cost in Theorem statement.

In the remaining part of the proof, we use the martingale difference method based on the Efron-Stein inequality to obtain the variance bound for~\(\chi_n.\) We begin by recalling that~\(\chi_n\) is the maximum cost of a spanning tree of the largest component of the random graph~\(G\) with vertex locations~\(\{X_j\}_{1 \leq j \leq n}\) and edge states and weights~\(\{(Z(f_k),W(f_k))\}_{1 \leq k \leq m}\) where~\(m = {n \choose 2}\) and~\(\{f_k\}_{1 \leq k \leq m}\) is a deterministic ordering of the edges of the complete graph~\(K_n.\)

For~\(1 \leq j \leq n,\) let~\(G_{mod}(j)\) be the random graph obtained when the mark~\(X_j\) of the vertex~\(j\) is replaced by an independent copy~\(X_j^{(c)},\) that is also independent of all random variables defined so far. Also let~\(\chi_{mod}(j)\) be the maximum cost of a spanning tree of the largest component of~\(G_{mod}(j).\) Similarly, for~\(1 \leq k \leq m,\) we let~\(G_{mod}(f_k)\) be the random graph obtained when the edge state and weight~\((Z(f_{k}),W(f_{k}))\) of the edge~\(f_{k}\) is replaced by an independent copy~\((Z^{(c)}(f_{k}),W^{(c)}(f_{k}))\) that is also independent of all random variables defined so far. As before let~\(\chi_{mod}(f_k)\) be the maximum cost of a spanning tree of the largest component of~\(G_{mod}(f_k).\)

With the notations presented in the above paragraph, we get from the Efron-Stein inequality (see Section~\(2,\) Eq.~\((2.1)\) of~\cite{steele}) that
\begin{equation}\label{var_bound_ax_max}
var(\chi_n) \leq \sum_{j=1}^{n} \mathbb{E}\left(\chi_n - \chi_{mod}(j)\right)^2 + \sum_{k=1}^{m}\mathbb{E}\left(\chi_n - \chi_{mod}(f_k)\right)^2
\end{equation}
For future use, we upper bound~(\ref{var_bound_ax_max}) in the following way: For~\(1 \leq j \leq n,\) let~\[G_{rem}(j) := G(\{j\}) \subset G\] be the random graph obtained after removing vertex~\(j\) from~\(G\) and let~\({\cal W}_{rem}(j)\) be the maximum cost spanning tree of the largest component of~\(G_{rem}(j)\) with corresponding cost~\(\chi_{rem}(j).\)  From the triangle inequality, we have for any~\(1 \leq j \leq n\) that
\[|\chi_n-\chi_{mod}(j)| \leq |\chi_n-\chi_{rem}(j)| + |\chi_{rem}(j) - \chi_{mod}(j)|\] and so squaring and taking expectations and using~\((a+b)^2 \leq 2(a^2+b^2),\) we obtain
\begin{align}
\mathbb{E}\left(\chi_n-\chi_{mod}(j)\right)^2 &\leq 2\mathbb{E}\left(\chi_n-\chi_{rem}(j)\right)^2 + 2\mathbb{E}\left(\chi_{mod}(j)-\chi_{rem}(j)\right)^2  \nonumber\\
&= 4\mathbb{E}\left(\chi_n-\chi_{rem}(j)\right)^2, \label{sekshi}
\end{align}
since~\(\chi_{mod}(j)\) has the same distribution as~\(\chi_n.\)

Similarly, let~\(G_{rem}(f_k) \subset G\) be the random graph obtained after removing edge~\(f_k\) from~\(G\) and let~\({\cal W}_{rem}(f_k)\) be the maximum cost spanning tree of the largest component of~\(G_{rem}(f_k)\) with corresponding cost~\(\chi_{rem}(f_k).\) Arguing as in~(\ref{sekshi}), we get that
\begin{equation}
\mathbb{E}\left(\chi_n-\chi_{mod}(f_k)\right)^2 \leq  4\mathbb{E}\left(\chi_n-\chi_{rem}(f_k)\right)^2. \label{sekshi2}
\end{equation}
Plugging~(\ref{sekshi2}) and~(\ref{sekshi}) into~(\ref{var_bound_ax_max}), we obtain
\begin{align}
var(\chi_n) &\leq 4\sum_{j=1}^{n}\mathbb{E}\left(\chi_n-\chi_{rem}(j)\right)^2 +4\sum_{k=1}^{m}\mathbb{E}\left(\chi_n - \chi_{rem}(f_k)\right)^2\nonumber\\
&= 4nI_{loc} + 4mI_{wt}, \label{i_loc_wt_est_max}
\end{align}
where
\[I_{loc} := \mathbb{E}\left(\chi_n-\chi_{rem}(1)\right)^2 \text{ and } I_{wt} := \mathbb{E}\left(\chi_n-\chi_{rem}(f_1)\right)^2\]  denote the scaled contributions due to randomness in vertex locations and edge states/weights, respectively.

In what follows, we estimate~\(I_{wt}\) and~\(I_{loc}\) in that order below. Because the condition~(\ref{p_cond_new}) holds, we know that the condition~(\ref{p_cond_max_ax}) in Lemma~\ref{lemma_deg} holds as well with appropriate constants~\(A_0\) and~\(B_0.\)\\
\emph{\underline{Step 1} (Estimating~\(I_{wt}\))}: Let~\(f_1 = (u_1,v_1)\) have~\(u_1\) and~\(v_1\) as endvertices. From the discussion prior to Lemma~\ref{lemma_deg}, we recall that~\(G(\{u_1,v_1\})\) is the graph obtained by removing the \emph{vertices}~\(u_1\) and~\(v_1\) from~\(G.\) Also recalling that~\(E_{con}(\{u_1,v_1\})\) is the event that~\(G(\{u_1,v_1\})\) is connected (see Lemma~\ref{lemma_deg}), we get from~(\ref{e_con_est_max}) that
\begin{equation}\label{madai_ax}
\mathbb{P}(E_{con}(\{u_1,v_1\})) \geq 1-\exp\left(-Dnp\right)
\end{equation}
for some constant~\(D > 0,\) not depending on the choice of~\(\{u_1,v_1\}.\) We also recall the event~\(E_{deg} := E_{deg}(\emptyset)\) defined  in~(\ref{e_deg_est_max}) that ensures that the degree of each vertex in~\(G\) is at least of the order of~\(np.\) Defining
\[E_{nice} :=  E_{con}\left(\emptyset\right) \bigcap E_{con}(\{u_1,v_1\})\] and choosing the constant~\(D > 0\) in~(\ref{madai_ax}) smaller if necessary, we invoke the union bound and get from the corresponding probability estimates~(\ref{e_con_est_max}) and~(\ref{madai_ax}) that
\begin{equation}\label{e_nice_est}
\mathbb{P}(E_{nice}) \geq 1-2\exp\left(-Dnp\right).
\end{equation}

We now split~\(I_{wt}\) as
\begin{equation}\label{i_wt_split_ax}
I_{wt} = I_{wt,1} + I_{wt,2},
\end{equation}
where
\[I_{wt,1} := \mathbb{E}\left(\chi_n-\chi_{rem}(f_1)\right)^2\ind(E_{nice})\] and
\[I_{wt,2} := \mathbb{E}\left(\chi_n-\chi_{rem}(f_1)\right)^2\ind(E^c_{nice}).\] In what follows, we estimate~\(I_{wt,2}\) and~\(I_{wt,1}\) in that order below.

If~\(E_{nice}^c\) occurs, then we use the direct bound~\[\chi_n \leq  \sum_{f\in K_n} c(f),\] the sum of costs of all edges in the complete graph~\(K_n.\) The same estimate holds for~\(\chi_{rem}(f_1)\) as well and so we get
\[|\chi_n - \chi_{rem}(f_1)| \leq \sum_{f \in K_n} c(f).\] Consequently
\begin{align}
I_{wt,2} &= \mathbb{E}\left(\chi_n - \chi_{rem}(f_1)\right)^2 \ind(E_{nice}^c) \nonumber\\
&\leq \mathbb{E}\left(\sum_{f \in K_n} c(f)\right)^2 \ind(E_{nice}^c) \nonumber\\
&= \mathbb{E}\left(\sum_{f \in K_n}c(f)\right)^2 \mathbb{P}\left(E_{nice}^c\right), \label{sakiye}
\end{align}
since~\(E_{nice}\) depends only on the edge states and is therefore independent of edge weights and vertex marks.

Using~\(\left(\sum_{i=1}^{l}a_i\right)^2 \leq l \sum_{i=1}^{l}a_i^2\) and recalling that there are~\(m={n \choose 2}\) edges in~\(K_n,\) we get that
\begin{align}
\mathbb{E}\left(\sum_{f \in K_n} W(f)\right)^2 &\leq m\sum_{f \in K_n} \mathbb{E}c^2(f)   \nonumber\\
&\leq D_1m^2 \nonumber\\
&= D_1n^4, \nonumber
\end{align}
for some constant~\(D_1> 0,\) since the edge costs have bounded second moments (see Lemma~\ref{lemma_hz}\((c)\)).  Plugging this into~(\ref{sakiye}) and using the estimate~(\ref{e_nice_est}) for~\(E_{nice}\) we get that
\begin{equation}\label{i_wt_two_est}
I_{wt,2} \leq D_1 n^4 \cdot e^{-Dnp} \leq e^{-D_2 np}
\end{equation}
for some constant~\(D_2 > 0,\) provided~\(np \geq M\log{n}\) for large enough constant~\(M.\) This obtains an upper bound for~\(I_{wt,2}.\)

To estimate~\(I_{wt,1},\) we assume henceforth that~\(E_{nice}\) occurs so that both~\(G\) and~\(G(\{u_1,v_1\})\) are connected. Because~\(E_{deg} \supset E_{join}\) also occurs both~\(u_1\) and~\(v_1\) are adjacent to at least~\(D_0 np\) vertices in~\(G.\) Therefore the connectivity of~\(G(\{u_1,v_1\})\) ensures that~\(G_{rem}(f_1)\) is connected as well and we let~\({\cal T}_n\) and~\({\cal T}_{rem}(f_1)\) be the maximum cost spanning trees of~\(G\) and~\(G_{rem}(f_1),\) respectively. Clearly, any spanning tree of~\(G_{rem}(f_1)\) is also a spanning tree of~\(G\) and so~\({\cal T}_{rem}(f_1)\) has cost at most~\(\chi_n;\) i.e.,
\begin{equation}\label{tim_one}
\chi_{rem}(f_1) \leq \chi_n.
\end{equation}

For the reverse direction, we see that~\(\tau_{rem}(f_1) < \tau_n\) only if~\(f_1 \in {\cal T}_n.\) If we remove~\(f_1\) from~\({\cal T}_n,\) then we get two subtrees~\({\cal R}_a\) and~\({\cal R}_b\) of~\({\cal T}_n,\) that are also trees in~\(G_{rem}(f_1).\) Since~\(G_{rem}(f_1)\) is connected, there must exist an edge~\(h_{ab} \in G_{rem}(f_1)\) such that the union~\[{\cal T}_{ab} := {\cal R}_a \cup \{h_{ab}\} \cup {\cal R}_b\] is connected. This is illustrated in Figure~\ref{fig_sub_trees22}, where the removed edge~\(f_1\) is represented by the dotted line~\(AB\) and the trees~\({\cal R}_a\) and~\({\cal R}_b\) are denoted by the triangles~\(R_a\) and~\(R_b.\) Adding the edge~\(h_{ab} = CD\) gives the tree~\({\cal T}_{ab} \subset G_{rem}(f_1).\)

\begin{figure}[tbp]
\centering
%\fbox{
\includegraphics[width=3.5in, trim= 0 500 250 90, clip=true]{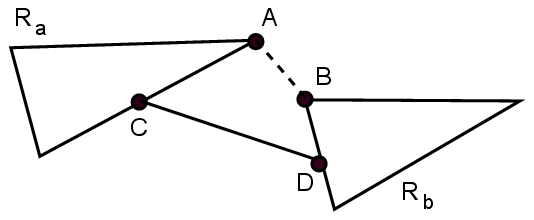}
%}
\caption{ Adding the edge~\(h_{ab} = CD\) to the trees~\(R_a\) and~\(R_b\) gives the spanning tree~\({\cal T}_{ab} \subset G_{rem}(f_1).\)}
\label{fig_sub_trees22}
\end{figure}

The tree~\({\cal T}_{ab}\) is a spanning tree of~\(G_{rem}(f_1)\) and has cost at least~\(\chi_n - c(f_1),\) where~\(c(f_1)\) as defined in~(\ref{cost_def}), is the cost of the edge~\(f_1\) and so we get that
\begin{equation}\label{tim_two}
\chi_{rem}(f_1) \geq \chi_n - c(f_1).
\end{equation}
Combining~(\ref{tim_two}) with~(\ref{tim_one}), we get that
\[|\chi_n-\chi_{rem}(f_1)| \ind(E_{nice}) \leq c(f_1) \ind(f_1 \in {\cal T}_n)\] and so squaring and taking expectations, we get that
\begin{equation}\label{anegan_ax}
I_{wt,1} \leq \mathbb{E}c^2(f_1)\ind\left(f_1 \in {\cal T}_n\right)= \frac{\mathbb{E}\rho_n}{m},
\end{equation}
where
\begin{equation}\label{rho_def}
\rho_n := \sum_{f \in K_n} c^2(f)\ind(f \in {\cal T}_n)
\end{equation}
is the sum of \emph{squares} of edge costs in the maximum cost spanning tree~\({\cal T}_n\) and the final estimate in~(\ref{anegan_ax}) follows from symmetry. As defined before,~\(m = {n \choose 2}\) is the number of edges in~\(K_n.\)

As in the proof of the expectation upper bound for~\(\chi_n\) described above, we use a segmentation approach to get that \[\rho_n \leq D \left(ny_n^2 + y_n^2\sum_{j}(j+1)^2N_{bad}(j)\right)\] for some constant~\(D > 0,\) where~\(y_n := J_c(np)\) and we recall that~\(N_{bad}(j)\) is the number of~\(j-\)bad edges in~\(G;\) i.e., the number of edges whose cost lies in the interval~\([2jy_n,2(j+1)y_n).\)   Using the fact that the scaling relation~(\ref{f_scale}) holds for some~\(s > 3\) strictly and  following an analogous analysis as in the derivation of the  upper bound for~\(\mathbb{E}\chi_n,\) we then get
\begin{equation}\label{rho_est}
\mathbb{E}\rho_n \leq D_1 n  H^2_c(np)
\end{equation}
for some constant~\(D_1  >0,\) provided~\(np \geq M\log{n}\) for some large constant~\(M  > 0.\) Plugging this into~(\ref{anegan_ax}), we get
\begin{equation}\label{i_one_est_ax}
mI_{wt,1} \leq D_1n  \cdot J_c^2(np).
\end{equation}

Combining~(\ref{i_one_est_ax}) with the estimate~(\ref{i_wt_two_est}) for~\(I_{wt,2}\) we get from~(\ref{i_wt_split_ax}) that
\begin{align}
mI_{wt} &= mI_{wt,1} + mI_{wt,2} \nonumber\\
&\leq D_1n \cdot H^2_c(np) + 4me^{-D_2np} , \nonumber
\end{align} where we recall that~\(m=n^2.\) Arguing as in the discussion following~(\ref{ewn_pen_fin}), we see that~\(J_c(np)\) is uniformly bounded away from zero. Therefore if~\(np \geq M\log{n}\) for a large enough constant~\(M,\) then~
\begin{equation}\label{i_wt_est_max}
mI_{wt} \leq 2D_1n \cdot H^2_c(np)
\end{equation}
for all~\(n\) large. This obtains the desired variance contribution estimate due to randomness in edge states and weights.

\emph{\underline{Step 2} (Estimating~\(I_{loc}\))}: To estimate the difference~\(\chi_n - \chi_{rem}(1),\) we proceed as in the proof of Step~\(1\) above with appropriate modifications. Recalling the events~\(E_{con}(.)\) and~\(E_{deg}(.)\) regarding the connectivity and vertex degrees of subgraphs of~\(G,\) we define the joint event
\begin{equation}\label{e_good_def}
E_{good} := E_{con}\left(\emptyset\right) \bigcap E_{con}(\{1\}) \bigcap E_{deg}(\{1\}) \bigcap E_{deg}(\emptyset)
\end{equation}
and argue as in~(\ref{e_nice_est}) to get that
\begin{equation}\label{e_good_est}
\mathbb{P}(E_{good}) \geq 1-4\exp\left(-Dnp\right),
\end{equation}
for some constant~\(D > 0.\) As in~(\ref{i_wt_split_ax}),  we split~\(I_{loc}\) as
\begin{equation}\label{i_loc_split_ax}
I_{loc} = I_{loc,1} + I_{loc,2},
\end{equation}
where
\[I_{loc,1} := \mathbb{E}\left(\chi_n-\chi_{rem}(1)\right)^2\ind(E_{good})\] and
\[I_{loc,2} := \mathbb{E}\left(\chi_n-\chi_{rem}(f_1)\right)^2\ind(E^c_{good})\] and estimate~\(I_{loc,2}\) and~\(I_{loc,1}\) in that order below.

If~\(E_{good}\) does not occur, i.e., if~\(E^c_{good}\) occurs, then we argue as in the derivation of~(\ref{i_wt_two_est}) to get that
\begin{equation}\label{i_loc_two_est}
nI_{loc,2} \leq e^{-D_1np}
\end{equation}
for some constant~\(D_1 > 0,\) provided~\(np \geq M\log{n}\) for a large enough constant~\(M.\)

Suppose now that~\(E_{good}\) occurs so that both~\(G\) and~\(G(\{1\})\) are connected. This implies that the maximum weight spanning tree of~\(G(\{1\})\) could be extended to obtain a spanning tree of~\(G\) and so~\(\chi_{rem}(1) \leq \chi_n.\) For the other direction, let~\({\cal T}_n\) be the maximum cost spanning tree of~\(G\) and let~\(d_{{\cal T}}(v)\)  be the degree of vertex~\(v\) in~\({\cal T}_n.\) also let~\(\{v_1,\ldots,v_w\}, w = d_{{\cal T}}(1)\) be the neighbours of the vertex~\(1\) in~\({\cal T}_n.\)

Removing the vertex~\(1,\) we obtain~\(t\) subtrees~\(\{{\cal S}_i\}_{1 \leq i \leq t}\) of~\(G(\{1\}).\) But because~\(G(\{1\})\) is connected, we add~\(t-1\) edges~\(\{h_j\}_{1 \leq j \leq t-1}\) to the union~\(\bigcup_{i=1}^{t} \{{\cal S}_i\}\) to obtain a spanning tree~\({\cal T}_{new}\) of~\(G(\{1\}).\) The above procedure is illustrated in Figure~\ref{fig_sub_trees} for the case~\(t=3.\)

\begin{figure}[tbp]
\centering
%\fbox{
\includegraphics[width=6in, trim= 220 200 50 110, clip=true]{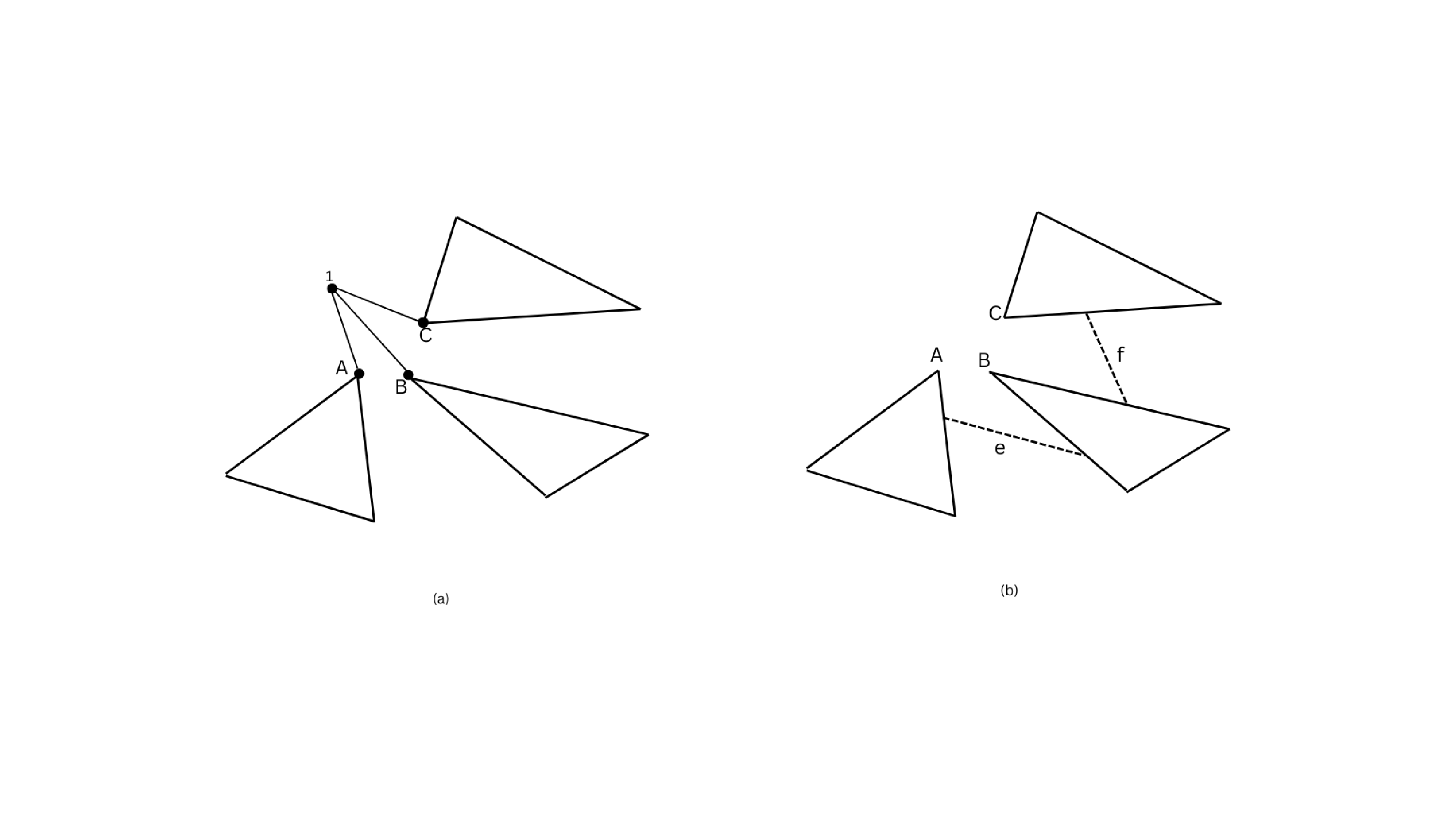}
%}
\caption{The subtrees~\({\cal S}_i, 1 \leq i \leq t=3\) containing the neighbours~\(v_1=A,v_2=B,v_3=C\) of the vertex~\(1\) are shown in~\((a).\) Removing vertex~\(1\) and adding the  edges~\(h_1=e\) and~\(h_2=f\) gives a spanning tree of the graph~\(G(\{1\})\) as shown in~\((b).\)}
\label{fig_sub_trees}
\end{figure}

The weight of~\({\cal T}_{new}\) is at least \[\chi_n - \sum_{j=1}^{t} c(1,v_j),\] where we recall from~(\ref{cost_def}) that~\(c(u,v)\) is the cost of the edge~\((u,v)\) with endvertices~\(u\) and~\(v.\) Thus
\[\chi_{rem}(1) \geq \chi_n - \sum_{j=1}^{t} c(1,v_j)\] and combining with the upper bound~\(\chi_{rem}(1) \leq \chi_n\) obtained before, we get that if~\(E_{good}\) occurs, then
\[|\chi_n - \chi_{rem}(1)| \leq \sum_{j=1}^{w} c(1,v_j).\] Letting~\(u \sim_{\tau} v\) denote that the vertices~\(u\) and~\(v\) are adjacent in~\({\cal T}_n,\) we get that
\begin{align}
|\chi_n - \chi_{rem}(1)|\ind(E_{good}) &\leq \sum_{u \sim_{\tau} 1} c(u,1) \ind(E_{good}) \nonumber\\
&\leq \sum_{u \sim_{\tau} 1} c(u,1) \ind(E_{deg}), \label{talpaa}
\end{align}
where we use the notation~\(E_{deg} = E_{deg}(\emptyset)\) and the final estimate in~(\ref{talpaa}) is true by the definition of~\(E_{good}\) in~(\ref{e_good_def}).

Squaring and taking expectations in~(\ref{talpaa}), we get
\begin{align}
I_{loc,1} &= \mathbb{E}\left(\chi_n - \chi_{rem}(1)\right)^2 \ind(E_{good}) \nonumber\\
&\leq \mathbb{E}\left(\sum_{u \sim_{\tau} 1} c(u,1)\right)^2 \ind(E_{deg}) \nonumber
\end{align}
and so using~\((\sum_{i=1}^{l} a_i)^2 \leq l \sum_{i=1}^{l}a_i^2\) and recalling that~\(d_{{\cal T}}(v)\) is the degree of~\(v\) in the maximum cost spanning tree~\({\cal T}_n,\) we get that
\begin{align}
I_{loc,1} &\leq \mathbb{E} \left(d_{{\cal T}}(1)\sum_{ u \sim_{\tau} 1}c^2(u,1) \right) \ind(E_{deg}) \nonumber\\
&= \frac{1}{n} \mathbb{E}\left(\sum_{v=1}^{n}d_{{\cal T}}(v) \sum_{u \sim_{\tau} v} c^2(u,v) \right) \ind(E_{deg}) \nonumber\\
&= \frac{1}{n} \mathbb{E}\ind(E_{deg}) \sum_{v=1}^{n} \sum_{u \neq v} c^2(u,v) d_{{\cal T}}(v) \ind(u \sim_{\tau} v) \nonumber\\
&= \frac{1}{n}\mathbb{E} \ind(E_{deg})\sum_{ (u,v) \in K_n} c^2(u,v) \left(d_{{\cal T}}(u) + d_{{\cal T}}(v)\right)  \ind((u,v) \in {\cal T}_n), \label{kalp_two}
\end{align}
where the first equality in~(\ref{kalp_two}) follows from symmetry and the  final summation in~(\ref{kalp_two}) is over all edges~\((u,v)\) in the complete graph~\(K_n.\)

Because~\(E_{deg}\) occurs, we know that each vertex has degree at most~\(D np\) for some constant~\(D > 0\) in the graph~\(G\) and so we get from~(\ref{kalp_two}) that
\[nI_{loc,1} \leq 2Dnp \mathbb{E}\rho_n,\]
where~\(\rho_n\) as defined in~(\ref{rho_def}) is the sum of \emph{squares} of edge costs in the maximum cost spanning tree~\({\cal T}_n.\) Further using the estimate~(\ref{rho_est}) for~\(\mathbb{E}\rho_n,\) we get that
\begin{equation}\label{tal_tum}
nI_{loc,1} \leq D_2 \cdot n^2p \cdot  H^2_c(np).
\end{equation}
for some constant~\(D_2 > 0.\)

Combining~(\ref{tal_tum}) with the estimate~(\ref{i_loc_two_est}) for~\(nI_{loc,2}\) and using~\(\mathbb{E}r^2(f) \leq \mu_{up}^2,\) we get that
\begin{equation}\label{i_loc_est_max}
nI_{loc} \leq D_3  n^2p \cdot  H^2_c(np).
\end{equation}
This obtains is the contribution to the variance due to the randomness in the vertex marks.  Finally, plugging the respective estimates~(\ref{i_loc_est_max})) and~(\ref{i_wt_est_max}) for~\(I_{loc}\) and~\(I_{wt}\) into the upper bound~(\ref{i_loc_wt_est_max}) for the variance of~\(\chi_n,\) we get the desired bound  in Lemma statement. This completes the proof  of the Lemma.~\(\qed\)

\emph{Proof of Theorem~\ref{thm_max_cst_up}}: To obtain the upper deviation bound in~(\ref{dev_bound_mast_up}),  we now argue as in the proof of Lemma~\ref{lemma_max_cst_up} to first get that
\begin{equation} \label{e_chi_n_new}
\mathbb{E} \chi_n \leq D n H_c(np) \;\;\; \text{ and } \;\;\; var(\chi_n) \leq D n^2p H_c(np)
\end{equation}
for some constant~\(D > 0.\) Indeed, say that an edge~\(f\) of the complete graph~\(K_n\) is~\(j-\)\emph{level} if its cost~\[c(f) \in \left[2jc_2H_c(np),2(j+1)c_2H_c(np)\right),\] where~\(c_2 > 0\) is the constant in the domination relation~(\ref{dom_cond}). Letting~\(N_{level}(j)\) be the total number of~\(j-\)level edges in~\(K_n\) and arguing as in the derivation of~(\ref{priscille_tits22_ax}), we get that
\begin{equation}\label{priscille_tits22_new}
\chi_n \leq 2c_2nH_c(np) + 2c_2H_c(np) \sum_{j \geq 1} (j+1)\cdot N_{level}(j).
\end{equation}

Using~(\ref{dom_cond}) and the fact that the edge weight ccdf satisfies the scaling relation~(\ref{f_scale}), we see that
\begin{align}
\mathbb{P}\left(c(f) > 2jc_2H_c(np)\right) &\leq \frac{1}{c_1} \mathbb{P}\left(W(f) > 2jH_c(np)\right) \nonumber\\
&\leq \frac{C_0}{c_1j^s} \mathbb{P}\left(W(f) > 2H_c(np)\right) \nonumber\\
&\leq \frac{C_0}{c_1j^s} \cdot \frac{1}{np},
\end{align}
by the definition of the inverse ccdf~\(H_c(.)\) as described in the paragraph containing~(\ref{h_c_def}). Therefore using~(\ref{priscille_tits22_new}) and arguing as in the derivation of the expectation upper bound for~\(\chi_n\) in~(\ref{exp_bound_up_mast}), we get the first estimate in~(\ref{e_chi_n_new}). Similarly, an analogous argument as in the derivation of the variance bound in~(\ref{exp_bound_up_mast}) obtains the second estimate in~(\ref{e_chi_n_new}). This completes the proof of~(\ref{e_chi_n_new}).

Finally, from the deviation lower bound in~(\ref{dev_bound_mast_low}) we also have that~\(\mathbb{E}\chi_n \geq D_1 n H_c(np)\) for some constant~\(D_1 > 0\) and so using the bounds in~(\ref{e_chi_n_new}) and invoking the Chebychev inequality
\[\mathbb{P}\left(|\chi_n - \mathbb{E}\chi_n| \geq \epsilon \mathbb{E} \chi_n\right) \leq \frac{var(\chi_n)}{(\epsilon\mathbb{E}\chi_n)^2}\] for~\(\epsilon > 0,\) we get
\begin{align}
&\mathbb{P}\left((1-\epsilon)D_1 nH_c(np) \leq \chi_n \leq (1+\epsilon) D nH_c(np)\right) \nonumber\\
&\;\;\;\;\;\geq\;\;1- \frac{Dn^2pH_c^2(np)}{D_1^2n^2H_c^2(np)} \nonumber\\
&\;\;\;\;\;=\;\;1-D_2p \label{skida_c}
\end{align}
for some constant~\(D_2 > 0.\) This obtains the desired upper bound in~(\ref{dev_bound_mast_up}) and therefore completes the proof of the Theorem.~\(\qed\)

\renewcommand{\theequation}{\thesection.\arabic{equation}}
\setcounter{equation}{0}
\section{Proof of Theorem~\ref{thm_spat_mast}}\label{sec_pf_mast_spat_thm}
\emph{Proof of Theorem~\ref{thm_spat_mast}\((a)\)}: Set~\(s_n := \frac{1}{\sqrt{np}}\) and divide the unit square~\(S\) into small disjoint~\(s_n \times s_n\) squares~\(\{R_i\}_{1 \leq i \leq N}\) where we assume for simplicity that~\(N = \frac{1}{s_n^2}\) is an integer; else we choose the side length of~\(R_i\) in the interval~\([s_n,2s_n)\) appropriately so that~\(N\) is an integer. This is possible since
\begin{equation}\label{s_n_est}
\left(\frac{1}{s_n}\right)^2 - \left(\frac{1}{2s_n}\right)^2 = \frac{3}{4s_n^2}  = \frac{3np}{4} \geq 1,
\end{equation}
for all~\(n\) large, since~\(np \rightarrow \infty\) by the statement of this Theorem. We label the squares as in Figure~\ref{fig_squares} so that~\(R_i\) and~\(R_{i+1}\) share an edge for each~\(1 \leq i \leq N-1.\)

Our strategy is to estimate the maximum size of matching in each~\(R_i\) and then collect these together to form an overall matching, whose size is comparable to~\(n.\) Each edge in the matching has length at most~\(s_n = \frac{1}{\sqrt{np}}\) and therefore cost factor at least~\((np)^{\alpha/2}.\) Since the edge weights are i.i.d., we invoke the law of large numbers to  obtain the desired lower bound. Details follow.

\begin{figure}[tbp]
\centering
%\fbox{
\includegraphics[width=3in, trim= 20 200 50 110, clip=true]{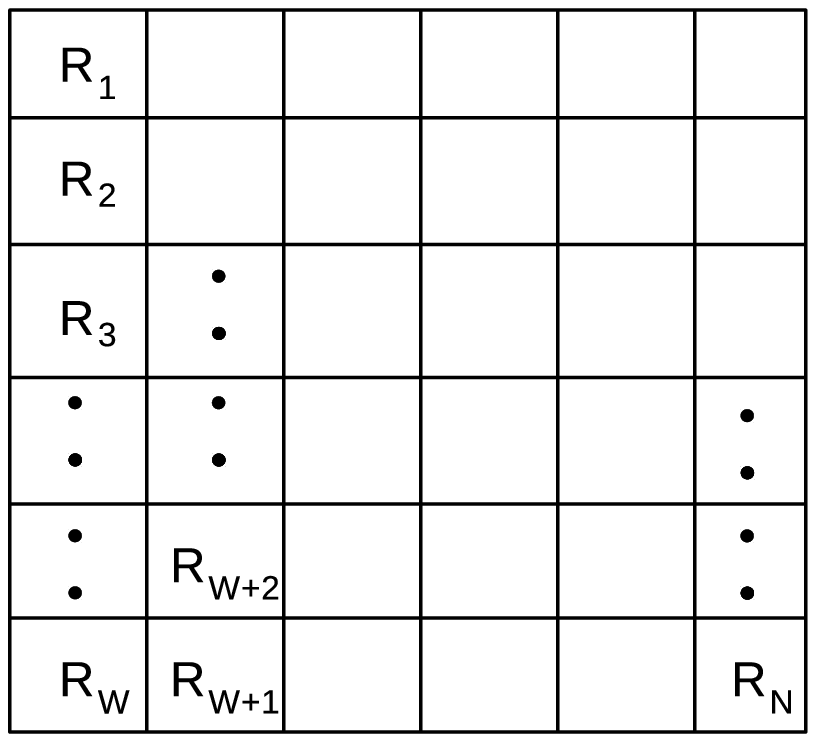}
%}
\caption{Tiling the unit square into~\(N = \frac{1}{s_n^2}\) smaller~\(s_n \times s_n\) squares~\(\{R_l\}_{1 \leq l \leq N}.\)}
\label{fig_squares}
\end{figure}

Let~\( {\cal V} =  {\cal V}(R_i)\) be the set of vertices located in~\(R_i\) and say that a set of edges~\(\{h_i\}_{1 \leq i \leq t}\) of~\(G\) is a \emph{matching} of size~\(t\) in~\(R_i,\) if~\(h_i\) and~\(h_j\) are vertex disjoint for~\(i \neq j\) and each~\(h_i\) has both its endvertices in~\({\cal V}.\) Letting~\(M_i\) denote the maximum size of a matching in~\(R_i,\) we show below that
\begin{equation}\label{match_i_est}
\mathbb{P}\left(M_i \geq \frac{\lambda}{p}\right) \geq 1- 2\exp\left(-\frac{\lambda}{p}\right),
\end{equation}
for some constant~\(\lambda > 0.\) Constants here and henceforth do not depend on the choice of~\(i\) or~\({\cal V}.\)

Indeed,  the size~\(N(R_i) := \#{\cal V}\) of the vertex set~\({\cal V}\) is Binomially distributed with parameters~\(n\) and~\(\int_{R_i} f \in [\epsilon_1 s_n^2, \epsilon_2 s_n^2]\) by the bounds in~(\ref{f_eq}). Consequently, the deviation estimate~(\ref{conc_est_f}) implies that
\begin{equation}\label{nri_est}
\mathbb{P}\left(D_1 ns_n^2 \leq N(R_i) \leq D_2ns_n^2\right) \geq 1-\exp\left(-D_1 ns_n^2\right),
\end{equation}
for some constants~\(D_1,D_2 > 0\) and each~\(1 \leq i \leq N = \frac{1}{s_n^2}.\) Recalling that~\(s_n= \frac{1}{\sqrt{np}}\) and setting \[E_{vert}(i) := \left\{D_1ns_n^2 \leq N(R_i) \leq D_2 ns_n^2\right\}\] and invoking the union bound, we get that
\begin{equation}\label{n_i_est_bax}
\mathbb{P}\left(E_{vert}(i) \right) \geq 1-\exp\left(-D_1ns_n^2\right).
\end{equation}

Assume henceforth that~\(E_{vert}(i)\) occurs and also that
\begin{equation}\label{L_def}
D_1 ns_n^2 = \frac{D_1}{p} =: 2L
\end{equation}
is even. Pick a subset~\({\cal W} \subset {\cal V}\) containing~\(2L\) vertices and split~\({\cal W} := {\cal L} \cup {\cal P}\) into two subsets containing~\(L\) vertices each. Say that~\({\cal M}\) is an~\(({\cal L},{\cal P})-\)matching if each edge of~\({\cal M}\) is present in~\(G\) and has one endvertex in~\({\cal L}\) and another endvertex in~\({\cal P}.\) Let~\({\cal M}(W)\) be an~\(({\cal L},{\cal P})-\)matching of maximum size~\(M_W\) and let~\({\cal A}_W \subset {\cal L}\) and~\({\cal B}_W \subset {\cal P}\) be the set of endvertices of edges, \emph{not} in~\({\cal M}(W).\)

In what follows, we estimate~\(M_W\) and thereby obtain a lower bound for the maximum size~\(M_i\) of a matching in the square~\(R_i.\) For~\({\cal A} \subset {\cal L}\) and~\({\cal B} \subset {\cal P},\) let~\(F({\cal A}, {\cal B})\) be the event that no edge of~\(G\) has one endvertex in~\({\cal A}\) and another endvertex in~\({\cal B}.\) If~\(M_W \leq \varepsilon L\) for some constant~\(0 < \varepsilon < 1,\) then the sets~\({\cal A}_W\) and~\({\cal B}_W\) each have size at least~\((1-\varepsilon)L\) and moreover~\(F({\cal A}_W, {\cal B}_W)\) occurs. This is because if there existed an edge~\(f_M\) of~\(G\) with one endvertex in~\({\cal A}_W\) and another endvertex in~\({\cal B}_W,\) then adding~\(f_M\) to~\({\cal M}(W)\) gives a~\(({\cal L},{\cal P})-\)matching of size at least~\(M_W+1,\) a contradiction to the maximality of~\({\cal M}(W).\)

Summarizing,
\[\{M_W \leq \varepsilon L\} \subseteq \bigcup_{{\cal A} \subset {\cal L}} \bigcup_{{\cal B} \subset {\cal P}} F({\cal A}, {\cal B}),\]
where the union is over all sets~\({\cal A} \subset {\cal L}\)  and~\({\cal B} \subset {\cal P},\) each containing~\(t \geq (1-\varepsilon)L\) vertices.
An application of the union bound gives
\begin{equation}\label{m_ax_est}
\mathbb{P}\left(M_W \leq \varepsilon L  \mid {\cal V}\right) \leq \sum_{{\cal A} \subset {\cal L}} \sum_{{\cal B} \subset {\cal P}} \mathbb{P}\left(F({\cal A}, {\cal B}) \mid {\cal V}\right),
\end{equation}
where~\(\mathbb{P}(. \mid {\cal V})\) is the probability distribution conditioned on the set of vertices~\({\cal V}\) located in~\(R_i.\)

For a fixed~\({\cal A}, {\cal B}\) each having at least~\((1-\varepsilon)L\) vertices, we see that the event~\(F({\cal A}, {\cal B})\) happens with probability at most
\begin{align}
(1-p)^{(1-\varepsilon)^2L^2} &\leq \exp\left(- (1-\varepsilon)^2 pL^2 \right)  \nonumber\\
&= \exp\left(- \frac{D_1(1-\varepsilon)^2 L}{2} \right) \nonumber\\
&\leq \exp\left(-\frac{D_1L}{8}\right), \label{chimi_changa}
\end{align}
where the second estimate in~(\ref{chimi_changa}) is true by the definition of~\(L\) in~(\ref{L_def}) and the final bound in~(\ref{chimi_changa}) is true, provided we choose~\(0 < \varepsilon < \frac{1}{2}.\)

The number of choices for~\({\cal A}\) is~\[ \sum_{(1-\varepsilon)L \leq k \leq L} {L \choose k} = \sum_{0 \leq k \leq \varepsilon L} {L \choose k} \leq \varepsilon L \cdot {L \choose \varepsilon L},\] by the unimodality of the binomial coefficient and the fact that~\(0 < \varepsilon < \frac{1}{2}.\) Further using~\({a \choose b} \leq \left(\frac{ae}{b}\right)^{b},\) we get that
\begin{align}\label{samsa_ax}
\varepsilon L {L \choose \varepsilon L} &\leq \varepsilon L  \left(\frac{e}{\varepsilon}\right)^{\varepsilon L} \nonumber\\
&\leq \varepsilon L \exp\left(C|\varepsilon \log{\varepsilon}|L\right) \nonumber\\
&\leq \exp\left(2C|\varepsilon \log{\varepsilon}|L\right) \nonumber\\
&\leq \exp\left(\frac{D_1L}{32}\right),
\end{align}
for all~\(n\) large, where~\(C> 0\) is a constant not depending on the choice of~\(\varepsilon,\) the constant~\(D_1 > 0\) is as in~(\ref{chimi_changa}), the penultimate bound in~(\ref{samsa_ax}) is true since~\(L \leq n\) is a constant multiple of~\(\frac{1}{p}\) by~(\ref{L_def}) and~\(\frac{1}{p} \rightarrow \infty\) by the statement of this Theorem and the final estimate in~(\ref{samsa_ax}) is valid if~\(\varepsilon  = \varepsilon(C,D_1)> 0\) is small enough, since~\(|x \log{x}| \rightarrow 0\) as~\(x \rightarrow 0.\)

The estimate~(\ref{samsa_ax}) holds for~\({\cal B}\) as well and so plugging~(\ref{samsa_ax}) and~(\ref{chimi_changa}) into~(\ref{m_ax_est}), we get that if the event~\(E_{vert}\) occurs (so that~\({\cal V}\) has at least~\(2L\) vertices), then
\begin{equation}\nonumber
\mathbb{P}\left(M_W \leq \varepsilon L  \mid {\cal V}\right) \leq \left(\exp\left(\frac{D_1L}{32}\right)\right)^2 \cdot \exp\left(-\frac{D_1L}{8}\right) = \exp\left(-\frac{D_1L}{16}\right)
\end{equation}
for some constant~\(D_2 > 0.\) Since~\(L\) is  of the order of~\(\frac{1}{p}\) by definition (see~(\ref{L_def})), we further get
\begin{equation}\nonumber
\mathbb{P}\left(M_W \leq \frac{D_3}{p}  \,\middle\vert\,  {\cal V}\right) \leq  \exp\left(-\frac{D_3}{p}\right)
\end{equation}
for some constant~\(D_3 > 0.\) Averaging over all possible~\({\cal V}\) and recalling that the above estimate is valid if~\(E_{vert}\) occurs, we get that
\begin{equation}\label{schmix}
\mathbb{P}\left(M_W \leq \frac{D_3}{p}  \,\middle\vert\,  E_{vert}\right) \leq  \exp\left(-\frac{D_3}{p}\right).
\end{equation}

For any two events~\(A\) and~\(B,\) we have that
\begin{align}
\mathbb{P}(A) &= \mathbb{P}(A \mid B) \mathbb{P}(B) + \mathbb{P}(A \mid B^c) \mathbb{P}(B^c) \nonumber\\
&\leq \mathbb{P}(A \mid B) + \mathbb{P}(B^c) \nonumber
\end{align}
and setting~\(A = \left\{M_W \leq \frac{D_3}{p}\right\}\) and~\(B = E_{vert}\) and combining~(\ref{schmix}) with the estimate~(\ref{n_i_est_bax}) for~\(E_{vert},\)  we finally get that
\begin{equation}\label{mi_est}
\mathbb{P}\left(M_W \leq \frac{D_3}{p}\right) \leq \exp\left(-\frac{D_3}{p}\right)   +  \exp\left(-D_1ns_n^2\right) \leq 2\exp\left(-\frac{D_4}{p}\right)
\end{equation}
for some constant~\(D_4 > 0,\) since~\(ns_n^2 = \frac{1}{p}\) by definition. Since the maximum size~\(M_i\) of a matching in~\(R_i\) is at least~\(M_W,\) this obtains~(\ref{match_i_est}).

The bound~(\ref{match_i_est}) estimates the probability that the maximum size~\(M_i\) of a matching in the square~\(R_i,\) is at least of the order of~\(\frac{1}{p}.\) There are~\(N = \frac{1}{s_n^2} = np \leq n\) squares in the tiling~\(\{R_i\}\) shown in Figure~\ref{fig_squares} and so setting \[E_{match} := \bigcap_{i=1}^{N} \left\{M_i \geq \frac{\lambda}{p}\right\}\] and invoking the union bound, we get that
\begin{equation}\label{e_match_est}
\mathbb{P}\left(E_{match}\right) \geq 1- n \cdot \exp\left(-\frac{\lambda}{p}\right).
\end{equation}
Recalling that~\(E_{con}\) denotes the event that the random graph~\(G\) is connected, we define~\(E_{net} := E_{con} \cap E_{match}\) and get from the connectivity estimate~(\ref{e_con_est_max}) and the union bound that
\begin{equation}\label{e_net_est_aabb}
\mathbb{P}\left(E_{net}\right) \geq  1-\exp\left(-\lambda_1np\right)-n \cdot \exp\left(-\frac{\lambda}{p}\right),
\end{equation}
for some constant~\(\lambda_1 > .\)

Suppose henceforth that~\(E_{net}\) occurs and for~\(1 \leq i \leq N = \frac{1}{s_n^2}  =np,\) let~\({\cal W}_i\) be a matching of maximum size in the~\(s_n \times s_n\) square~\(R_i\) so that each~\({\cal W}_i\) has at least~\(\frac{\lambda}{p}\) edges. The union~\({\cal W} := \bigcup_{i=1}^{N} {\cal W}_i\) is matching of~\(G\) containing~\(\frac{\lambda}{p} \cdot N = \lambda n\) edges. Arguing as in the discussion prior to~(\ref{azuma}), we extend~\({\cal W}\) to obtain a spanning tree~\({\cal S}_{fin}\) of~\(G.\) Moreover, the edges in~\({\cal W}\) have i.i.d.\ weights with bounded second moments and so  setting~\[V_{fin} := \sum_{h \in {\cal W}} W(h)\] be the total \emph{weight} of all edges in~\({\cal W},\) we use the Chebychev inequality and argue as in the discussion prior to~(\ref{e_net_est_max2}),  to get that
\begin{equation}\label{e_net_est_aabb_2}
\mathbb{P}\left(\{V_{fin} \geq \lambda_2 n \} \bigcap E_{net}\right) \geq 1-\frac{\lambda_3}{n} - \exp\left(-\lambda_1np\right)-n \cdot \exp\left(-\frac{\lambda}{p}\right)
\end{equation}
for some constants~\(\lambda_2,\lambda_3 >0.\)

Each edge in~\({\cal W}\) has length at most~\(s_n\sqrt{2} = \sqrt{\frac{2}{np}}\) and so an edge cost factor of at least~\(\left(\frac{1}{s_n\sqrt{2}}\right)^{\alpha} = \lambda_4 (np)^{\alpha/2}\) for some constant~\(\lambda_4 > 0.\) Consequently, arguing as in the discussion prior to~(\ref{chill_ammal}), we get that if~\(\{V_{fin} \geq \lambda_2 n \} \bigcap E_{net}\) occurs, then~\(\chi_n \geq \lambda_5 n (np)^{\alpha/2}\) for some constant~\(\lambda_5 > 0.\) Since~\(np \geq M\log{n}\) by Theorem statement, we choose the constant~\(M >0\) large enough so the estimate~(\ref{e_net_est_aabb_2}) obtains the desired bound~(\ref{dev_bds_mast_spat}). This completes the proof of part~\((a)\) of the Theorem.~\(\qed\)

%Let~\(F^{(fc)}(x) := \mathbb{P}\left(d^{-\alpha}(X_u, X_v) > x\right)\) be the edge cost factor ccdf and let~\(J^{(fc)}_c(z)\) be the inverse edge cost factor ccdf, as defined in~(\ref{h_c_def}). denote the edge cost facto

\emph{Proof of Theorem~\ref{thm_spat_mast}\((b)\)}: Our strategy is to use the proof strategy of Lemma~\ref{lemma_max_cst_up}  and begin with some preliminary computations. Indeed, the condition~(\ref{p_cond_new}) in the statement of Theorem~\ref{thm_max_cst_low} is trivially true with say~\(a_0 = \frac{1}{2}, b_0 = 1\) and~\(\gamma_0 = \frac{1}{4},\) since all edges have the same probability~\(p.\)

The next step is to establish that there is a constant~\(D > 0\) such that for all~\(x > D,\) we have
\begin{equation}\label{fct_scale}
\frac{1}{Dx^{2/\alpha}} \leq F^{(ct)}_c(x)  = \mathbb{P}\left(c(u,v) > x\right) \leq \frac{D}{x^{2/\alpha}},
\end{equation}
where we recall for convenience that~\(c(u,v) = d^{-\alpha}(X_u,X_v) W(u,v)\) is the cost of the edge~\((u,v)\) with endvertices~\(u\) and~\(v\) and~\(W(u,v)\) is the edge weight with ccdf~\(F_c(.).\) Constants here and henceforth do not depend on the constant~\(\gamma > 0\) in the statement of the Theorem. Also as a by product of the lower bound in~(\ref{fct_scale}), we have that
\[\mathbb{E}c^2(u,v) = \int x F^{(ct)}_c(x) \geq \int_{D}^{\infty} \frac{dx}{Dx^{-1+2/\alpha}}  = \infty,\] if~\(\alpha \geq 1.\) This demonstrates the unboundedness of the edge cost second moment for~\(\alpha \geq 1,\) as described in the discussion following the statement of Theorem~\ref{thm_spat_mast}.

%GET ALS LOWER BD AND EXPLAIN THAT  ALPHA > 1 UBDDD  CST ETC....

We begin by demonstrating that there are constants~\(C_1, C_2 > 0\) such that
\begin{equation}\label{ec_dst_est}
\mathbb{P}\left(d(X_u, X_v) \leq y \right) \geq C_1y^2 \text{ for all } 0 < y < \frac{1}{2}
\end{equation}
and
\begin{equation}\label{ec_dst_est_2}
\mathbb{P}\left(d(X_u, X_v) \leq y \right) \leq C_2y^2 \text{ for all }  y >0.
\end{equation}
Indeed, if~\(B(x,a)\) is the ball of radius~\(a\) centred at~\(x \in S,\) then given~\(X_u,\) we see that~\(X_v\) lies within distance~\(y\) from~\(X_u\) with probability~\(\int_{B(X_v,y) \cap S} f \leq \epsilon_2 \pi y^2,\) by the density upper bound in~(\ref{f_eq}). Averaging over~\(X_u,\) we get~(\ref{ec_dst_est_2}). Similarly, if~\( y < \frac{1}{2},\) then irrespective of the location~\(X_v,\) at least one quadrant of the ball~\(B(X_v,y)\) is contained in~\(S\) and so the lower density bound in~(\ref{f_eq}) implies that~\(\int_{B(X_v,y) \cap S} f \geq\frac{\epsilon_1 \pi y^2}{4}.\) Again averaging over~\(X_u,\) we get~(\ref{ec_dst_est}).

Using~(\ref{ec_dst_est}), we get the lower bound in~(\ref{fct_scale}) as follows. Let~\(w_0 > 0\) be small enough so that~\[\mathbb{P}\left(W(u,v) \geq w_0\right) \geq \frac{1}{2}.\] We then get
\begin{align}
\mathbb{P}\left(c(u,v) > x\right) &= \mathbb{P}\left(d^{-\alpha}(X_u, X_v)W(u,v) > x\right) \nonumber\\
&\geq  \mathbb{P}\left(d^{-\alpha}(X_u, X_v)W(u,v) > x, W(u,v) \geq w_0\right) \nonumber\\
&\geq  \mathbb{P}\left(d^{-\alpha}(X_u, X_v) > \frac{x}{w_0}, W(u,v) \geq w_0\right)\nonumber\\
&= \mathbb{P}\left(d^{-\alpha}(X_u, X_v) > \frac{x}{w_0}\right)\mathbb{P}\left(W(u,v) \geq w_0\right) \nonumber\\
&\geq \frac{1}{2}\mathbb{P}\left(d^{-\alpha}(X_u, X_v) > \frac{x}{w_0}\right) \nonumber\\
&= \frac{1}{2}\mathbb{P}\left(d(X_u, X_v) < \left(\frac{w_0}{x}\right)^{1/\alpha}\right) \nonumber
\end{align}
For~\(x > 2w_0,\) the estimate in~(\ref{ec_dst_est}) obtains the lower bound in~(\ref{f_scale}).

For the upper bound in~(\ref{f_scale}), we argue as follows. The edge weights have bounded~\(q^{th}\) moment for some~\(q \geq \frac{2}{\alpha}+1,\) by Theorem statement and so recalling that~\(F_c\) denotes the edge weight ccdf, we have that \[\mathbb{E}W^{1+2/\alpha}(u,v) = \int x^{2/\alpha} F_c(x) dx  < \infty.\] Splitting the integral, we get that \[\int x^{2/\alpha} F_c(x) dx = \sum_{k \geq 1} \int_{x=k-1}^{k} x^{2/\alpha} F_c(x) dx \geq \sum_{k \geq 1} (k-1)^{2/\alpha} F_c(k+1)\] is finite. Further using~\((a+b)^{z} \leq 2^{z}(a^{z} + b^{z}) \leq 2^{z+1} a^{z}\) with~\(a=k-1\) and~\(b=1,\) we get that
\begin{equation}\label{skilma}
\sum_{k \geq 1} k^{2/\alpha} \cdot F_c(k)  < \infty.
\end{equation}

Using~(\ref{skilma}) and~(\ref{ec_dst_est_2}), we prove~(\ref{fct_scale}) as via a chain of relations as follows:
\begin{align}
\mathbb{P}\left(c(u,v) > x\right) &= \mathbb{P}\left(d^{-\alpha}(X_u,X_v) W(u,v) > x\right) \nonumber\\
&= \sum_{k \geq 0} \mathbb{P}\left(d^{-\alpha}(X_u,X_v)W(u,v) > x,\;k \leq W(u,v) < k+1\right) \nonumber\\
&\leq \sum_{k \geq 0 } \mathbb{P}\left(d^{-\alpha}(X_u,X_v) > \frac{x}{k+1},\; k \leq W(u,v) < k+1\right)\nonumber\\
&= \sum_{k \geq 0 } \mathbb{P}\left(d^{-\alpha}(X_u,X_v) > \frac{x}{k+1}\right) \mathbb{P}\left(k \leq W(u,v) < k+1\right)\nonumber\\
&\leq \sum_{k \geq 0 } \mathbb{P}\left(d^{\alpha}(X_u,X_v) < \frac{k+1}{x}\right) F_c(k) \nonumber\\
&= \sum_{k \geq 0 } \mathbb{P}\left(d(X_u,X_v) < \left(\frac{k+1}{x}\right)^{1/\alpha}\right) F_c(k) \nonumber\\
&\leq \sum_{k \geq 0 } C_2 \left(\frac{k+1}{x}\right)^{2/\alpha} F_c(k) \label{skisha_2}\\
&\leq \frac{C_3}{x^{2/\alpha}},\label{skisha}
\end{align}
for some constant~\(C_3 > 0,\) where~\(C_2 > 0\) is the constant in~(\ref{ec_dst_est_2}) and~(\ref{skisha}) is a consequence of the summation condition~(\ref{skilma}) in the statement of this Theorem. This proves~(\ref{fct_scale}).

The estimates~(\ref{fct_scale}) and~(\ref{ec_dst_est}) together facilitate the application of the proof strategy of Lemma~\ref{lemma_max_cst_up}. Indeed, for~\(j \geq 0\) say that an edge~\((u,v)\) is~\(j-\)\emph{bad} if its cost~\[c(u,v) \in [2j\cdot(np)^{\alpha/2}, 2(j+1)\cdot (np)^{\alpha/2}).\] Using~(\ref{fct_scale}) and the fact that~\(s := \frac{2}{\alpha} > 2,\) we argue as in the derivation of~(\ref{priscille_tits22_ax}) to get that
\begin{equation}\label{priscille_tits33_ax}
\chi_n \leq 2n(np)^{\alpha/2} + \sum_{j \geq 1} 2(j+1)(np)^{\alpha/2}N_{bad}(j),
\end{equation}
where~\(N_{bad}(j)\) is the total number of~\(j-\)bad edges in the random graph~\(G.\) Further following the analysis preceding~(\ref{ewn_pen_fin}), we get that
\begin{equation}\label{ewn_pen_fin_22_ax_ax}
\mathbb{E}\chi_n \leq Dn(np)^{\alpha/2} + D n^2e^{-D_1np}
\end{equation}
for some constants~\(D,D_1 > 0.\) Since~\(np \geq \gamma \log{n},\) we choose~\(\gamma >0\) large enough  to obtain the desired expectation upper  bound for~\(\chi_n\) in the Theorem statement. A direct application of the Markov inequality then implies the deviation upper bound~(\ref{dev_bds_mast_spat_up}) for~\(\chi_n\) for the case~\(\frac{2}{3} < \alpha < 1.\)

If~\(\alpha < \frac{2}{3},\) then again using~(\ref{fct_scale}) and the fact that~\(s = \frac{2}{\alpha} > 3,\) we argue as in the variance bound for~\(\chi_n\) in~(\ref{exp_bound_up_mast}) to get  that~\(var(\chi_n) \leq C n^2p (np)^{\alpha}\)
for~\(C > 0\) constant. Using the Chebychev inequality and arguing as in the derivation of the estimate~(\ref{skida_c}) in the proof of Theorem~\ref{thm_max_cst_up},  we get the deviation upper bound~(\ref{dev_bds_mast_spat_up}) for~\(\chi_n\) for the case~\(0 < \alpha < \frac{2}{3}.\) This completes the proof of the Theorem.~\(\qed\)

%\begin{equation}\label{exp_bound_up_mast}
%\mathbb{E}\chi_n \leq \gamma n  \cdot J_c(np)\;\;\;\text{ and }\;\;\; var(\chi_n) \leq  \gamma n^2p   \cdot J_c^2(np).
%\end{equation}

%Also suppose there are constants~\(C_0,x_0 > 0\) and~\(s  > 2\) such that the edge \emph{cost} ccdf~\(F_c^{(ct)}\) satisfies the scaling relation~(\ref{f_scale}) for all~\(a > 1\) and all~\(x  >x_0.\) There are constants~\(M,\gamma >0\) such that if~\(p \geq \frac{M\log{n}}{n},\) then
%\begin{equation}\label{exp_bound_up_mast}
%\mathbb{E}\chi_n \leq \gamma n  \cdot J_c(np)\;\;\;\text{ and }\;\;\; var(\chi_n) \leq  \gamma n^2p   \cdot J_c^2(np).
%\end{equation}

\renewcommand{\theequation}{\thesection.\arabic{equation}}
\setcounter{equation}{0}
\section{Proof of Corollaries~\ref{cor_repeat_mast} and~\ref{cor_example_two}}\label{sec_pf_mast_cor}
\emph{Proof of Corollary~\ref{cor_repeat_mast}}: We see that the graph~\(G\) is homogenous with edge probability~\(p\) and so~(\ref{p_cond_new}) is satisfied. Moreover,~\(r(X_i,X_j) \leq 1\) by definition and~\[\mathbb{E}r(X_i,X_j) \geq \mathbb{P}\left(\{X_i=0\} \cup \{X_j =0\} \right) \geq 2\epsilon_0 - \epsilon_0^2 > 0\] and so~(\ref{cost_fact_cond}) also holds. The edge weight ccdf~\(F_c(x) = e^{-x}\) for~\(x > 0\) and so the inverse ccdf~\(H_c(.)\) as defined in~(\ref{h_c_def}) satisfies~\(H_c(z) = \log{z}\) for~\(z > 1.\) Plugging this into~(\ref{dev_bound_mast_low}), we see that
\begin{equation}\label{faap_Two}
\mathbb{P}\left(E_{con} \bigcap \{\chi_n \geq Dn \log(np)\}\right) = 1-o(1),
\end{equation}
for some constant~\(D > 0.\)

We now use Theorem~\ref{thm_max_cst_up} to demonstrate the near optimality of~(\ref{faap_Two}). As a first step, we ensure that the conditions~\((a)-(b)\) are satisfied. Indeed, the edge weights are exponential, i.e.,~\(F_c(x) = e^{-x}\) for~\(x > 0\) and so for any~\(a, x > 1\) and~\(s > 0\) we have that
\[e^{-ax} = e^{-(a-1)x}e^{-x} \leq e^{-(a-1)} e^{-x} \leq \frac{C}{a^s} e^{-x}\]
for some sufficiently large constant~\(C = C(a,s) > 0.\) Thus~(\ref{f_scale}) is true. Moreover, from the definition of the cost factor in~(\ref{cost_def_obi}), we see that~\(r(X_i,X_j) \leq 1\) and so the cost~\(c(i,j)\) of the edge~\((i,j)\) satisfies~\(c(i,j) \leq W(i,j).\) For~\(x > 0,\) this implies that
\[F_c^{(ct)}(x) = \mathbb{P}(c(i,j) > x) \leq \mathbb{P}\left(W(i,j) > x\right) = F_c(x)\] and so the domination condition~(\ref{dom_cond}) also holds. Consequently, the bound~(\ref{dev_bound_mast_up}) implies that
\begin{equation}\label{faap_three}
\mathbb{P}\left(E_{con} \bigcap \{\chi_n \leq D_0n \log(np)\}\right) = 1-o(1)
\end{equation}
for some constant~\(D_0 > 0.\) Combining with~(\ref{faap_Two}), we obtain the deviation bounds in~(\ref{faap_gen}).

%EXP BDS HEEEE!!!!

Finally, the variance bound in~(\ref{dev_bound_exp_up}) implies that~\(var(\chi_n) \leq \theta n^2p \left(\log(np)\right)^2\) for some constant~\(\theta > 0\) and the lower deviation bound in~(\ref{faap_Two}) gives us that~\(\mathbb{E}\chi_n \geq \theta_0 n \log(np)\) for some constant~\(\theta_0 > 0.\) Therefore
\[\mathbb{E}\left(\frac{\chi_n}{\mathbb{E}\chi_n}-1\right)^2 \leq \frac{\theta }{\theta_0^2} \cdot p \rightarrow 0,\] provided~\(p = o(1).\) ~\(\qed\)

\emph{Proof of Corollary~\ref{cor_example_two}\((a)\)}: We verify that the conditions in Theorem~\ref{thm_max_cst_low}-\ref{thm_max_cst_up} hold in that order. Clearly, since~\(p(u,v)= p \geq \frac{1}{n^{\beta}},\) the condition~(\ref{p_cond_new}) holds, say, with~\(a_0 = \gamma_0 = \frac{3}{4}\) and~\(b_0 = 1.\)
By Corollary statement, the edge weight ccdf~\(F_c(x)\) is continuous for all large~\(x\) and by definition, the edge cost factor satisfies~\[r(X_u,X_v) = d^{-\alpha}(X_u,X_v) \geq (\sqrt{2})^{-\alpha},\] since the distance between any two points in the unit square is at most~\(\sqrt{2}.\) This obtains the lower expectation bound for the cost factor in~(\ref{cost_fact_cond}).

To get the upper bound in~(\ref{cost_fact_cond}), we use the estimate~(\ref{ec_dst_est_2}) to deduce that
\begin{equation}\label{dim_dim}
\mathbb{P}\left(d^{-\alpha}(X_u,X_v)>x\right) \leq\frac{D}{x^{2/\alpha}} \text{ for all }x >x_0,
\end{equation}
where~\(D, x_0 > 0\) are large constants not depending on~\(x.\) Consequently~\(\mathbb{E}r^2(X_u, X_v) = \mathbb{E}d^{-2\alpha}(X_u, X_v)\) is at most \[ x_0^2 + \int_{x_0}^{\infty} x\mathbb{P}\left(d^{-\alpha}(X_u, X_v)  > x\right) dx \leq x_0 +\int_{x_0}^{\infty}\frac{\pi \epsilon_2}{x^{2/\alpha-1}} < \infty,\] since~\(\alpha < 1.\) Thus the bounds in~(\ref{cost_fact_cond}) are satisfied and so the  conditions in Theorem~\ref{thm_max_cst_low} hold.

Next, to verify the scaling condition~(\ref{f_scale}) in the statement of Theorem~\ref{thm_max_cst_up}, we use  the bounds for the edge weight ccdf in~(\ref{heav_tail_ccdf}). Indeed, if~\(a > 1\) and~\(x > 0\) is large, then~(\ref{heav_tail_ccdf}) implies that \[F_c(ax) \leq \frac{A_2}{a^{s}x^{s}}=  \frac{A_2/A_1}{a^{s}} \cdot \frac{A_1}{x^s} \leq \frac{A_2/A_1}{a^{s}}  F_c(x).\] Thus~(\ref{f_scale}) is true.

Finally, to see if the domination condition~(\ref{dom_cond}) is true, we first show that if there are constants~\(D_1,D_2 > 0\) such that the edge weight ccdf~\(F_c\) satisfies
\begin{equation}\label{aazhi}
\sum_{k\geq 1} F_c\left(\frac{x}{k}\right) \cdot \frac{1}{k^{2/\alpha}} \leq D_1 F_c(D_2x)
\end{equation}
for all~\(x>0,\) then~(\ref{dom_cond}) holds. Later we verify that~(\ref{aazhi}) indeed holds in the current example.

Recalling that~\(c(u,v) = d^{-\alpha}(X_u,X_v)W(u,v)\) is the cost of the edge~\((u,v)\) with endvertices~\(u\) and~\(v,\) we have that
\begin{align}
\mathbb{P}\left(c(u,v) > x\right) &= \sum_{k \geq 0 } \mathbb{P}\left(d^{-\alpha}(X_u,X_v)W(u,v) > x, k \leq  d^{-\alpha}(u, v) < k+1\right) \nonumber\\
&\leq \sum_{k \geq 0 } \mathbb{P}\left(W(u,v) > \frac{x}{k+1}, k \leq  d^{-\alpha}(u, v) < k+1\right)\nonumber\\
&=\sum_{k \geq 0 } F_c\left(\frac{x}{k+1}\right) \mathbb{P}\left(k \leq  d^{-\alpha}(u, v) < k+1\right)\nonumber\\
&\leq \sum_{k \geq 0 } F_c\left(\frac{x}{k+1}\right) \mathbb{P}\left(d^{-\alpha}(u, v) \geq k\right)\nonumber\\
&\leq F_c(x) + \sum_{k \geq 1} F_c\left(\frac{x}{k+1}\right)\frac{D}{k^{2/\alpha}}\nonumber
\end{align}
for some constant~\(D > 0,\) by~(\ref{dim_dim}). Using~(\ref{aazhi}), we then get that~(\ref{dom_cond}) holds.

Finally, it remains to verify that~(\ref{aazhi}) is true. The ccdf bounds in~(\ref{heav_tail_ccdf}) imply that~\(\frac{B_1}{y^{s}} \leq F_c(y) \leq \frac{B_2}{y^s}\) for \emph{all}~\(y > 1\) and some constants~\(B_1,B_2 > 0\) and so if~\(k \leq x,\) then we get that~\[F_c\left(\frac{x}{k}\right) \leq \frac{B_2 k^{s}}{x^s}.\] For~\(k >x,\) we simply use~\(F_c(.) \leq 1.\) Splitting the summation in~(\ref{aazhi}) into~\(\sum_{k \leq x}\) and~\(\sum_{k > x}\) we then get that
\[ \sum_{k\geq 1} F_c\left(\frac{x}{k}\right) \cdot \frac{1}{k^{2/\alpha}} \leq I_1 + I_2,  \]
where \[I_1 := \sum_{k \leq x} \frac{B_2 k^{s}}{x^s} \cdot \frac{1}{k^{2/\alpha}} \text{ and } I_2 := \sum_{k > x} \frac{1}{k^{2/\alpha}}.\]

We have that \[I_1 \leq \frac{B_2}{x^s} \sum_{k \geq 1} \frac{1}{k^{2/\alpha-s}} \leq \frac{B_3}{x^s}\] for some constant~\(B_3 > 0,\) since~\(\frac{2}{\alpha} > s+1.\) Similarly, comparing with integrals we have that \[I_2 \leq B_4\int_{x}^{\infty} \frac{dy}{y^{2/\alpha}} \leq \frac{B_5}{x^{2/\alpha-1}},\] for some constants~\(B_4,B_5 > 0,\) not depending on the choice of~\(x.\) Combining these two estimates and using the fact that~\(s < \frac{2}{\alpha}-1,\) we get that
\[\sum_{k\geq 1} F_c\left(\frac{x}{k}\right) \cdot \frac{1}{k^{2/\alpha}} \leq \frac{B_3}{x^s} + \frac{B_5}{x^{2/\alpha-1}} \leq \frac{B_6}{x^{s}} \leq B_7F_c(x),\] for some constants~\(B_6,B_7> 0,\) again not depending on the choice of~\(x.\) Thus~(\ref{aazhi}) is true.

The inverse ccdf~\(H_c(z)\) defined in~(\ref{h_c_def}) satisfies
\begin{equation}\label{h_c_exp}
\gamma_1 z^{1/s} \leq H_c(z) \leq \gamma_2 z^{1/s}
\end{equation} for some constants~\(\gamma_1,\gamma_2 > 0\) and all~\(z\) large and so from Theorems~\ref{thm_max_cst_low}-\ref{thm_max_cst_up}, we get the deviation and expectation bounds in the Corollary statement.

To get~\(L^2-\)convergence, we use the variance bound in Theorem~\ref{thm_max_cst_up} which states that~\(var(\chi_n) \leq D n^2p H_c^2(np)\) for constant~\(D > 0,\) together with the expectation lower bound derived above that gives~\(\mathbb{E}\chi_n \geq D_1 n H_c(np),\) again for some constant~\(D_1 > 0.\) Combining these, we get that~\[\mathbb{E}\left(\frac{\chi_n}{\mathbb{E}\chi_n}-1\right)^2 \leq D_2 p = o(1)\] for some constant~\(D_2 > 0\) and this completes the proof of part~\((a)\) of the Corollary.~\(\qed\)

\emph{Proof of Corollary~\ref{cor_example_two}\((b)\)}: We verify that the conditions in Theorem~\ref{thm_spat_mast} hold and first consider the case~\( \frac{2}{s-1} < \alpha < 1.\) In this case~\(s > 3\) and so the edge weights have bounded second moments, implying that the conditions in~(\ref{edge_vt_ax}) are satisfied. Moreover~\(p = \frac{1}{n^{\beta}}, 0 <\beta < 1,\) ensures that the lower deviation bound in~(\ref{dev_bds_mast_spat}) holds with high probability.

In fact the condition~\(\frac{2}{s-1} < \alpha\) states that~\(s > \frac{2}{\alpha}+1\) and so the edge weights have bounded~\(r^{th}\) moment for any~\(s > r > \frac{2}{\alpha}+1:\) Indeed \[\mathbb{P}\left(W(u,v) > x\right) \leq \frac{D}{x^{s}}\] for some constant~\(D > 0\) and all~\(x > x_0\) large and so \begin{align}
\mathbb{E}W^{r}(u,v) &\leq x_0^{r} +   \int_{x_0}^{\infty} x^{r-1} F_c(x) dx \nonumber\\
&\leq x_0^{r} +\int_{x_0}^{\infty} \frac{D}{x^{s-r+1}} dx
\end{align}
which is finite, since~\(r < s\) strictly. Thus the upper deviation and expectation bounds in~(\ref{dev_bds_mast_spat_up}) hold as well.  Combining with the discussion in the above paragraph, we get the deviation and expectation bounds for~\(\tau_n\) in Corollary statement.

Finally, if~\(\alpha < \frac{2}{3},\) then Theorem~\ref{thm_spat_mast} implies that~\(var(\chi_n) \leq D_1 n^2p \cdot (np)^{\alpha}\) for some constant~\(D_1 > 0\) and combining this with the lower expectation bound for~\(\chi_n\) derived above, implies that \[\mathbb{E}\left(\frac{\chi_n}{\mathbb{E}\chi_n} - 1\right)^2 \leq D_2p = o(1),\] for some constant~\(D_2 > 0.\) This obtains the~\(L^2-\)convergence of~\(\tau_n,\) appropriately scaled and centred, and  therefore completes the proof of the Corollary.~\(\qed\)

\renewcommand{\theequation}{\thesection.\arabic{equation}}
\setcounter{equation}{0}
\section{Proof of Theorem~\ref{thm_min_cst_weak}}\label{sec_pf_min_weak}
We begin with a generic result for the MST cost lower bound in terms of the edge \emph{cost} inverse ccdf~\(J(.).\)
\begin{lemma}\label{lemma_mst_low} Suppose the condition~(\ref{p_cond_new}) in the statement of Theorem~\ref{thm_max_cst_low} holds with~\(0 < p = p(n) < 1.\) There are constants~\(\theta_1,\theta_2 > 0\) such that if~\(p \geq \frac{\theta_1 \log{n}}{n},\) then
\begin{equation}\label{mst_low_ax}
\mathbb{P}\left(\tau_n \geq \theta_2 n J\left(\frac{\theta_2}{np}\right)\right) \geq 1-\theta_1 \cdot p,
\end{equation}
where~\(J(.)\) is the inverse cost cdf as defined in~(\ref{h_def}).
\end{lemma}
HEE1??

\emph{Proof of Lemma~\ref{lemma_mst_low}}: For~\(\varepsilon > 0,\) say that an edge~\(h = (u,v)\) is \emph{bad} if its cost~\[c(h) \leq \frac{1}{2} J\left(\frac{\varepsilon}{np}\right),\] where~\(J(.)\) is the inverse ccdf defined in~(\ref{h_def}). Our strategy to obtain the lower deviation bound for~\(\tau_n\) is as follows: We demonstrate that if condition~\((I)\) in Theorem statement occurs, then with high probability, i.e., with probability~\(1-o(1),\) the random graph~\(G\) is connected and contains~\(O(\varepsilon n)\) bad edges. This would then imply that the total cost of the edges in \emph{any} spanning tree of~\(G\) is at least
\[(n-1- O(\varepsilon n)) \cdot \frac{1}{2} J\left(\frac{\varepsilon}{np}\right),\] completing the proof of the Lemma. Details follow.

Letting
\begin{equation}\label{n_bad_tot}
N_{bad} := \sum_{h \in G} \ind(h \text{ is bad})
\end{equation}
be the total number of edges in~\(G,\) we begin by estimating~\(N_{bad}.\) By the definition of the inverse cost ccdf~\(J(.)\) in~(\ref{h_def}), we see that
\begin{equation}\label{bad_edge}
\mathbb{P}\left(h \text{ is bad}\right) \leq \frac{\varepsilon}{np}.
\end{equation}

Let~\(E_{deg} := E_{deg}(\emptyset)\) be the event defined in the proof of Lemma~\ref{lemma_deg}\((a)\) in Section~\ref{sec_prelim} that ensures that each vertex in~\(G\) has degree at most~\(2B_0 np\) where~\(B_0 > 0\) is the constant in the condition~(\ref{p_cond_max_ax}). As in the proof of Theorem~\ref{thm_max_cst_low}, condition~\((I)\) in Theorem statement implies that~(\ref{p_cond_max_ax}) holds and so the estimate~(\ref{e_deg_def_ax2}) implies that there is a constant~\(D > 0\) such that
\begin{equation}\label{e_deg_ax_ax}
\mathbb{P}(E_{deg}) \geq 1-\exp(-Dnp).
\end{equation}

Henceforth we let~\(\omega \in E_{deg}\) be any realization and let \[\mathbb{P}_{\omega}(.) := \mathbb{P}\left(. \mid \omega\right)\] be the probability distribution conditioned on the realization~\(\omega.\) By definition, the degree of each vertex in the random graph~\(G = G(\omega)\) is no more than~\(2B_0np\) and so by the standard handshaking argument, the total number of edges in~\(G\) is at most~\(B_0n^2p.\) Consequently, we get from~(\ref{bad_edge}) that
\begin{equation}\label{n_bad_est}
\mathbb{E}_{\omega}N_{bad} \leq B_0 n^2p \cdot \frac{\varepsilon}{np} = \varepsilon B_0 n.
\end{equation}

To obtain a high probability estimate for~\(N_{bad},\) we also estimate its variance. Indeed, letting~\(A_h\) denote the event that the edge~\(h\) of the complete graph~\(K_n\) is bad, we get that~\(N_{bad} = \sum_{h \in G} \ind(A_h).\) Consequently,
\begin{align}
var_{\omega}(N_{bad}) &:= \mathbb{E}_{\omega} N^2_{bad} -\left(\mathbb{E}_{\omega}N_{bad}\right)^2 \nonumber\\
&= \sum_{h \in G} I_1(h) + \sum_{h_1 \neq h_2 \in G} I_2(h_1,h_2), \label{var_n_bad_expr}
\end{align}
where
\[I_1(h) := \mathbb{P}_{\omega}(A_h) - \mathbb{P}^2_{\omega}(A_h)\]
and
\[I_2(h_1,h_2) := \mathbb{P}_{\omega}(A_{h_1} \cap A_{h_2}) - \mathbb{P}_{\omega}(A_{h_1}) \mathbb{P}_{\omega}(A_{h_2}).\]

Clearly,~\(I_1(h) \leq \mathbb{P}_{\omega}(A_h)\) and so the first term in~(\ref{var_n_bad_expr}) satisfies
\begin{equation}\label{first_term_est}
I_1 := \sum_{h \in G} I_1(h) \leq \sum_{h \in G} \mathbb{P}_{\omega}(A_h) = \mathbb{E}_{\omega} N_{bad}.
\end{equation}
To evaluate~\(I_2(h_1,h_2),\) we see that the events~\(A_{h_1}\) and~\(A_{h_2}\) are \emph{independent} if the edges~\(h_1\) and~\(h_2\) do not share an endvertex. Letting~\(h_1 \sim h_2\) denote that~\(h_1\) and~\(h_2\) share an endvertex, the second term in~(\ref{var_n_bad_expr}) is therefore rewritten as
\begin{align}
I_2 &:= \sum_{h_1 \in G} \sum_{h_2 \sim h_1} I_2(h_1,h_2) \nonumber\\
&\leq \sum_{h_1 \in G} \sum_{h_2 \sim h_1} \mathbb{P}_{\omega}(A_{h_1} \cap A_{h_2}) \nonumber\\
&\leq \sum_{h_1 \in G} \sum_{h_2 \sim h_1} \mathbb{P}_{\omega}(A_{h_1}). \nonumber
\end{align}
Since the degree of each vertex in~\(G\) is at most~\(2B_0np,\) we see that there are at most~\(4B_0np\) edges in~\(G\) that share an endvertex with~\(h_1\) and so
\begin{align}
I_2 &\leq 4B_0np\sum_{h_1 \in G} \mathbb{P}_{\omega}(A_{h_1}) \nonumber\\
&= 4B_0np \mathbb{E}_{\omega} N_{bad}. \label{i_2_est}
\end{align}

Plugging~(\ref{i_2_est}) and~(\ref{first_term_est}) into the variance expression~(\ref{var_n_bad_expr}), we get that
\begin{equation}
var_{\omega}(N_{bad}) \leq \mathbb{E}_{\omega}(N_{bad}) \left(1+4B_0np\right) \leq 5B_0np \mathbb{E}_{\omega}N_{bad}. \nonumber
\end{equation}
Applying the Chebychev inequality, we then get for~\(t > 0\) that
\begin{align}
\mathbb{P}_{\omega}\left(|N_{bad} - \mathbb{E}_{\omega} N_{bad}| \geq t\right) &\leq \frac{var_{\omega}(N_{bad})}{t^2} \nonumber\\
&\leq \frac{5B_0np}{t^2} \cdot \mathbb{E}_{\omega}(N_{bad}) \label{jalp_ax}
\end{align}
Setting~\(t= \varepsilon B_0 n\) and recalling the estimate~(\ref{n_bad_est}) for~\(\mathbb{E}_{\omega} N_{bad}\) we then get that
\[\mathbb{P}_{\omega}\left(N_{bad} \geq 2\varepsilon B_0n\right) \leq \frac{5p}{\varepsilon B_0}\]
for all realizations~\(\omega \in E_{deg}.\)

Averaging over the realizations in~\(E_{deg}\) and using the estimate~(\ref{e_deg_ax_ax}), we then get
\begin{align}
\mathbb{P}\left(N_{bad} \leq 2\varepsilon B_0 n\right) &\geq \left(1-\frac{5p}{\varepsilon B_0}\right)\left(1-\exp\left(-Dnp\right)\right) \nonumber\\
&\geq 1-\frac{5p}{\varepsilon B_0} - e^{-Dnp}. \label{n_bad_est_two}
\end{align}
We recall that the estimate~(\ref{e_con_est_max}) in Lemma~\ref{lemma_deg} of Section~\ref{sec_prelim}  obtains bounds for the event~\(E_{con}\) that~\(G\) is connected and implies that
\begin{equation}\label{e_con_est_new_ax}
\mathbb{P}(E_{con}) \geq 1-\exp\left(-Cnp\right)
\end{equation}
for some constant~\(C > 0\) and so the union bound implies that
\begin{equation} \nonumber
\mathbb{P}\left(E_{con} \bigcap \left\{N_{bad} \leq 2\varepsilon B_0 n\right\}\right) \geq 1-\frac{5p}{\varepsilon B_0} - e^{-Dnp} - e^{-Cnp}
\end{equation}
Recalling that~\(p \geq \frac{M\log{n}}{n},\) we choose the constant~\(M > 0\) larger if necessary and ensure that
\begin{equation}\label{e_con_joint}
\mathbb{P}\left(E_{con} \bigcap \left\{N_{bad} \leq 2\varepsilon B_0 n\right\}\right) \geq 1-D_1 p
\end{equation}
for some constant~\(D_1 > 0\) and all~\(n\) large.

If~\(E_{con} \cap \left\{N_{bad} \leq 2\varepsilon B_0n \right\}\) occurs, then the random graph~\(G\) is connected and any spanning tree~\({\cal T}\) of~\(G\) has~\(n-1\) edges. Since at most~\(2\varepsilon B_0 n\) of the edges in~\({\cal T}\) have cost less than~\(\frac{1}{2}J\left(\frac{\varepsilon}{np}\right),\) we get that the total cost of the edges in~\({\cal T}\) is at least
\begin{equation}\nonumber
\left(n-1-2\varepsilon B_0n\right) \cdot \frac{1}{2} J\left(\frac{\varepsilon}{np}\right) \geq \frac{n}{4} J\left(\frac{\varepsilon}{np}\right),
\end{equation}
provided we fix~\(\varepsilon > 0\) is small enough. Combining this with the estimate~(\ref{e_con_joint}), we get the desired bound~(\ref{mst_low_ax}) in Lemma statement. This completes the proof of the Lemma.~\(\qed\)

\emph{Proof of Theorem~\ref{thm_min_cst_weak}}: We first obtain the lower and upper deviation bounds for~\(\tau_n\) in that order below and then argue that the expectation bounds follow as a direct consequence. From Lemma~\ref{lemma_mst_low}, we already have a lower bound for the MST cost~\(\tau_n\) in terms of the inverse edge cost cdf~\(J(.).\) We now use the domination condition~\((II)\) to relate~\(J(.)\) with the inverse edge \emph{weight} cdf~\(H(.).\) Indeed, letting~\(c_1,c_2 > 0\) be the constants in~(\ref{sandwich_cond}), we see that  if~\(0 < c_1z < 1\) and~\(z < 1,\) then \[\mathbb{P}\left(W(h) \leq \frac{2J(c_1z)}{c_2}\right) \geq c_1\mathbb{P}\left(c(h) \leq 2J(c_1z)\right) \geq \frac{1}{z}\] by the definition of~\(J(.)\) in~(\ref{h_def})
and so \[H(z) \leq \frac{2J(c_1z)}{c_2} \text{ or equivalently } J(z) \geq \frac{c_2}{2}H\left(\frac{z}{c_1}\right),\] for all~\(0 < z< \frac{1}{\max(c_1,1)}.\) Plugging this into~(\ref{mst_low_ax}) (this is valid since~\(np \rightarrow \infty\) by theorem statement),  we get the lower deviation bound for~\(\tau_n\) in~(\ref{mn_comp_bounds}). The expectation lower bound for~\(\tau_n\) is a direct consequence of the deviation lower bound and the fact that~\(p  = o(1).\)

For the upper deviation bound for~\(\tau_n,\) we perform some additional computations. Say that an edge~\(h = (u,v)\) of~\(K_n\) is \emph{effective} if its cost factor~\(r(X_u,X_v) \leq K\) where~\(K \geq 1\) is a constant to be determined later and let~\(G_{eff} \subset G\) be the subset of the random graph~\(G\) obtained by retaining all effective edges. Also let~\(\Gamma_{eff} \subset K_n\) be the set of all effective edges present in the complete graph~\(K_n.\) We estimate the connectivity of~\(G_{eff} \subset \Gamma_{eff}\) as follows. Given~\(X_u = x\) and~\(\varepsilon  >0,\) the condition~\((I)\) in Theorem statement together with the Markov inequality implies that if~\(K \geq 1\) is large enough constant
\begin{equation}\label{k_choice}
\mathbb{P}\left(r(x,X_v) \geq K\right) \leq  \frac{B}{K} \leq \varepsilon,
\end{equation}
irrespective of the value of~\(x.\) Therefore if~\(N_{eff}(u)\) is the number of edges~\(K_n\) containing~\(u\) as an endvertex  that are \emph{not} effective, then given~\(X_u=x,\) we get that~\(N_{eff}(u)\) is stochastically dominated from above by a Binomial random variable with parameters~\(n-1\) and~\(\varepsilon.\) Therefore, the deviation estimate~(\ref{conc_est_f}) implies that \[\mathbb{P}\left(N_{eff}(u) \geq 2\varepsilon n \mid X_u=x\right) \leq \exp\left(-Cn\right)\]
for some constant~\(C > 0\) not depending on the choice of~\(x.\) Setting
\[E_{eff} := \bigcap_{u=1}^{n} \left\{N_{eff}(u) \leq 2\varepsilon n\right\}\] and invoking the union bound, we get that
\begin{equation}\label{e_eff_est}
\mathbb{P}\left(E_{eff}\right) \geq 1-ne^{-Cn}.
\end{equation}

Suppose~\(\Gamma\) is a realization of~\(\Gamma_{eff}\) that satisfies the occurrence and let~\(\mathbb{P}_{\Gamma}(.) = \mathbb{P}( . \mid \Gamma_{eff}=\Gamma)\) be the distribution conditioned on the occurrence of~\(\Gamma.\) We set
\[q(u,v)
= \left\{
\begin{array}{ll}
p(u,v), &  \text{ if } (u,v) \in \Gamma \\
&\\
0, & \text{ otherwise}.
\end{array}
\right.
\]
With the above notations, we see that each edge of~\(K_n\) is independently present in the random graph~\(G_{eff}\) with probability~\(q(u,v).\)

Say that an edge~\(h = (u,v)\) of~\(K_n\) is \emph{light} if its weight
\begin{equation}\label{light_def}
W(h) \leq 2H\left(\frac{\lambda \log{n}}{np}\right) = 2\varphi_n,
\end{equation}
where~\(\lambda > 0\) is a constant to be determined later. By the definition of inverse cdf~\(H(.),\) we see that \[\mathbb{P}\left( h \text{ is light}\right) \geq \frac{\lambda \log{n}}{np}.\] We already know that the edge~\(h\) is independently present in~\(G_{eff}\) with probability~\(q(u,v)\) and so if~\(G_{light} \subset G_{eff}\) is the subgraph obtained by retaining all light edges of~\(G,\) then~\(h\) is present in~\(G_{light}\) with probability
\begin{equation}\label{thmisa}
q_l(h) = q_l(u,v) := \frac{\lambda \log{n}}{n} \cdot \frac{q(u,v)}{p}.
\end{equation}

Let~\(0 < \gamma_0 < \frac{1}{2}\) be the constant in the  condition~(\ref{p_cond_new}) stated in Theorem~\ref{thm_max_cst_up} and let~\({\cal S}\) is any set of~\(s \geq (\gamma_0 + 3\varepsilon) n\) vertices. Since the original edge probabilities~\(\{p(u,v)\}\) satisfy~(\ref{p_cond_new}) and each vertex is adjacent to at most~\(2\varepsilon n\) edges that are not effective, we get from~(\ref{p_cond_new}) and~(\ref{thmisa}) that
\begin{equation}\label{tulma}
\sum_{u \in {\cal S}} q(u,v) \geq \gamma_0 nq,
\end{equation}
where~\( q := \frac{\lambda \log{n}}{n}.\) Consequently, letting~\(E_{con, light}\) denote the event that~\(G_{light}\) is connected, we get from the connectivity estimate~(\ref{e_con_est_max}) derived in Lemma~\ref{lemma_deg}\((b)\)  that
\begin{align}
\mathbb{P}_{\Gamma}\left(E_{con, light}\right) &\geq 1-\exp\left(-Cnq\right) \nonumber\\
&\geq 1-\exp\left(-C\lambda \log{n}\right) \label{lils_tits}
\end{align}
for some constant~\(C > 0.\)  Here and henceforth, constants do not depend on the choice of~\(\Gamma\) or~\(\lambda.\)

Similarly, arguing as above, the expected degree of each vertex in~\(G_{light}\) given~\(\Gamma\) is at least~\(2D_0 nq = 2D_0 \lambda \log{n}\) for some constant~\(D_0 > 0.\)  Letting~\(E_{deg,light}\) denote the event that each vertex is adjacent to at least~\(D_0 \lambda \log{n}\) other vertices in~\(G_{light},\) we get that
\begin{equation}
\mathbb{P}_{\Gamma}\left(E_{deg, light}\right) \geq 1-\exp\left(-C\lambda \log{n}\right).  \label{lils_tits_2}
\end{equation}
Defining
\[E_{light} := E_{con, light} \cap E_{deg, light},\] we get from~(\ref{lils_tits_2}),~(\ref{lils_tits}) and the union bound that
\[ \mathbb{P}_{\Gamma}\left(E_{light}\right) \geq 1-2\exp\left(-C\lambda \log{n}\right). \]
Given~\(\gamma > 0\) we now choose the constant~\(\lambda > 0\) large enough so that
\[\mathbb{P}_{\Gamma}\left(E_{light}\right) \geq 1-\frac{1}{n^{9+2\gamma}}\]
and then average over all~\(\Gamma\) satisfying the occurrence of~\(E_{eff}\) and use the estimate~(\ref{e_eff_est}) to get that
\begin{align}\label{g_light_con}
\mathbb{P}(E_{light}) &\geq \mathbb{P}\left(E_{light} \cap E_{eff}\right) \nonumber\\
&\geq \left(1-\frac{1}{n^{9+2\gamma}}\right) \cdot \left(1-n \cdot e^{-Cn}\right) \nonumber\\
&\geq 1-\frac{2}{n^{9+2\gamma}},
\end{align}
for all~\(n\) large.

%The connectivity estimate~(\ref{e_con_est_max}) in Lemma~\ref{lemma_deg} therefore implies that the random subgraph~\(G_{mk} \subset G\) obtained by retaining all effective edges of~\(G\) is connected with probability at least~\(1-\exp\left(-Dnp\right)\) for some constant~\(D > 0\) not depending on the choice of~\(\omega.\) In other words, defining~\(E_{mk}\) to be the event that~\(G_{mk}\) is connected, we have that
%\begin{equation}\label{e_mk_con_est_ax}
%\mathbb{P}_{\omega}(E_{mk}) \geq 1- e^{-Dnp}.
%\end{equation}

If~\(G_{light}\) is connected, then it contains a spanning tree~\({\cal T}_{light}\) each of whose edge has weight at most~\(H(q).\) Since~\(G_{light}\) contains only effective edges, the cost of each edge in~\(G_{light}\) is at most~\(K \cdot H(q),\) where~\(K \geq 1\) is the constant in~(\ref{k_choice}). Therefore the total cost of the~\(n-1\) edges in~\({\cal T}_{light}\) is at most~\[K \cdot (n-1) \cdot H(q) \leq K \cdot n \cdot H(q).\] The relation~(\ref{g_light_con}) then obtains the desired upper deviation bound for~\(\tau_n\) in Theorem statement.

To derive the expectation upper bound for~\(\tau_n,\) we also consider the case when~\(G_{light}\) is not connected. In this case, the MST cost~\(\tau_n\) is upper bounded by the total cost of all the edges of~\(K_n\) and so
\begin{equation}\label{tau_Two}
\tau_n\ind(E^c_{light}) \leq \sum_{h \in K_n} c(h) \ind(E^c_{light}).
\end{equation}
By the Cauchy-Schwartz inequality, we have that
\[\mathbb{E}c(h) \ind(E^c_{light}) \leq \left(\mathbb{E}c^2(h)\right)^{\frac{1}{2}} \cdot \left(\mathbb{P}(E^c_{light})\right)^{\frac{1}{2}}\]
and from the cost factor and edge weight moment condition~\((I)\) in Theorem statement, we know that
\[\mathbb{E}c^2(h) = \mathbb{E}r^2(X_1,X_2) \mathbb{E}W^2(f) \leq D.\]
for some constant~\(D > 0.\) Plugging this into~(\ref{tau_Two}) and using the fact that there are at most~\(n^2\) edges of~\(K_n,\) we get that
\begin{equation}\label{thmulp}
\mathbb{E}\tau_n \ind(E^c_{light}) \leq  D_1 n^2  \cdot \frac{1}{n^{3+\gamma}}  \leq \frac{D_1}{n^{1+\gamma}}
\end{equation}
for all~\(n\) large and some constant~\(D_1 > 0.\)

If~\(E_{light}\) does occur, i.e., if~\(G_{light}\) is connected, then as discussed prior to~(\ref{tau_Two}), we know that~\(\tau_n \leq KnH(q).\) Consequently, we get from~(\ref{thmulp}) that
\begin{equation}\label{skedoosh}
\mathbb{E}\tau_n \leq KnH(q) + \frac{D_1}{n^{1+\gamma}}.
\end{equation}
Since~\(\gamma > 0\) is arbitrary,  this obtains the expectation upper bound for~\(\tau_n\) in Theorem statement.

For the variance bound, we use Efron-Stein inequality and begin with some preliminary computations.  Analogous to~\(E_{light},\) define the event~\(E_{light}(\{1,2\})\) for the graph~\(G(\{1,2\})\) obtained by removing the vertices~\(1\) and~\(2\) from~\(G.\) Arguing as in the derivation of~(\ref{g_light_con}), we see that~(\ref{g_light_con}) is satisfied by~\(G(\{1,2\})\) as well.  Recalling the event~\(E_{deg} := E_{deg}(\emptyset)\) defined prior~(\ref{e_deg_ax_ax}), we see that occurrence of~\(E_{deg}\) ensures that each vertex in~\(G\) has degree at most~\(2B_0 np\) where~\(B_0 > 0\) is the constant in the condition~(\ref{p_cond_max_ax}).  Finally, recalling that~\(E_{con}({\cal B})\) denotes the event that the graph~\(G({\cal B})\) (obtained by removing the vertices of the deterministic set~\({\cal B}\) from~\(G\)) is connected, we define the joint event \[E_{comb} := E_{con}(\emptyset) \bigcap E_{con}(\{1,2\}) \bigcap  E_{deg} \bigcap E_{light} \bigcap E_{light}(\{1,2\})\] and get from the corresponding estimates~(\ref{e_con_est_new_ax}),~(\ref{e_deg_ax_ax}) and~(\ref{g_light_con}) that
\[\mathbb{P}\left(E_{comb}\right) \geq 1-e^{-Cnp} - e^{-Dnp} - \frac{2}{n^{9+2\gamma}}.\] Since~\(np \geq M \log{n}\) by Theorem statement, given~\(\gamma > 0,\) we choose~\(M > 0\) large enough so that
\begin{equation}\label{e_comb_est_ax}
\mathbb{P}(E_{comb}) \geq 1-\frac{1}{n^{8+2\gamma}}
\end{equation}
for all~\(n\) large.

We now use the event~\(E_{comb}\)  together with the martingale difference method based on the Efron-Stein inequality to obtain the desired variance bound for~\(\tau_n.\) As in the derivation of~(\ref{i_loc_wt_est_max}), we have that
\begin{align}
var(\tau_n) \leq  4nI_{loc} + 4mI_{wt}, \label{i_loc_wt_est}
\end{align}
where \[I_{loc} := \mathbb{E}\left(\tau_n-\tau_{rem}(1)\right)^2 \text{ and } I_{wt} := \mathbb{E}\left(\tau_n-\tau_{mod}(f_1)\right)^2\] respectively denote the scaled contributions due to randomness in vertex locations and edge states/weights, respectively. As in~(\ref{i_loc_wt_est_max}), we let that~\(G_{rem}(j) \subset G\) be the random graph obtained after removing vertex~\(j, 1 \leq j \leq n,\) from~\(G\) and denote~\({\cal T}_{rem}(1)\) to the minimum cost spanning tree of the largest component of~\(G,\) with corresponding cost~\(\tau_{rem}(1).\) Similarly~\(G_{rem}(f_k) \subset G\) is the random graph obtained after removing edge~\(f_k, 1 \leq k \leq m = {n \choose 2}\) from~\(G\) and~\({\cal T}_{rem}(f_k)\) is the minimum cost spanning tree of the largest component of~\(G_{rem}(f_k)\) with corresponding cost~\(\tau_{rem}(f_k).\)

In what follows, we estimate~\(I_{loc}\) and~\(I_{wt}\) in that order below.\\
\emph{\underline{Step 1} (Estimate for~\(I_{loc}\))}: Recalling the event~\(E_{comb}\) defined prior to~(\ref{e_comb_est_ax}), we split
\begin{equation}\label{i_loc_split}
I_{loc} = I_{loc,1} + I_{loc,2}
\end{equation}
where
\[I_{loc,1} := \mathbb{E}\left(\tau_n-\tau_{rem}(1)\right)^2\ind(E_{comb}) \text{ and } I_{loc,2} := \mathbb{E}\left(\tau_n-\tau_{rem}(1)\right)^2\ind(E^c_{comb})\] and estimate~\(I_{loc,2}\) and~\(I_{loc,1}\) in that order.

To bound~\(I_{loc,2},\) we  use the fact that~\(\tau_n\) is no more than the total cost of all edges of the complete graph~\(K_n;\) i.e.,~\[\tau_n \leq \sum_{h \in K_n} c(h).\] The same bound holds for~\(\tau_{rem}(1)\) as well and so
\[\left(\tau_n-\tau_{rem}(1)\right)^2 \leq \left(\sum_{h \in K_n}c(h)\right)^2.\] Applying Cauchy-Schwarz inequality, we obtain that
\begin{equation}\label{thmix_ax}
I_{loc,2} \leq \left(\mathbb{E}\left(\sum_{h \in K_n} c(h)\right)^{2}\right)^{\frac{1}{2}} \cdot \left(\mathbb{P}(E^c_{comb})\right)^{\frac{1}{2}}.
\end{equation}
Using~\((\sum_{i=1}^{l}a_i)^2 \leq l \sum_{i=1}^{l}a_i^2\) for positive~\(\{a_i\}\) and the fact that there are~\({n \choose 2} \leq n^2\) edges in~\(K_n,\) we get that
\begin{equation}\label{sum_ch}
\mathbb{E}\left(\sum_{ h \in K_n} c(h)\right)^2 \leq n^2 \sum_{h \in K_n} \mathbb{E}c^2(h)
\end{equation}
and for any edge~\(h = (u,v)\) with endvertices~\(u\) and~\(v,\) we have that
\[\mathbb{E}c^2(h) = \mathbb{E}r^2(X_u,X_v) W^2(u,v) = \mathbb{E}r^2(X_u,X_v) \mathbb{E}W^2(u,v).\] Since the cost factors and the edge weights have bounded second moments by conditions~\((I)-(II)\) in Theorem statement, we get that~\(\mathbb{E}c^2(h) \leq D^2\) for some constant~\(D > 0.\) Plugging this into~(\ref{sum_ch}) and again using the fact that there are at most~\(n^2\) edges in~\(K_n,\) we get that the final term in~(\ref{sum_ch}) is at most~\(D^2n^4.\) Substituting this into~(\ref{thmix_ax}) we get
\begin{equation}\label{thmix_ax_2}
I_{loc,2} \leq D n^2 \cdot \left(\mathbb{P}(E^c_{comb})\right)^{\frac{1}{2}}
\end{equation}
and plugging the estimate~(\ref{e_comb_est_ax}) for~\(E_{comb}\) into~(\ref{thmix_ax_2}), we get that
\begin{equation}\label{i_loc_two_est_ax}
I_{loc,2} \leq D_1 n^2 \cdot \frac{2}{n^{4+\gamma}}  = \frac{2D_1}{n^{2+\gamma}}
\end{equation}
for some constant~\(D_1 > 0.\) This obtains an estimate for~\(I_{loc,2}.\)

We now evaluate~\(I_{loc,1}\) and therefore assume henceforth that~\(E_{comb}\) occurs. We recall that~\(G_{rem}(1) = G(\{1\}) \subset G\) is obtained after removing the vertex~\(1\) from~\(G\) and because~\(E_{con}(\emptyset) \cap E_{con}(\{1\}) \supset E_{join}\)  occurs, both~\(G\) and~\(G_{rem}(1)\) are connected. We let~\({\cal T}_n\) and~\({\cal T}_{rem}(1)\) be the respective minimum cost spanning trees of~\(G\) and~\(G_{rem}(1).\) To estimate the cost difference~\(\tau_n-\tau_{rem}(1),\) we let~\({\cal N}_{light}(1)\) is the set of neighbours of the vertex~\(1\) in~\(G_{light}\) and let~\(j_0 \in {\cal N}_{light}(1)\) be any vertex. This is valid since the event~\(E_{deg, light} \supset E_{light} \supset E_{comb}\) ensures that each vertex is adjacent to at least order of~\(\log{n}\) other vertices in~\(G_{light}.\)

Adding the edge~\((1,j_0)\) to~\({\cal T}_{rem}(1)\) gives a spanning tree of~\(G\) and the edge~\((1,j_0)\) has weight at most~\(2\varphi_n\) by~(\ref{light_def}) and cost factor at most~\(K,\) by definition. Therefore
\begin{equation}
\tau_n \leq \tau_{rem}(1) + 2K\varphi_n. \label{zendaya}
\end{equation}
This obtains an upper bound for~\(\tau_n\) in terms of~\(\tau_{rem}(1).\)

For obtaining a lower bound for~\(\tau_n\) in terms of~\(\tau_{rem}(1),\) we let \[{\cal Q} := \{v_1,\ldots, v_{\Delta}\}\] be the set of neighbours of the vertex~\(1\) in the MST~\({\cal T}_n \subset G.\) Removing the vertex~\(1,\) we obtain~\(\Delta\) trees~\(\{{\cal S}_l\}_{1 \leq l \leq \Delta}\) of~\(G_{rem}(1),\) such that~\({\cal S}_l\) contains~\(v_l\) as the root. This is illustrated in Figure~\ref{fig_sub_trees_ax}\((a)\) for the case~\(\Delta=3,\) where the solid triangles represent the trees~\({\cal S}_i,i=1,2,3.\) Because the graph~\(G_{light}(\{1\})\) obtained by removing vertex~\(1\) from~\(G_{light}\) is connected (see definition of~\(E_{comb}\) prior to~(\ref{e_comb_est_ax})) there are edges~\(\{h_1,\ldots,h_{\Delta-1}\} \in G_{light}(\{1\})\) such that the union~\[\bigcup_{l=1}^{\Delta} \{{\cal S}_l\} \bigcup \bigcup_{l=1}^{\Delta-1} \{h_l\}\] forms a spanning tree of~\(G_{light}(1)\) (and hence~\(G_{rem}(\{1\})\)). This is illustrated in Figure~\ref{fig_sub_trees_ax}\((b),\) where the dotted lines represent the edges~\(\{h_l\}_{1 \leq l \leq \Delta-1}.\)

\begin{figure}[tbp]
\centering
%\fbox{
\includegraphics[width=6in, trim= 220 200 50 110, clip=true]{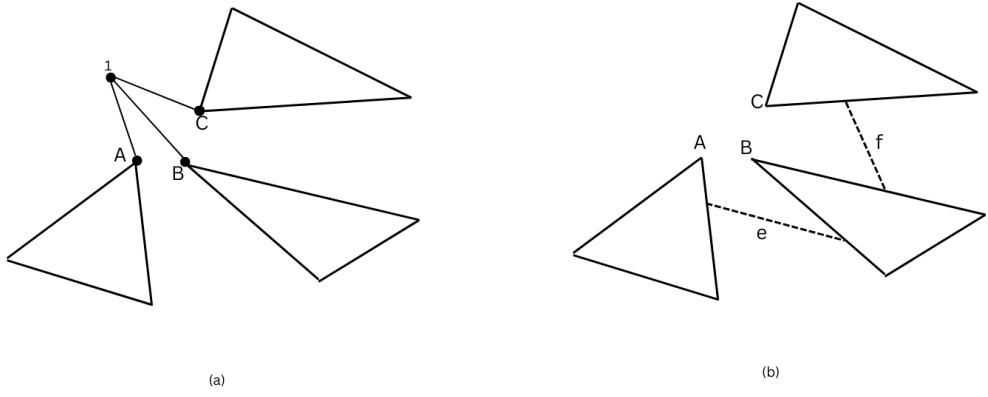}
%}
\caption{The subtrees~\({\cal S}_i, 1 \leq i \leq t=3\) containing the neighbours~\(v_1=A,v_2=B,v_3=C\) of the vertex~\(1\) are shown in~\((a).\) Removing vertex~\(1\) and adding the  edges~\(h_1=e\) and~\(h_2=f\) gives a spanning tree of the graph~\(G_{rem}(1),\) as shown in~\((b).\)}
\label{fig_sub_trees_ax}
\end{figure}

Since each~\(h_l\) is light and effective  and therefore has weight at most~\(2\varphi_n\) and cost factor at most~\(K,\) we get that
\begin{equation}
\tau_{rem}(1) \leq \tau_n + 2K\varphi_n (\Delta-1).  \label{tau_rem_up}
\end{equation}
Combining~(\ref{zendaya}) and~(\ref{tau_rem_up}), we obtain
\[|\tau_n-\tau_{rem}(1)|\ind(E_{comb}) \leq \ind(E_{comb})2K\varphi_n\Delta\]  and plugging this into the expression for~\(I_{loc,1}\) in~(\ref{i_loc_split}), we obtain
\[I_{loc,1} \leq 4K^2\varphi_n^2 \mathbb{E}\Delta^2\ind(E_{comb}),\]
where the event~\(E_{comb}\) is as defined prior to~(\ref{e_comb_est_ax}).

Using the fact that the event~\(E_{deg} \supset E_{comb}\) occurs, we have that the degree of every vertex in~\(G\) is at most~\(Dnp\) for some constant~\(D > 0\) and so
\[I_{loc,1} \leq D_1 \varphi_n^2 np \mathbb{E}\Delta \ind(E_{comb}) \leq D_1 \varphi_n^2 np \mathbb{E}\Delta \ind(E_{con}) = D_1 \varphi_n^2 np \mathbb{E}\Delta(1) \ind(E_{con}),\] for some constant~\(D_1 > 0,\) using the notation~\(\Delta(j)\) for the degree of the vertex~\(j\) in the MST~\({\cal T}_n \subset G.\) Invoking symmetry, we finally get
\begin{equation} \label{bharg_ax_2}
I_{loc,1} \leq D_1\varphi_n^2 np \mathbb{E}\Delta(1)\ind(E_{con})= D_1np \varphi_n^2\cdot \frac{1}{n}\sum_{j=1}^{n} \mathbb{E}\Delta(j)\ind(E_{con}),
\end{equation}
for some constant~\(D_0 > 0,\) since the event~\(E_{con} \) that~\(G\) is connected does not depend on the index~\(j\) in the summation.

If~\(G\) is connected, then the tree~\({\cal T}_n\) has~\(n-1\) edges and using the relation that the sum of vertex degrees is twice the number of edges in any graph, we get that~\[\sum_{j=1}^{n}\Delta(j) = 2(n-1) \leq 2n.\] Substituting this into~(\ref{bharg_ax_2}) we obtain~\(I_{loc,1} \leq 2D_4np\) and combining with the estimate~(\ref{i_loc_two_est_ax}) for~\(I_{loc,2},\) we see that the contribution~\(I_{loc}\) defined in~(\ref{i_loc_wt_est}) satisfies
\begin{equation}\label{i_loc_est_axxx}
nI_{loc} \leq 2D_1n^2p \varphi_n^2 + \frac{D_1}{n^{1+\gamma}}.
\end{equation}
for all~\(n\) large. This completes the first step of the derivation of the variance bound.

\emph{\underline{Step 2} (Estimate for~\(I_{wt}\))}: Recalling  that our goal is to estimate~\(I_{wt} = \mathbb{E}\left(\tau_n  -\tau_{rem}(f_1)\right)^2,\) we use the event~\(E_{comb}\) defined prior to~(\ref{e_comb_est_ax}) and split
\begin{equation}\label{i_wt_split}
I_{wt} := I_{wt,1} + I_{wt,2},
\end{equation}
where
\[I_{wt,1} := \mathbb{E}\left(\tau_n-\tau_{rem}(f_1)\right)^2\ind(E_{comb}) \text{ and } I_{wt,2} := \mathbb{E}\left(\tau_n-\tau_{rem}(f_1)\right)^2\ind(E^c_{comb})\] and estimate~\(I_{wt,2}\) and~\(I_{wt,1}\) in that order below.

If~\(E_{comb}^c\) occurs, then we argue as in the derivation of~(\ref{i_loc_two_est_ax}) to get that
\begin{equation}\label{i_wt_two_est_ax}
I_{wt,2} \leq \frac{1}{n^{3+\gamma}}
\end{equation}
for all~\(n\) large. To estimate~\(I_{wt,1},\) we assume henceforth that~\(E_{comb}\) occurs so that~\(G\) is connected and establish the connectivity of the graph~\(G_{mod}(f_1)\) as follows. We assume that the edge~\(f_1 = (1,2)\) has~\(1\) and~\(2\) as endvertices. The occurrence of the event~\(E_{comb}\) ensures that the graph~\(G(\{1,2\})\) obtained after removing both~\(1\) and~\(2\) from~\(G\) is connected. Moreover the event~\(E_{light} \supset E_{join}\) also implies that both~\(1\) and~\(2\) have at least order of~\(\log{n}\) neighbours in~\(G_{light}\) and so both~\(1\) and~\(2\) each have at least one neighbour in the graph~\(G_{light}(f_1)\) obtained by only removing the edge~\(f_1\) from~\(G_{light}.\) Thus~\(G_{light}(f_1)\) and hence~\(G_{rem}(f_1)\) are both connected and let~\(b_i, i=1,2\) be a  neighbour of~\(i\) in~\(G_{light}(f_1).\)

The discussion above implies that both~\(G\) and~\(G_{rem}(f_1)\) are connected with respective MSTs~\({\cal T}_n\) and~\({\cal T}_{rem}(f_1),\) having corresponding costs~\(\tau_n\) and~\(\tau_{rem}(f_1)\geq \tau_n.\) Therefore removing~\(f_1\) from the MST~\({\cal T}_n\) and adding the effective and light edges~\((1,b_1)\) and~\((2,b_2),\) we get a new spanning tree~\({\cal T}_{rep}\) of~\(G_{rem}(f_1).\) Thus~\(\tau_{rem}(1) \leq \tau_n + 2 \cdot 2\varphi_n \cdot K.\) Summarizing the above, we get \[\tau_n \leq \tau_{rem}(f_1) \leq \tau_n + 4K\varphi_n.\]

Clearly,~\(\tau_n = \tau_{rem}(f_1)\) if~\(f_1\) does not belong to MST~\({\cal T}_n\)  and so the discussion in the previous paragraph implies that
\[|\tau_n-\tau_{rem}(f_1)|\ind(E_{comb}) \leq 4K\varphi_n\ind(f_1 \in {\cal T}_n). \]
Squaring and taking expectations, we get
\begin{equation}
I_{wt,1} = \mathbb{E}\left(\tau_n-\tau_{rem}(f_1)\right)^2\ind(E_{comb}) \leq 16K^2\varphi_n^2\mathbb{P}\left(f_1 \in {\cal T}_n\right). \label{topless_tits_two}
\end{equation}
For any edge~\(f_k \in K_n, 1 \leq k \leq m = {n \choose 2},\) we have by symmetry that \[\mathbb{P}(f_k \in {\cal T}_n) = \mathbb{P}(f_1\in {\cal T}_n)\] and so
\[\mathbb{P}\left(f_1 \in {\cal T}_n\right) = \frac{1}{m} \sum_{k=1}^{m} \mathbb{P}\left(f_k \in {\cal T}_n\right) \leq \frac{n-1}{m} = \frac{n}{m}, \] since any tree of~\(K_n\) has at most~\(n-1\) edges. Substituting this into~(\ref{topless_tits_two}), we get that~\(I_{wt,1} \leq \frac{Dn \varphi_n^2 }{m}\) for some constant~\(D > 0\) and combining  with the estimate~(\ref{i_wt_two_est_ax}) for~\(I_{wt,2},\) we obtain
\begin{align}
mI_{wt} &= mI_{wt,1} + mI_{wt,2} \nonumber\\
&\leq D_1 n \cdot \varphi_n^2 +  D_1m \cdot \frac{1}{n^{3+\gamma}} \nonumber\\
&\leq D_1n\varphi_n^2 + \frac{D_1}{n^{1+\gamma}}, \label{hilla}
\end{align}
for some constant~\(D_1 > 0,\) since the number of edges~\(m\) in~\(K_n\) is~\({n \choose 2} \leq n^2.\) Plugging this and~(\ref{i_loc_est_axxx}) into the estimate~(\ref{i_loc_wt_est}), we obtain the desired variance bound in Theorem statement. This completes the proof of the Theorem.~\(\qed\)

\renewcommand{\theequation}{\thesection.\arabic{equation}}
\setcounter{equation}{0}
\section{Proof of Theorem~\ref{thm_min_cst_strong}}\label{sec_pf_min_strong}
We begin with a preliminary lemma regarding the expected value of the minimum of i.i.d. edge weights. Formally, let~\(\{Y_j\}_{j \geq 1}\) be i.i.d.\ with cdf~\(F\) and for~\(l \geq 1,\) let \[Z_l := \min_{1 \leq j \leq l} Y_j.\] Recalling the edge weight inverse cdf~\(H\) defined in the paragraph containing~(\ref{h_def}), we let
\[t_l := H\left(\frac{1}{l}\right) \leq s_l := H\left(\frac{\lambda \log{n}}{l}\right)\] for constant~\(\lambda > 0\) and set~\[\mathbb{P}_l := \mathbb{P}\left(. \mid Z_l \leq 2s_l\right)\] to be the distribution conditioned on the event that~\(Z_l\) is at most~\(2s_l.\)
\begin{lemma}\label{lemma_z} If the condition~\((B)\) in the statement of Theorem~\ref{thm_min_cst_strong} holds, then there is a constant~\(\gamma = \gamma(\lambda) > 0\)  such that for all~\(l \geq \gamma \log{n},\) we have
\begin{equation}\label{exp_z_bounds}
\gamma^{-1} t_l \leq \mathbb{E}_lZ_l \leq \gamma t_l.
\end{equation}
\end{lemma}

\emph{Proof of Lemma~\ref{lemma_z}}:  Clearly, for~\(x > 0\) we have that
\[\mathbb{P}(Z_l  > x) = \mathbb{P}\left(\min_{1 \leq j \leq l}Y_j >x\right) = (1-F(x))^{l}\] and so \[\mathbb{P}(Z_l > 2s_l) = (1-F(2s_l))^{l}. \] By the definition of the inverse cdf~\(H(.),\) we have that~\(F(2s_l) \geq \frac{\lambda \log{n}}{l}\) and so
\begin{equation}\label{z_up}
\mathbb{P}(Z_l > 2s_l) \leq \left(1-\frac{\lambda \log{n}}{l}\right)^{l} \leq \exp\left(-\lambda \log{n}\right) = o(1).
\end{equation}
Thus
\begin{align}
\mathbb{E}_lZ_l &= \frac{1}{\mathbb{P}(Z_l \leq 2s_l)} \int_{0}^{2s_l} (1-F(x))^{l} dx  \nonumber\\
&= (1+o(1)) \int_{0}^{2s_l}(1-F(x))^{l} dx. \label{prahlad}
\end{align}

Again using the definition of the inverse cdf, we have that
\[\mathbb{P}\left(Y_1 \leq \frac{t_l}{2}\right) \leq \frac{1}{l}\] and so from~(\ref{prahlad}), we obtain
\begin{align}
\mathbb{E}Z_l &\geq \frac{1}{2}\int_{0}^{t_l/2} (1-F(x))^{l} dx \nonumber\\
&\geq \frac{1}{2}\int_{0}^{t_l/2} \left(1-\frac{1}{l}\right)^{l} dx \nonumber\\
&\geq \frac{1}{2e} \int_0^{t_l/2} dx \nonumber\\
&= \frac{t_l}{4e}. \nonumber
\end{align}
This obtains the lower bound in~(\ref{exp_z_bounds}).

For the expectation upper bound, we use the scaling relation~(\ref{scale_two_cond}). Indeed, since~\(0 < z_0 := F\left(\frac{x_0}{2}\right) < 1,\) we have for~\(l \geq \gamma \log{n}\) that~\[\frac{\lambda \log{n}}{l} \leq  \frac{\lambda}{\gamma} \leq z_0,\] provided~\(\gamma > 0\) is large enough. Fixing such a~\(\gamma,\) we use the monotonicity of the inverse cdf~\(H(z)\) to obtain that
\begin{equation}\label{s_l_ax_est}
s_l = H\left(\frac{\lambda \log{n}}{l} \right) \leq H(z_0).
\end{equation}

We now invoke the strict monotonicity of the cdf~\(F\) near~\(\frac{x_0}{2},\) to argue that
\begin{equation}\label{h_z_ntttt}
H(z_0) = \frac{x_0}{2}.
\end{equation}
Indeed, by the definition of inverse cdf in~(\ref{h_def}), we know that~\(F(H(z_0)+\epsilon) > z_0\) for all~\(\epsilon > 0\) and so allowing~\(\epsilon \downarrow 0\) and using the right continuity of~\(F,\) we get that~\[F(H(z_0)) \geq z_0 = F\left(\frac{x_0}{2}\right).\] Consequently~\(H(z_0) \geq \frac{x_0}{2}.\) If~\(H(z_0)\) were strictly larger than~\(\frac{x_0}{2},\) then there exists~\(\eta > 0\) such that~\(H(z_0)(1-\eta) > \frac{x_0}{2}\) and since~\(F\) is strictly increasing in a neighbourhood of~\(\frac{x_0}{2},\) we can choose~\(\eta > 0\) smaller if necessary to get that~\[F(H(z_0)(1-\eta)) > F\left(\frac{x_0}{2}\right) = z_0,\] strictly.  This contradicts the definition of the inverse cdf in~(\ref{h_def}) and so~(\ref{h_z_ntttt}) is true.

Combining~(\ref{s_l_ax_est}) and~(\ref{h_z_ntttt}), we get that~\(2s_l \leq x_0\) and so the scaling relation~(\ref{scale_two_cond}) is applicable for all integers~\(k\) satisfying~\(kt_l \leq s_l.\) Splitting the summation in~(\ref{prahlad}), we now get
\begin{align}
\mathbb{E}_lZ_l &= (1+o(1)) \int_0^{2s_l} (1-F(x))^{l} dx \nonumber\\
&\leq 2 \int_0^{s_l} (1-F(x))^{l} dx \nonumber\\
&\leq 2 \sum_{k \leq s_l/t_l} \int_{2kt_l}^{2(k+1)t_l} (1-F(x))^{l} dx \nonumber\\
&\leq 2 \sum_{k} \int_{2kt_l}^{2(k+1)t_l} \exp\left(-lF(x)\right) dx \nonumber\\
&\leq 2 \sum_{k} \exp\left(-lF(2kt_l)\right) 2t_l \nonumber\\
&\leq  4t_l \sum_{k} \exp\left(-l D(\log{k})^{1+\theta} F(2t_l)\right) \nonumber\\
&\leq 4t_l \sum_{k} \exp\left(-D(\log{k})^{1+\theta}\right) \label{hiranya}
\end{align}
where~\(D > 0\) is the constant in~(\ref{scale_two_cond}) and the final estimate in~(\ref{hiranya}) is true since~\[F(2t_l) = \mathbb{P}\left(Y_1 \leq 2H\left(\frac{1}{l}\right)\right)  \geq \frac{1}{l},\] by the definition of  inverse cdf~\(H(.).\) The summation in the final term of~(\ref{hiranya}) being finite, we then obtain the expectation upper bound in~(\ref{exp_z_bounds}) as well. This completes the proof of the Lemma.~\(\qed\)

\emph{Proof of Theorem~\ref{thm_min_cst_strong}}: We use  an iterative path construction involving small weight edges of~\(G\) to establish the upper bound for~\(\tau_n.\) Let~\({\cal N}(v)\) be the set of neighbours of the vertex~\(v\) in the random graph~\(G.\) Let~\(v_1= 1\) and for~\(j \geq 1\) let~\(v_{j+1} \in {\cal N}(v_j)\) be the vertex such that
\begin{equation}\label{weight_ax_est}
R_j := W(v_j, v_{j+1}) = \min_{v \in {\cal N}(v_j)\setminus \{v_1,\ldots,v_{j-1}\}} W(v_j,v)
\end{equation}
with the notation that~\(\{v_1,\ldots,v_{j-1}\} = \emptyset\) for~\(j=1\) and the minimum of an empty set is~\(\infty.\) In words~\(v_{j+1}\) forms the edge with least weight amongst all neighbours of~\(v_j\) not encountered so far. If~\(L\) is the smallest integer~\(j\) such that~\(R_j = \infty,\) then we define the path~\({\cal P} := (v_1,\ldots,v_L)\) and for completeness, set~\(v_j := v_L\) and~\(R_j := \infty\) for~\(L \leq j \leq n-1.\)

%GIVE ONLY EXISTENCE OF PATH FOR N-16J0 EDGES!!!!

Setting~\(J_0 := \frac{32\theta \log{n}}{p}\) for some constant~\(\theta > 0\) to be determined later,  we show below that~\({\cal P}\) has at least~\(n-J_0\) edges with high probability and estimate the weight of path~\({\cal P}(J_0)\) formed by the first~\(n-J_0\) edges of~\({\cal P}.\) We then extend~\({\cal P}(J_0)\) to a spanning tree of~\(G\) using edges with predetermined weights and thereby upper bound the MST weight~\(\tau_n.\)

If~\(N_j(v_j)\) is the number of neighbours of the~\(j^{th}\) vertex~\(v_j\) in~\({\cal N}(v_j) \setminus\{v_1,\ldots,v_{j-1}\},\) then~\(N_j(v_j)\) depends only on the state of edges in~\(G\) having both endvertices outside~\(\{v_1,\ldots,v_{j-1}\}.\) Moreover, given~\(\{v_1,\ldots,v_j\},\) we see that~\(N_j(v_j)\) is stochastically dominated from below by a Binomial random variable with parameters~\(n-j\) and~\(p.\) In anticipation, we define the events
\begin{equation}\label{fj_def}
E_j(v_j) := \left\{N_j(v_j) \geq \frac{(n-j)p}{2}\right\}\;\;\text{ and }\;\; F_j := \bigcap_{l=1}^{j}E_j(v_j)
\end{equation}
with  the notation that~\(F_0^c = \emptyset.\) If~\({\cal H}_{j-1}\) be the sigma-field generated by the states and weights of all edges containing at least one endvertex in~\(\{v_1,\ldots,v_{j-1}\},\) then~\(F_{j-1} \in {\cal H}_{j-1}\) and so we get from the deviation estimate~(\ref{conc_est_f}) that
\begin{equation}\label{ej_vj_est}
\mathbb{P}\left(E^c_j(v_j) \mid {\cal H}_{j-1}\right) \ind(F_{j-1}) \leq \exp\left(-\frac{(n-j)p}{16}\right) \leq \exp\left(-\frac{np}{32}\right),
\end{equation}
for all~\(1 \leq j \leq n-J_0 = n-\frac{32\theta \log{n}}{p},\) provided~\(np \geq 64\theta \log{n}.\)

%Given~\(\gamma > 0,\) we now choose~\(np \geq M\log{n}\) for some constant~\(M >16(2+\gamma),\) strictly so that~\(n \geq 16J_0,\) where~\(J_0 = \frac{(2+\gamma)\log{n}}{p}\) is as in Theorem statement.

Taking averages in~(\ref{ej_vj_est}) we get that
\[\mathbb{P}\left(F^c_j \cap F_{j-1}\right) = \mathbb{P}\left(E^c_j(v_j) \cap F_{j-1}\right) \leq  \exp\left(-\frac{np}{32}\right)\]
for all~\(1 \leq j \leq n - J_0\) and so defining \[E_{path} := F_{J_0} = \bigcap_{1 \leq j \leq n-J_0} E_j(v_j),\]
we therefore get  that
\begin{equation}\label{e_path_est}
\mathbb{P}(E^c_{path}) = \sum_{j=1}^{n-J_0}\mathbb{P}(F^c_j \cap F_{j-1}) \leq n \exp\left(-\frac{np}{32}\right).
\end{equation}
Given~\(\gamma > 0,\) we now choose~\(np \geq M\log{n}\) for some large constant~\(M > 64\theta\)  so that the final expression in~(\ref{e_path_est}) is at most~\(\frac{1}{n^{1+\gamma}}.\) Finally, recalling the event~\(E_{comb}\) defined prior to~(\ref{e_comb_est_ax}), we set
\[E_{join} := E_{comb} \cap E_{path}\] and get from the respective estimates~(\ref{e_comb_est_ax}) and~(\ref{e_path_est}) that
\begin{equation}\label{e_comb_est_min}
\mathbb{P}(E_{join}) \geq 1-\frac{2}{n^{1+\gamma}}.
\end{equation}

If~\(E_{join}\) occurs, then  the path~\({\cal P}\) in the above iterative construction contains at least~\(n- J_0\) edges and from~(\ref{weight_ax_est}), we know that the weight of the~\(j^{th}\) edge of~\({\cal P}\) is~\(R_j = W(v_j,v_{j+1}).\) Therefore, the total weight of the path~\({\cal P}(J_0)\) formed by the first~\(n-J_0\) edges in~\({\cal P},\) is~\(\sum_{j=1}^{n-J_0} R_j.\)  Since the subgraph~\(G_{light}\) consisting of light and effective edges is connected (see discussion prior to the definition of~\(E_{comb}\) in~(\ref{e_comb_est_ax})), we also see that~\({\cal P}(J_0)\) can be extended to a spanning tree~\({\cal T}_{ext}\) of~\(G\) by adding~\(J_0-1\) edges of~\(G_{light}.\) Each edge of~\(G_{light}\) has weight at most~\(2\varphi_n\) (see~(\ref{light_def})) and the cost of an edge~\(h\) with weight~\(W(h)\) is at most~\(B \cdot W(h)\) for constant~\(B > 0,\) by Theorem statement. Consequently, the minimum cost~\(\tau_n,\) which is no more than the total cost of all edges in~\({\cal T}_{ext},\) satisfies
\begin{equation}\label{aane_wala}
\tau_n \leq B\sum_{j= 1}^{n-J_0}R_j + 2BJ_0\varphi_n.
\end{equation}

It remains to estimate the sum~\(\sum_{j}R_j,\) which we do by iteration as follows. Let~\(\{Y_i\}_{1 \leq  i \leq n-1}\) be i.i.d.\ random variables each having the same distribution as the edge weights and let~\(\{Z_i\}_{1 \leq i \leq n-1}\) be independent random variables where~\(Z_i\) has the same distribution as~\(\min_{1 \leq j \leq i} Y_j.\) By definition, the weight~\(R_j = W(v_j,v_{j+1})\) of the~\(j^{th}\) added edge to~\({\cal P}\) depends only on the state and weight of edges having at least one endvertex in~\(\{v_1,\ldots,v_{j}\}.\) Moreover, if the event~\(E_j(v_j)\) as defined prior to~(\ref{ej_vj_est}) occurs, then~\(v_j\) has at least
\begin{equation}\label{m_j_def}
m(j) := \frac{(n-j)p}{2}
\end{equation} neighbours in~\({\cal N}(v_j) \setminus \{v_1,\ldots,v_{j-1}\}.\) Consequently, given~\(\{v_1,\ldots,v_j\}\) and that the event~\( F_j = \bigcap_{l=1}^{j} E_l(v_l)\) defined in~(\ref{fj_def}) occurs, we see that~\(R_j\) is stochastically dominated from below by the random variable~\(Z_{m(j)}.\)

Denoting~\({\cal G}_j\)  to be the sigma-field generated by the states of all edges containing at least one endvertex in~\(\{v_1,\ldots, v_j\}\) and the weights of all edges containing at least one endvertex in~\(\{v_1,\ldots, v_{j-1}\},\)  we get from the above discussion that
\begin{equation}\label{rjx_est}
\mathbb{P}\left(R_j \geq x \mid {\cal G}_j\right) \ind(F_j) \leq \mathbb{P}\left(Z_{t(j)} \geq x\right) \ind(F_j),
\end{equation}
for any~\(x \in \mathbb{R}\) and for each~\(1 \leq j \leq n-J_0.\) Setting~\(T_j := \sum_{l=1}^{j}R_l,\) we get for any~\(x,y \in \mathbb{R}\) that
\begin{align}
\mathbb{P}\left(T_j \geq x \mid {\cal G}_j\right) \ind(F_j)\ind(T_{j-1}=y)
&\leq \mathbb{P}\left(R_j \geq x-y \mid {\cal G}_j\right) \ind(F_j) \ind(T_{j-1} = y) \nonumber\\
&\leq \mathbb{P}\left(Z_{m(j)} \geq x-y\right) \ind(F_j)\ind(T_{j-1} = y), \nonumber
\end{align}
where~\(m(j) = \frac{(n-j)p}{2}\) is as defined in~(\ref{m_j_def}). Taking expectations and summing over~\(y,\) we then obtain
\[\mathbb{P}\left(\{T_j \geq x\} \bigcap F_j\right) \leq \mathbb{P}\left(\left\{T_{j-1} + Z_{m(j)} \geq x\right\} \bigcap F_j\right)\]
leading to the recursion
\begin{equation}\label{tham_tam}
\mathbb{P}\left(T_j \geq x\right) \leq \mathbb{P}\left(T_{j-1} + Z_{m(j)} \geq x\right) + \mathbb{P}(F_j^c).
\end{equation}

Again arguing as above with~\(j-1\) replaced by~\(j,\) we have that
\[\mathbb{P}\left(T_{j-1} + Z_{m(j)} \geq x\right) \leq \mathbb{P}\left(T_{j-2} + Z_{m(j-1)} + Z_{m(j)} \geq x\right) + \mathbb{P}(F_{j-1}^c)\]
and proceeding iteratively with~(\ref{tham_tam}), we get that
\begin{equation}\label{kurma}
\mathbb{P}\left(T_j \geq x\right) \leq \mathbb{P}\left(Z_{tot}(j) \geq x\right) + \sum_{l=1}^{j}\mathbb{P}(F_l^c)
\end{equation}
for each~\(1 \leq j \leq n-J_0,\)
where \[Z_{tot}(j) := Z_{m(1)} + \ldots + Z_{m(j)}.\]

From the estimate~(\ref{ej_vj_est}), we see that
\[\mathbb{P}\left(E_{j}(v_j) \mid F_{j-1}\right) \leq \exp\left(-\frac{np}{32}\right)\] for each~\(1 \leq j \leq n-J_0 = n-1/p\) and so proceeding iteratively, we see that the event~\(F_j = \bigcap_{l=1}^{j}E_l(v_l)\) does \emph{not} occur with probability
\[\mathbb{P}(F_j^c) \leq j \cdot \exp\left(-\frac{np}{32}\right) \leq n \cdot \exp\left(-\frac{np}{32}\right).\] Plugging this into~(\ref{kurma}) and again using~\(j \leq n,\) we get that
\[\mathbb{P}\left(T_j \geq x\right) \leq \mathbb{P}\left(Z_{tot}(j) \geq x\right) +  n^2 \cdot \exp\left(-\frac{np}{32}\right)\]
for each~\(1 \leq j \leq n-J_0.\) Recalling that~\(T_j = \sum_{l=1}^{j} R_j\) we then obtain
\begin{equation}\label{tj_est_one_ax}
\mathbb{P}\left(\sum_{j=1}^{n-J_0} R_j \geq x\right) \leq \mathbb{P}\left(Z_{tot}(n-J_0) \geq x\right) +  n^2 \cdot \exp\left(-\frac{np}{32}\right)
\end{equation}

In the final ingredient of this proof, we use Azuma-Hoeffding inequality to estimate~\(Z_{tot} := Z_{tot}(n-J_0)\)  and begin with some preliminary definitions. Letting~\(\lambda > 0\) be a constant to be determined later and recalling the term~\(s_l = H\left(\frac{\lambda \log{n}}{l}\right)\) in Lemma~\ref{lemma_z}, we see that~\(s_{m(j)} = H\left(\frac{32\lambda \log{n}}{(n-j)p}\right).\) Further recalling that~\(J_0 = \frac{32\theta \log{n}}{p}\) we choose~\(\theta > \lambda\) and get that
\begin{equation}\label{nu_wt_expr}
\sum_{j=1}^{n-J_0} s^2_{m(j)} = \sum_{j=J_0}^{n-1} H^2\left(\frac{32\lambda \log{n}}{jp}\right) =: \mu_{wt},
\end{equation}
as in the Theorem statement.

Defining the event \[A_j := \left\{Z_{m(j)} \leq 2s_{m(j)}\right\}\] for~\(1 \leq j \leq n-J_0,\) we get from the corresponding estimate~(\ref{z_up}) that \[\mathbb{P}(A_j) = 1-\exp\left(-\lambda \log{n}\right) = 1-\frac{1}{n^{\lambda}}.\] Given~\(\gamma > 0,\) we choose~\(\lambda = \gamma+2\) and define the event \[A := \bigcap_{j=1}^{n-J_1} A_j,\] to get from the union bound that
\begin{equation}\label{a_est}
\mathbb{P}(A) \geq 1-\frac{1}{n^{1+\gamma}}.
\end{equation}
Letting~\(\mathbb{P}_A := \mathbb{P}(.\mid A)\) denote the distribution conditioned on the occurrence of~\(A,\) we have from the Azuma-Hoeffing inequality (Lemma~\ref{lemmax}\((b)\)) that
\[\mathbb{P}_A\left(|Z_{tot} - \mathbb{E}_A Z_{tot}| \geq  t\right) \leq 2 \exp\left(-\frac{t^2}{4\mu_{wt}}\right)\]
for~\(t \geq 0.\) Setting~\(t=  \mathbb{E}_AZ_{tot},\) we further obtain
\begin{equation}\label{azumar}
\mathbb{P}_A\left(Z_{tot} \geq 2\mathbb{E}_A Z_{tot}\right) \leq 2 \exp\left(-\frac{\left(\mathbb{E}_AZ_{tot}\right)^2}{4\mu_{wt}}\right).
\end{equation}

To estimate~\(\mathbb{E}_AZ_{tot},\) we recall the term~\(t_l = H\left(\frac{1}{l}\right)\) in Lemma~\ref{lemma_z} and get from the expectation bounds~(\ref{exp_z_bounds}) that there is a constant~\(c  > 0\) such that if~\(m(j) = \frac{(n-j)p}{32} \geq c \log{n}\) or equivalently if~\(j \leq n-\frac{32c \log{n}}{p},\) then
\begin{equation}\label{zmj_est}
c^{-1}t_{m(j)} \leq \mathbb{E}_A Z_{m(j)} \leq c t_{m(j)}.
\end{equation}
Recalling that~\(J_0 = \frac{32\theta \log{n}}{p},\) we now choose the constant~\(\theta > c\)  and get that
\[\sum_{j=1}^{n-J_0} t_{m(j)} = \sum_{j=1}^{n-J_0} H\left(\frac{32}{(n-j)p}\right) = \sum_{j=J_0}^{n-1} H\left(\frac{32}{jp}\right) =: \nu_{wt},\] as in Theorem statement. Thus
\begin{equation}\label{exp_zmj_bounds}
c^{-1} \nu_{wt} \leq \mathbb{E}_AZ_{tot} = \sum_{j=1}^{n-J_0} \mathbb{E}_A Z_j \leq c \nu_{wt}
\end{equation}
and plugging this into~(\ref{azumar}), we have
\begin{equation}\label{azumar_two}
\mathbb{P}_A\left(Z_{tot} \geq 2c\nu_{wt}\right) \leq 2 \exp\left(-\frac{D\nu_{wt}^2}{\mu_{wt}}\right)
\end{equation}
for some constant~\(D > 0.\)

Combining~(\ref{azumar_two}) with the estimate~(\ref{a_est}) for the occurrence of the event~\(A,\) we then obtain
\[\mathbb{P}\left(Z_{tot} \geq 2c\nu_{wt}\right) \leq 2\exp\left(-\frac{D\nu_{wt}^2}{\mu_{wt}}\right) + \frac{1}{n^{1+\gamma}}\]
and recalling the relation~(\ref{tj_est_one_ax}), we then obtain
\begin{equation}\nonumber
\mathbb{P}\left(\sum_{j=1}^{n-J_0} R_j \geq 2c\nu_{wt}\right) \leq 2\exp\left(-\frac{D\nu_{wt}^2}{\mu_{wt}}\right) + \frac{1}{n^{1+\gamma}} +  n^2 \cdot \exp\left(-\frac{np}{32}\right).
\end{equation}
As before, we select~\(np \geq M\log{n}\) for a large enough constant~\(M > 0\) so that the final term above is at most~\(\frac{1}{n^{1+\gamma}}.\) The estimate~(\ref{aane_wala}) then obtains the desired deviation bound for the minimum cost~\(\tau_n\) in Theorem statement.

For the expectation upper bound, we define the event \[E_{nice} := \left\{\tau_n \leq    \kappa \nu_{wt} + \frac{\kappa \varphi_n \log{n}}{p}  \right\}\] and get that
\begin{equation}\label{simma}
\mathbb{E}\tau_n \ind(E_{nice}) \leq \kappa \nu_{wt} + \frac{\kappa \varphi_n \log{n}}{p}.
\end{equation}
If~\(E_{nice}\) does not occur, then~\(\tau_n\) is upper bounded by the sum of the weights of all edges in the complete graph~\(K_n.\) Arguing as in the derivation of~(\ref{thmix_ax_2}), we have that
\begin{equation}\label{thmix_ax_3}
\mathbb{E}\tau_n \ind(E_{nice}^c)  \leq D n^2 \cdot \left(\mathbb{P}(E^c_{nice})\right)^{\frac{1}{2}}
\end{equation}
for some constant~\(D > 0.\) Combining this with~(\ref{simma}) and using the estimate~(\ref{mn_up_wt}) for~\(E^c_{nice}\) with~\(\gamma > 0\) arbitrary, we  obtain the desired expectation bound for~\(\tau_n\) in Theorem statement. This completes the proof of the Theorem.~\(\qed\)

\renewcommand{\theequation}{\thesection.\arabic{equation}}
\setcounter{equation}{0}
\section{Proof of Theorem~\ref{thm_spat_mst}}\label{sec_pf_thm_spat_mst}
We begin with the proof of the lower bound, which we establish using Lemma~\ref{lemma_mst_low}. Since~\(p(u,v) = p\) for all edges~\((u,v),\) the connectivity condition~(\ref{p_cond_new}) in the statement of Theorem~\ref{thm_max_cst_low} holds. Consequently, the estimate~(\ref{mst_low_ax}) implies that there are constants~\(\theta_1,\theta_2 > 0\) such that if~\(p \geq \frac{\theta_1 \log{n}}{n},\) then
\[\mathbb{P}\left(\tau_n \geq \theta_2 n J\left(\frac{\theta_2}{np}\right)\right) \geq 1-\theta_1 \cdot p,\]
where~\(J(.)\) is the inverse cost cdf as defined in~(\ref{h_def}). Choosing~\(\theta_1 > 0\) larger and~\(\theta_2 > 0\) smaller if necessary, we also get from the connectivity estimate~(\ref{e_con_est_max}) that if~\(p \geq \frac{\theta_1 \log{n}}{n},\) then
\[\mathbb{P}\left(E_{con}\right) \geq  1-\exp\left(-\theta_2 np\right),\]
where we recall that~\(E_{con}\) is the event that the random graph~\(G\) is connected. Combining via the union bound gives
\begin{equation}\label{semma_dhool}
\mathbb{P}\left(E_{con}\bigcap \left\{ \tau_n \geq \theta_2 n J\left(\frac{\theta_2}{np}\right) \right\}\right) \geq 1-\theta_1 \cdot p - \exp\left(-\theta_2 np\right).
\end{equation}

As a first step towards estimating~\(J(.),\) we bound the edge cost cdf. Indeed,  since the edge cost factor defined in~(\ref{cst_def_mst_spat}) and the edge weight are at most~\(1,\) so is the edge cost. We demonstrate that the cost~\(c(u, v)\) of the edge~\((u, v)\) with endvertices~\(u\) and~\(v\) satisfies
\begin{equation}\label{timmf_2}
\mathbb{P}\left(c(u,v) \leq x\right) \leq C \cdot x^{2/\alpha},
\end{equation}
for some constant~\(C > 0\) and all~\(0 < x  <1.\)

We have
\begin{align}
\mathbb{P}\left(c(u,v) \leq x\right) &= \mathbb{P}\left(\frac{d^{\alpha}(X_u, X_v)}{(\sqrt{2})^{\alpha}} \cdot W(u, v) \leq x \right) \nonumber\\
&= \sum_{k \geq 1}   \mathbb{P}\left(\frac{d^{\alpha}(X_u, X_v)}{(\sqrt{2})^{\alpha}} \cdot W(u, v) \leq x,\; \frac{1}{k+1} < W(u, v) \leq \frac{1}{k} \right) \nonumber\\
&\leq \sum_{k \geq 1}   \mathbb{P}\left(d^{\alpha}(X_u, X_v) \leq (k+1) x (\sqrt{2})^{\alpha},\; \frac{1}{k+1} < W(u, v) \leq \frac{1}{k} \right) \nonumber\\
&= \sum_{k \geq 1}   \mathbb{P}\left(d^{\alpha}(X_u, X_v) \leq (k+1) x (\sqrt{2})^{\alpha}\right) \mathbb{P}\left( \frac{1}{k+1} < W(u, v) \leq \frac{1}{k} \right) \nonumber\\
&\leq \sum_{k \geq 1}   \mathbb{P}\left(d^{\alpha}(X_u, X_v) \leq (k+1) x (\sqrt{2})^{\alpha}\right) F\left(\frac{1}{k}\right). \label{tik_tjk}
\end{align}

From the estimate~(\ref{ec_dst_est_2}) for the Euclidean distance, we know that
\begin{align}
\mathbb{P}\left(d^{\alpha}(X_u, X_v) \leq (k+1) x (\sqrt{2})^{\alpha}\right) &= \mathbb{P}\left(d(X_u, X_v) \leq ((k+1) x)^{1/\alpha} \sqrt{2}\right) \nonumber\\
&\leq C (k+1)^{2/\alpha} x^{2/\alpha}, \nonumber
\end{align} for some constant~\(C > 0\) and all~\(0 < x  <1.\) Therefore, we get from~(\ref{tik_tjk}) that
\begin{align}
\mathbb{P}\left(c(u,v) \leq x\right) &\leq   C x^{2/\alpha} \sum_{k \geq 1} (k+1)^{2/\alpha} F\left(\frac{1}{k}\right)\nonumber\\
&\leq C_1 x^{2/\alpha} \sum_{k \geq 1}k^{2/\alpha} F\left(\frac{1}{k}\right)\nonumber\\
&\leq C_2 x^{2/\alpha}, \label{timmf}
\end{align}
for some constants~\(C_1,C_2 > 0,\) where the final estimate in~(\ref{timmf}) follows from the condition~(\ref{tail_cnd_mst}). This proves~(\ref{timmf_2}).

For~\(0 < z < 1,\) we set~\(x_z := \left(\frac{z}{2(1+C_2)}\right)^{\alpha/2} \in (0,1)\) in~(\ref{timmf_2}) and get that~\[\mathbb{P}\left(c(u,v) \leq x_z\right) \leq \frac{C_2z}{2(1+C_2)} \leq \frac{z}{2}.\] This in turn that the inverse cdf~\(J(.)\) defined in~(\ref{h_def}) must satisfy~\(J(z) \geq x_z.\) Setting~\(z = \frac{\theta_2}{np},\) where~\(\theta_2 > 0\) is the constant in~(\ref{semma_dhool}), obtains the desired  lower bound for~\(\tau_n\) in~(\ref{mst_dev_bds_spat}) and~(\ref{mst_exp_bds_spat}).

For the upper bound for~\(\tau_n,\) we  use a segmentation technique similar to that described in the proof of Theorem~\ref{thm_spat_mast}\((a).\) Tile the unit square~\(S\) into small~\(s_n \times s_n\) squares~\(\{R_i\}_{1 \leq i \leq N}\) as in Figure~\ref{fig_squares}, where~\(s_n = \sqrt{\frac{\zeta_0\log{n}}{np}}\) for some constant~\(\zeta_0 > 0\) to be determined later and~\(N = \frac{1}{s_n^2}.\) Since~\(np \geq M \log{n}\) for constant~\(M >0\) by this Theorem statement, we choose~\(M  = M(\zeta_0)> 0\) large enough, follow the same argument as in~(\ref{s_n_est}) and assume that~\(N\) is an integer.

Let~\(G_i \subset G\) be the subgraph of~\(G\) induced by the vertices present in~\(R_i.\) Below, we demonstrate that with high probability, each~\(G_i\) is connected and then ``stitch" the individual spanning subtrees to create an overall spanning tree of~\(G.\) Recalling that~\(N(R_i)\) is the number of vertices present in~\(R_i\) (see discussion prior to~(\ref{n_i_est_bax}), we get from the estimate~(\ref{n_i_est_bax}) that
\begin{equation}\label{n_i_est_bax_2}
\mathbb{P}\left(E_{vert}(i)\right) \geq 1-\exp\left(-D_1m\right),
\end{equation}
where~\[m := ns_n^2,\;\;\;E_{vert}(i) = \{D_1m \leq N(R_i) \leq D_2 m\}\] and~\(D_1,D_2 > 0\) are constants not depending on the choice of~\(i\) or the constants~\( M,\zeta_0 >0\) described in the previous paragraph.

Assuming~\(E_{vert}(i)\) occurs, we now demonstrate that the subgraph~\(G_i \subset G\) induced by the vertices of~\(R_i,\) is connected with high probability. Indeed, we see that if~\(E_{vert}(i)\) occurs, then there are order of~\(m\) vertices in~\(R_i,\) any two of which are joined by an edge with probability~\(p.\) Since~\(m = ns_n^2 = \frac{\zeta_0\log{n}}{p}\) and~\(np \geq M\log{n}\) by Theorem statement, we choose~\(M > \zeta_0\) larger we have that
\begin{equation}\label{gliksi}
mp = ns_n^2p = \zeta_0\log{n} \geq \zeta_0 \log{m}
\end{equation} and choosing~\(\zeta_0 > 0\) large enough, we see that  the connectivity estimate~(\ref{e_con_est_max}) is then applicable for~\(G_i.\)

Defining~\(E_{con}(i)\) to be the event that~\(G_i\) is connected, we get from~(\ref{e_con_est_max}) that
\begin{equation}\label{gen_flight}
\mathbb{P}\left(E_{con}(i) \mid E_{vert}(i)\right) \geq 1- \exp\left(-D_3 mp\right),
\end{equation} for some constant~\(D_3 > 0,\) again not depending on the choice of~\(i,M\) or~\(\zeta_0.\) Combining this with~(\ref{n_i_est_bax_2}), we get that
\begin{align}
\mathbb{P}\left(E_{con}(i) \cap E_{vert}(i) \right) &\geq \left(1-\exp\left(-D_3 mp\right)\right) \left(1-\exp\left(-D_1 ns_n^2\right)\right) \nonumber\\
&\geq  1-\exp\left(-D_3 mp\right) - \exp\left(-D_1 ns_n^2\right)  \nonumber\\
&= 1- \exp\left(-D_3 \zeta_0\log{n}\right) - \exp\left(-\frac{D_1\zeta_0 \log{n}}{p}\right) \nonumber\\
&\geq 1- 2\exp\left(-D_4 \zeta_0\log{n}\right), \label{thismika}
\end{align}
for some constant~\(D_4 > 0,\) not depending on the choice of~\(i,M\) or~\(\zeta_0.\)

Given~\(\gamma > 0,\) we choose~\(\zeta_0 > 1\) larger if necessary and get from~(\ref{thismika}) that
\[\mathbb{P}\left(E_{con}(i) \cap E_{vert}(i)\right) \geq  1 - \frac{1}{n^{\gamma+2}},\]
for each~\(1 \leq i \leq N.\) Letting~\[F_{con} := \bigcap_{i=1}^{N} E_{con}(i) \cap E_{vert}(i),\] we get by an application of the union bound that
\begin{equation}\label{f_tt_est}
\mathbb{P}(F_{con}) \geq 1 - \frac{N}{n^{\gamma+2}} \geq 1 - \frac{1}{n^{1+\gamma}},
\end{equation}
since
\begin{equation}\label{N_est}
N = \frac{1}{s_n^2} = \frac{np}{\zeta_0 \log{n}} \leq n,
\end{equation} by our choice of~\(\zeta_0 > 1.\)

The occurrence of the event~\(F_{con}\) ensures that each~\(G_i\) contains a spanning tree~\({\cal T}_i\) formed by the vertices located in~\(R_i.\) Our final ingredient establishes the existence of ``cross" edges that facilitate ``stitching" of these individual trees~\(\{{\cal T}_i\}.\) Suppose~\(F_{tot}\) occurs and for~\(1 \leq i \leq N-1,\) let~\(E_{cross}(i)\) be the event that there is no edge of~\(G\) having one endvertex in~\(R_i\) and the other endvertex in~\(R_{i+1}.\) Given that~\(F_{con}\) occurs, both~\(R_i\) and~\(R_{i+1}\) contain at least~\(D_1m = D_1ns_n^2 = \frac{D_1\zeta_0 \log{n}}{p}\) vertices each and so
\begin{align}
\mathbb{P}(E^c_{cross}(i) \mid F_{con}) &\leq (1-p)^{(D_1m)^2} \nonumber\\
&\leq \exp\left(-p D_1^2m^2\right) \nonumber\\
&= \exp\left(-\frac{D_1^2\zeta_0^2 (\log{n})^2}{p}\right) \nonumber\\
&\leq \exp\left(-D_1^2\zeta_0^2 (\log{n})^2\right). \nonumber
\end{align}

Letting \[F_{cross} := \bigcap_{i=1}^{N-1} E_{cross}(i),\] we again apply the union bound and get that
\begin{align}
\mathbb{P}(F_{cross} \mid F_{con}) &\geq 1- N \cdot \exp\left(-D_1^2\zeta_0^2 (\log{n})^2\right) \nonumber\\
&\geq 1- n \cdot \exp\left(-D_1^2\zeta_0^2 (\log{n})^2\right)  \nonumber
\end{align}
using the estimate~(\ref{N_est}). Combining this with the estimate~(\ref{f_tt_est}) gives
\begin{align}
\mathbb{P}(F_{cross} \cap F_{con}) &\geq 1-\frac{1}{n^{1+\gamma}} -n \cdot \exp\left(-D_1^2\zeta_0^2 (\log{n})^2\right) \nonumber\\
&\geq 1-\frac{2}{n^{1+\gamma}}, \label{glimska}
\end{align}
for all~\(n\) large.

Suppose~\(F_{cross} \cap F_{con}\) occurs so that each~\(G_i, 1 \leq i \leq N = \frac{1}{s_n^2},\) is connected and contains order of~\(m  = ns_n^2\) vertices. Let~\({\cal T}_i\) be any spanning tree of~\(G_i\) and let~\(h_i, 1 \leq i \leq N-1\) be any edge of~\(G\) having one endvertex in the square~\(R_i\) and the other endvertex in~\(R_{i+1},\) whose existence is guaranteed by the event~\(F_{cross}.\) The union \[{\cal T}_{stit} := \bigcup_{i=1}^{N} {\cal T}_i \bigcup \bigcup_{i=1}^{N-1} \{h_i\}\] is a spanning tree of~\(G,\) each of whose edge has length at most~\(4s_n.\) Since the edge weights are~\(\leq 1\) by this Theorem statement, we get that the minimum cost of a spanning tree of~\(G\) is at most~\((n-1)(4s_n)^{\alpha}\) and so
\begin{equation}\label{tell_tale}
\tau_n \ind(F_{cross} \cap F_{con}) \leq (n-1) (4s_n)^{\alpha} \leq C n \cdot \left(\frac{\log{n}}{np}\right)^{\alpha/2}
\end{equation} for some constant~\(C > 0.\)

Using~(\ref{tell_tale}) and  the estimate~(\ref{glimska}) with~\(\gamma =1,\) we get that
\[\mathbb{P}\left(\tau_n \leq C n \cdot \left(\frac{\log{n}}{np}\right)^{\alpha/2}\right) \geq 1- \frac{2}{n^2} \geq 1-p,\] since~\(p\) is at least of the order of~\(\frac{\log{n}}{n},\) by this Theorem statement. This obtains the desired upper deviation bound for~\(\tau_n\) and therefore completes the proof of~(\ref{mst_dev_bds_spat}). Following a similar analysis as in the derivation of~(\ref{mn_up_wt_exp}), we also obtain the expectation  upper  bound for~\(\tau_n\) in~(\ref{mst_exp_bds_spat})  and this completes the proof of the Theorem.~\(\qed\)

\renewcommand{\theequation}{\thesection.\arabic{equation}}
\setcounter{equation}{0}
\section{Proof of Corollaries~\ref{cor_mst_repeat} and~\ref{cor_example_two_mst}}\label{sec_pf_cor_mst}

\emph{Proof of Corollary~\ref{cor_mst_repeat}}: We begin by verifying the conditions~\((I)-(II)\) in the statement of Theorem~\ref{thm_min_cst_weak}.  The connectivity condition~(\ref{p_cond_new}) in the statement of Theorem~\ref{thm_max_cst_low} is trivially true, for example, with~\(a_0 = \frac{1}{4}, b_0 = 1\) and~\(\gamma_0 = \frac{3}{4}.\)

For any~\(x \in \{0,1\},\) we see that~\(r(x,X_1) \in \{1,h_n\}\) and if~\(r(x,X_1) =h_n,\) then definitely~\(X_1 = 1.\) Thus \[\mathbb{E}r^2(x,X_1) \leq 1 + h_n^2 g_n\] is bounded since~\(\limsup h_n^2g_n < \infty,\) by Corollary statement. Since the edge weights are bounded, we see that condition~\((I)\) is satisfied. Finally, we have that the cost~\(c(u,v) = r(X_u,X_v) \cdot W(u,v)\) of any edge~\(h = (u,v)\) is at least as large as its weight~\(W(u,v),\) by the definition of the cost factor in~(\ref{cost_def_obi}) and so for any~\(x > 0\) we have that \[F^{(ct)}(x) = \mathbb{P}(c(h) \leq x) \leq  \mathbb{P}(W(h) \leq x) = F(x).\] This implies that the domination condition~(\ref{sandwich_cond}) also holds and so the conditions~\((I)-(II)\) in the statement of Theorem~\ref{thm_min_cst_weak} are satisfied.

Since the edge weights are uniform in~\([0,1],\) we see that~\(F(x) = x, 0 < x < 1\) and so~\(H(z) = z\) for~\(0 < z < 1.\) Thus
\begin{equation}\label{thilsa}
\varphi_n = H\left(\frac{\lambda \log{n}}{np}\right) = \frac{\lambda \log{n}}{np}
\end{equation}
and moreover~\(\varphi_n\) is at least of the order of~\(\frac{\log{n}}{n},\) since the edge probability~\(p = \frac{1}{n^{\beta}}, 0 <\beta < 1,\) by Corollary statement.  We therefore set~\(\gamma = 3\) in Theorem~\ref{thm_min_cst_weak} so that~\(\varphi_n^2\) is much larger than~\(\frac{1}{n^{1+\gamma}}.\) The deviation bounds in~(\ref{mn_comp_bounds}) then imply~(\ref{disco}). Moreover, from the expectation bounds in~(\ref{mst_var_bonda_ax}) we see that
\[\frac{\delta_1}{p} \leq \mathbb{E}\tau_n \leq \frac{\delta_2 \log{n}}{p} + \frac{1}{n^{1+\gamma}} \leq \frac{2\delta_2 \log{n}}{p}.\] This obtains the expectation bounds in~(\ref{trimsa}).

Finally, the variance bound in~(\ref{mst_var_bonda_ax}) and our choice of~\(\gamma = 3,\) implies that there is a constant~\(\lambda_1 > 0\) such that
\[var(\tau_n) \leq \lambda_1 n^2p \varphi_n^2 + \frac{1}{n^{1+\gamma}} \leq 2\lambda_1 n^2p \varphi_n^2.\] Recalling from~(\ref{thilsa})that~\(\varphi_n  = H\left(\frac{\lambda \log{n}}{np}\right) = \frac{\lambda \log{n}}{np},\) we get that~\( var(\tau_n) \leq D \frac{(\log{n})^2}{p} \) for some constant~\(D >0.\) Consequently, the expectation lower bound in~(\ref{trimsa}) implies that
\[\mathbb{E}\left(\frac{\tau_n}{\mathbb{E}\tau_n}-1\right)^2 \leq D_1 p \cdot (\log{n})^2  = \frac{D_1 (\log{n})^2}{n^{\beta}} \longrightarrow 0,\] for some constant~\(D_1 > 0.\) This completes the proof of the Corollary.~\(\qed\)

\emph{Proof of Corollary~\ref{cor_example_two_mst}\((a)\)}: As in the proof of Corollary~\ref{cor_mst_repeat} above, we begin by verifying the conditions~\((I)-(II)\) in the statement of Theorem~\ref{thm_min_cst_weak}. As above,   we see that the connectivity condition~(\ref{p_cond_new}) in the statement of Theorem~\ref{thm_max_cst_low} is true with~\(a_0 = \frac{1}{4}, b_0 = 1\) and~\(\gamma_0 = \frac{3}{4}.\)
Moreover, the Euclidean distance between any two points in the unit square~\(S\) is at most~\(\sqrt{2}\) and so the edge cost factor~\(r(x,y) \leq  1\) for all~\(x,y \in S.\) Thus condition~\((I)\) in Theorem~\ref{thm_min_cst_weak} is true.

We now verify that the domination relation~(\ref{sandwich_cond}) holds under slightly general conditions.  Specifically, we show that if the edge weights are at most~\(1\) and
\begin{equation}\label{mc_sims_ax}
\sum_{k \geq 1} \frac{1}{k^{2/\alpha}} F\left((k+1)x\right) \leq DF(x)
\end{equation}
for some constant~\(D > 0\) and all~\(0 < x < 1,\) then~(\ref{sandwich_cond}) holds for some constants~\(c_1,c_2 > 0.\)

Indeed, since the edge cost factor~\(r(x, y) \leq 1\) (see discussion following~(\ref{cst_def_mst_spat})), the cost~\(c(u,v)\) of the edge~\((u,v)\) with endvertices~\(u\) and~\(v,\) is at most~\(1\) and for~\(0 < x < 1\) we have that
\begin{align}
\mathbb{P}\left(c(u,v ) \leq x\right) &= \mathbb{P}\left(r(X_u, X_v) \cdot W(u, v) \leq x\right) \nonumber\\
&= \sum_{k \geq 1}  \mathbb{P}\left(r(X_u, X_v) \cdot W(u, v) \leq x, \frac{1}{k+1} < r(X_u, X_v) \leq \frac{1}{k}\right) \nonumber\\
&\leq  \sum_{k \geq 1}  \mathbb{P}\left(W(u, v) \leq (k+1)x, \frac{1}{k+1} < r(X_u, X_v) \leq \frac{1}{k}\right)\nonumber\\
&= \sum_{k \geq 1}  F((k+1)x) \mathbb{P}\left(\frac{1}{k+1} < r(X_u, X_v) \leq \frac{1}{k}\right)\nonumber\\
&\leq  \sum_{k \geq 1}  F((k+1)x) \mathbb{P}\left(r(X_u, X_v) \leq \frac{1}{k}\right). \label{cv_ax}
\end{align}
Recalling that~\(r(X_u, X_v)  = d^{\alpha}(X_u,X_v),\) we then use the estimate~(\ref{ec_dst_est_2}) to get that
\[\mathbb{P}\left(r(X_u, X_v) \leq \frac{1}{k}\right) = \mathbb{P}\left(d(X_u, X_v) \leq \frac{1}{k^{1/\alpha}}\right) \leq \frac{C}{k^{2/\alpha}}\] for some constant~\(C> 0.\) Plugging this into~(\ref{cv_ax}) and using~(\ref{mc_sims_ax}) gives that the edge cost cdf~\(F^{(ct)}(x) \leq CD F(x)\) for all~\(0 < x < 1.\) Choosing~\(C > D^{-1}\) larger if necessary, this extends to all~\(x > 0,\) since~\(F^{(ct)}(1) = F(1) = 1,\) as discussed in the beginning of the paragraph. Thus the domination relation~(\ref{sandwich_cond}) in Theorem~\ref{thm_min_cst_strong} holds.

To verify~(\ref{mc_sims_ax}), we use~\(F(y) \leq y^{1/\delta}\) for \emph{all}~\(y > 0\) and get for~\(0 < x< 1\) that
\begin{align}
\sum_{k \geq 1} \frac{1}{k^{2/\alpha}} F\left((k+1)x\right) &\leq x^{1/\delta}  \sum_{k \geq 1} \frac{(k+1)^{1/\delta}}{k^{2/\alpha}} \nonumber\\
&\leq D_1 x^{1/\delta} \sum_{k \geq 1} \frac{1}{k^{2/\alpha - 1/\delta}} \nonumber\\
&\leq D_2x^{1/\delta},\label{elesia}
\end{align}
for some constants~\(D_1,D_2 > 0,\) where the final estimate in~(\ref{elesia}) is true since~\(\alpha < \frac{2\delta}{1+\delta},\) by Corollary statement. This implies that~(\ref{mc_sims_ax}) and therefore~(\ref{sandwich_cond}) are true. Consequently the conditions~\((I)-(II)\) in Theorem~\ref{thm_min_cst_weak} are satisfied.

From the definition of the inverse cdf in~(\ref{h_def}) we get that~\[H(z) = z^{\delta},\;\;\;0 < z < 1\] and so the terms~\(\zeta_n,\varphi_n\)  in~(\ref{var_phi_ax}) evaluate to~\[\zeta_n = \frac{\lambda_1}{(np)^{\delta}}\;\;\text{ and }\;\;\varphi_n = \lambda_2 \left(\frac{\log{n}}{np}\right)^{\delta},\] for some constants~\(\lambda_1,\lambda_2 > 0.\) Since~\(np = n^{1-\beta}\) and~\(0 < \delta < 1,\) by Corollary statement, we get that~\(\frac{1}{(np)^{2\delta}}\) is much larger than~\(\frac{1}{n^{1+\gamma}}\) for any~\(\gamma > 1.\) Fixing such a~\(\gamma,\) the bounds in~(\ref{mn_comp_bounds}) and~(\ref{mst_var_bonda_ax}) then imply that there are constants~\(\gamma_1,\gamma_2 > 0\) such that
\begin{equation}\label{skilpf}
\mathbb{P}\left(E_{con} \bigcap \left\{\frac{\gamma_1 n}{(np)^{\delta}} \leq \tau_n \leq \gamma_2 n\left(\frac{\log{n}}{np}\right)^{\delta} \right\} \right) \geq 1- \gamma_2p =  1-o(1),
\end{equation}
\begin{equation}\label{skilpg}
\frac{\gamma_1 n}{(np)^{\delta}} \leq \mathbb{E}\tau_n \leq \gamma_2 n\left(\frac{\log{n}}{np}\right)^{\delta} \text{ and } var(\tau_n) \leq \gamma_2 n^2p \left(\frac{\log{n}}{np}\right)^{2\delta}.
\end{equation}

The lower bounds in~(\ref{skilpf}) and~(\ref{skilpg}) directly imply the lower bounds in~(\ref{skilp_fa}) and~(\ref{skilp_ga}) in Corollary statement. Moreover, combining the variance upper bound and the expectation lower bound in~(\ref{skilpg}) we get that
\[\mathbb{E}\left(\frac{\tau_n}{\mathbb{E}\tau_n} -1\right)^2 \leq \gamma_2 n^2p \left(\frac{\log{n}}{np}\right)^{2\delta} \cdot \frac{(np)^{2\delta}}{\gamma_1^2n^2} = \frac{\gamma_2}{\gamma_1^2} \cdot p \cdot (\log{n})^{2\delta} \rightarrow 0,\] since~\(p = \frac{1}{n^{\beta}}.\) This obtains the~\(L^2-\)convergence of~\(\tau_n,\) scaled and centred.

The upper bounds in~(\ref{skilpf}) and~(\ref{skilpg}) have an extra logarithmic factor compared to~(\ref{skilp_fa}) and~(\ref{skilp_ga}). In the rest of the proof, we remove this extra factor by appealing to Theorem~\ref{thm_min_cst_strong}. Clearly condition~\((A)\) in the statement of   Theorem~\ref{thm_min_cst_strong} is satisfied with~\(B = 1.\) Also, since~\(F(x) = x^{1/\delta}\) for~\(0 < x < 1,\) we set~\(x_0 = \frac{1}{2}\) and see that if~\( k \geq 2\) and~\(0< x < \frac{1}{2k}\) then \[F(kx) = (kx)^{1/\delta} = k^{1/\delta} \cdot x^{1/\delta} = k^{1/\delta} F(x) \geq D(\log{k})^2 F(x)\] for some constant~\(D > 0,\) not depending on the choice of~\(k.\) Thus condition~\((B)\) in Theorem~\ref{thm_min_cst_strong} holds as well.

To apply the upper bounds for~\(\tau_n\) derived in Theorem~\ref{thm_min_cst_strong}, we let~\(\gamma > 1\) be a constant as above and  let~\(\kappa = \kappa(\gamma) > 0\) be the constant appearing in the terms~\(\nu_{wt}\) and~\(\mu_{wt}\) defined in~(\ref{nu_wt_def}).  Since~\(F(x) = x^{1/\delta}\) for~\(0 < x < 1,\) we get that~\(H(z) = z^{\delta}\) for~\(0 < z < 1\) and so setting~\(J_0 :=  \frac{\kappa \log{n}}{p},\) we have that
\begin{equation}\label{nu_exp_ax}
\nu_{wt} = \sum_{j = J_0}^{n-1} \frac{1}{(jp)^{\delta}} = \frac{1}{p^{\delta}} \sum_{j=J_0}^{n-1} \frac{1}{j^{\delta}}.
\end{equation}
Comparing with integrals, we see that \[\sum_{j=J_0}^{n-1} \frac{1}{j^{\delta}} \leq \int_{J_0}^{n-1} \frac{dx}{(x-1)^{\delta}} = \frac{(n-2)^{1-\delta} -(J_0-1)^{1-\delta}}{1-\delta} \leq \frac{n^{1-\delta}}{1-\delta},\] since~\(0 < \delta < 1.\) Substituting this into~(\ref{nu_exp_ax}), we get that
\begin{equation}\label{nu_up}
\nu_{wt} \leq \frac{n^{1-\delta}}{p^{\delta}(1-\delta)} =  \frac{1}{1-\delta} \cdot \frac{n}{(np)^{\delta}}.
\end{equation}

Similarly
\[\sum_{j=J_0}^{n-1} \frac{1}{j^{\delta}} \geq \int_{J_0}^{n-1} \frac{dx}{(x+1)^{\delta}}  =  \frac{n^{1-\delta} - (J_0+1)^{1-\delta}}{1-\delta} \geq \frac{n^{1-\delta}}{2(1-\delta)},\] for all~\(n\) large, since~\(J_0 = \frac{\kappa \log{n}}{p} = \theta n^{\beta} \log{n}\) is much smaller than~\(n.\) Again plugging this into~(\ref{nu_exp_ax}) and combining with~(\ref{nu_up}), we get that
\begin{equation}\label{nu_bounds}
\frac{1}{2(1-\delta)} \cdot \frac{n}{(np)^{\delta}} \leq \nu_{wt} \leq \frac{1}{1-\delta} \cdot \frac{n}{(np)^{\delta}}.
\end{equation}

Similarly, arguing as above, we see that if~\(2\delta < 1,\) then \[\mu_{wt} = \sum_{j=J_0}^{n-1} \left(\frac{\log{n}}{jp}\right)^{2\delta} = \left(\frac{\log{n}}{p}\right)^{2\delta} \sum_{j=J_0}^{n-1} \frac{1}{j^{2\delta}} = O\left(\frac{n (\log{n})^{2\delta}}{(np)^{2\delta}}\right).\] On the other hand, if~\(2\delta \geq 1,\) then
\[\mu_{wt} \leq  \left(\frac{\log{n}}{p}\right)^{2\delta}\sum_{j=J_0}^{n-1}\frac{1}{j} = O\left(\log{n} \left(\frac{\log{n}}{p}\right)^{2\delta}\right).\] Together with the bounds for~\(\nu_{wt}\) in~(\ref{nu_bounds}), we deduce that if~\(2\delta < 1\) then~\(T := \frac{\nu_{wt}^2}{\mu_{wt}}\) is at least of the order of~\(\frac{n}{(\log{n})^{2\delta}}.\) If~\(2\delta \geq 1,\) then~\(T\) is at least  of the order of~\(\frac{n^{2-2\delta}}{(\log{n})^{2\delta}}.\)

In either of the above cases,~\(T\) is at least~\(n^{b}\) for some~\(0 < b < 1\) and so recalling that~\(\gamma > 1\) and~\(\kappa = \kappa(\gamma) > 0\) are constants, the bounds in~(\ref{mn_up_wt}) and~(\ref{mn_up_wt_exp}) imply \[ \mathbb{P}\left(E_{con} \bigcap \left\{ \tau_n \leq \kappa \nu_{wt} + \frac{\kappa \varphi_n \log{n}}{p}   \right\}\right) \geq 1- \frac{2}{n^{1+\gamma}} \]
and
\[\mathbb{E}\tau_n \leq \kappa \nu_{wt} + \frac{\kappa \varphi_n \log{n}}{p} + \frac{2}{n^{1+\gamma}},\] for all~\(n\) large, where~\(\varphi_n = \varphi_n(\kappa)\) is as defined in~(\ref{var_phi_ax}). From~(\ref{nu_bounds}), we already know that~\(\nu_{wt}\) is  of the order of~\(\frac{n}{(np)^{\delta}} \rightarrow \infty\) since~\(p = \frac{1}{n^{\beta}}, 0 < \beta < 1\) and so to complete the proof of Corollary~\ref{cor_example_two_mst}\((a),\) it suffices to demonstrate that~\(\frac{\varphi_n \log{n}}{p}\) is at most of the  order of~\(\frac{n}{(np)^{\delta}}.\)  But this is true since, by definition,~\[\varphi_n = H\left(\frac{\kappa \log{n}}{p}\right) = \left(\frac{\kappa \log{n}}{p}\right)^{\delta}\] is much less than~\(\frac{n}{(np)^{\delta}} = \frac{n^{1-\delta}}{p^{\delta}},\) for all~\(n\) large.

This obtains the desired upper bounds for~\(\tau_n\) and therefore completes the proof of Corollary~\ref{cor_example_two_mst}\((a).\)~\(\qed\)

\emph{Proof of Corollary~\ref{cor_example_two_mst}}\((b)\): Since the edge weights are~\(\leq 1\) a.s., it suffices to verify that the edge weight cdf~\(F\) satisfies~(\ref{tail_cnd_mst}). Indeed, since~\(\alpha > \frac{2\delta}{1-\delta},\) we have that~\(\frac{1}{\delta} - \frac{2}{\alpha} > 1\) strictly and so
\[\sum_{k \geq 1} k^{2/\alpha} F\left(\frac{1}{k}\right) = \sum_{k \geq 1} k^{2/\alpha} \cdot \frac{1}{k^{1/\delta}} < \infty\] and so~(\ref{tail_cnd_mst}) is true. Consequently, Theorem~\ref{thm_spat_mst} obtains the desired deviation and expectation bounds for~\(\tau_n\) in the  Corollary statement.

For the~\(L^2-\)convergence, we use the variance bound derived in~(\ref{skilpg}). This is applicable since the edge weight and the edge cost factor are both bounded. Moreover, from above we have that~\(\mathbb{E}\tau_n \geq D n \cdot (np)^{\alpha/2}\) for some constant~\(D > 0\) and so, we get from~(\ref{skilpg}) that \[\mathbb{E}\left(\frac{\tau_n}{\mathbb{E}\tau_n} -1\right)^2 \leq D_1 p (\log{n})^{2\delta} \cdot (np)^{\alpha-2\delta},\]
for some constant~\(D_1 > 0.\) Since~\(p = \frac{1}{n^{\beta}},\) the  term~\(p (\log{n})^{2\delta} \cdot (np)^{\alpha-2\delta}\) is~\(o(1)\) if~\(\beta > \frac{\alpha-2\delta}{1+\alpha-2\delta}.\) This obtains the~\(L^2-\)convergence of~\(\tau_n,\) scaled and centred and therefore completes the proof of the Corollary.~\(\qed\)

\subsection*{\em Data Availability Statement}
Data sharing not applicable to this article as no datasets were generated or analysed during the current study.

\subsection*{\em Acknowledgement}
I thank Professors Rahul Roy, Federico Camia and C. R. Subramanian for crucial comments that led to an improvement of the paper. I also thank IMSc, IISER Bhopal and University of Bristol for my fellowships.

\subsection*{\em Conflict of Interest and Funding Statement}
I certify that there is no actual or potential conflict of interest in relation to this article. No funding or assistance was received in preparation of this manuscript.

\bibliographystyle{plain}

\end{document}